\documentclass[a4paper, 11pt]{article}
\usepackage{amsfonts}
\usepackage {amssymb}
\usepackage {amsmath}
\usepackage {amsmath}
\usepackage {amsthm}
\usepackage{graphicx}
\usepackage {amscd}
\usepackage[colorlinks, linkcolor=blue, anchorcolor=black, citecolor=red]{hyperref}

\usepackage{geometry}  

\newcommand{\sddb}{{\sqrt{-1}\partial\bar{\partial}}}

\newcommand{\cP}{\mathbb P}
\newcommand{\vol}{{\rm vol}}
\newcommand{\ord}{{\rm ord}}
\newcommand{\lct}{{\rm lct}}
\newcommand{\vphi}{\varphi}
\newcommand{\diam}{{\rm diam}}
\newcommand{\cR}{{\mathcal{R}}}
\newcommand{\cS}{{\mathcal{S}}}
\newcommand{\bC}{{\mathbb{C}}}
\newcommand{\cC}{{\mathcal{C}}}
\newcommand{\bep}{{\bar{\epsilon}}}
\newcommand{\tr}{{\rm tr}}
\newcommand{\FS}{{\rm FS}}
\newcommand{\bP}{{\mathbb{P}}}
\newcommand{\hvphi}{{\hat{\varphi}}}
\newcommand{\cM}{{\mathcal{M}}}
\newcommand{\cD}{{\mathcal{D}}}
\newcommand{\cE}{{\mathcal{E}}}
\newcommand{\NA}{{\rm NA}}

\newcommand{\cJ}{{\mathcal{J}}}

\newcommand{\cL}{{\mathcal{L}}}
\newcommand{\triv}{{\rm triv}}
\newcommand{\bQ}{{\mathbb{Q}}}
\newcommand{\reg}{{\rm reg}}

\newcommand{\id}{{\rm id}}
\newcommand{\cB}{{\mathcal{B}}}
\newcommand{\bR}{{\mathbb{R}}}
\newcommand{\CM}{{\rm CM}}
\newcommand{\hdelta}{\hat{\delta}}
\newcommand{\osc}{{\rm osc}}
\newcommand{\cH}{{\mathcal{H}}}
\newcommand{\cO}{{\mathcal{O}}}
\newcommand{\cF}{{\mathcal{F}}}
\newcommand{\bZ}{{\mathbb{Z}}}

\newcommand{\bD}{{\mathbb{D}}}
\newcommand{\Vol}{{\rm Vol}}
\newcommand{\red}{{\rm red}}
\newcommand{\sing}{{\rm sing}}
\newcommand{\supp}{{\rm Supp}}

\newtheorem{thm}{Theorem}[section]
\newtheorem{prop}[thm]{Proposition}
\newtheorem{defn}[thm]{Definition}

\newtheorem{cor}[thm]{Corollary}
\newtheorem{rem}[thm]{Remark}

\newtheorem{lem}[thm]{Lemma}

\begin{document}

\title{On the Yau-Tian-Donaldson conjecture for singular Fano varieties}
\author{Chi Li, Gang Tian, Feng Wang}
\date{}

\maketitle

\abstract{
We prove the Yau-Tian-Donaldson conjecture for any $\bQ$-Fano variety that has a log smooth resolution of singularities such that a negative linear combination of exceptional divisors is relatively ample and the discrepancies of all exceptional divisors are non-positive. In other words, if such a Fano variety is K-polystable, then it admits a K\"{a}hler-Einstein metric. This extends the previous result for smooth Fano varieties to this class of singular $\bQ$-Fano varieties, which includes all $\bQ$-factorial $\bQ$-Fano varieties that admit crepant log resolutions.
}

\tableofcontents

\section{Introduction}

It's an important problem in K\"{a}hler geometry to construct K\"{a}hler-Einstein metrics on algebraic varieties.  In general there are obstructions to the existence of such canonical metrics. The Yau-Tian-Donaldson conjecture states that a smooth Fano variety admits a K\"{a}hler-Einstein metric if and only if it is K-polystable. This conjecture has been solved (see \cite{Tia15} and also \cite{CDS15}).

Expecting applications in algebraic geometry, it is natural to consider singular varieties. For canonically polarized varieties and Calabi-Yau varieties with klt (for {\it Kawamata log terminal}) singularities, existence of singular K\"{a}hler-Einstein metrics have been studied (see \cite{EGZ09}). 
In these cases, when set-up appropriately, there are essentially no obstruction to the existence as in the smooth case.
In the Fano case, there is a class of singular $\bQ$-Fano varieties, the so-called $\bQ$-smoothable $\bQ$-Fano varieties, for which the Yau-Tian-Donaldson conjecture has been proved (see \cite{SSY16, LWX14}). Such varieties are Gromov-Hausdorff limits of smooth K\"{a}hler-Einstein manifolds and appear on the compactification of moduli space of smooth K\"{a}hler-Einstein (or, equivalently, K-polystable) varieties (see \cite{LWX14}). However, such singular varieties are rather restricted, e.g., not all singular $\bQ$-Fano varieties are smoothable as mentioned above. 

In this paper, we will prove the Yau-Tian-Donaldson conjecture for a much large class of singular $\bQ$-Fano varieties. By a $\bQ$-Fano variety, we always mean a projective variety, denoted by $X$, such that $-K_X$ is an ample $\bQ$-Cartier divisor and $X$ has at worst klt singularities. Notice that the klt assumption is a necessary condition for the existence of singular K\"{a}hler-Einstein metrics (see \cite{BBEGZ, DS14}). By Hironaka's theorem, there always exists a (non-unique) log resolution of singularities $\mu: Y\rightarrow X$ such that $\mu$ is isomorphic over $X^\reg$ and $\mu^{-1}(X^\sing)$ is supported on a simple normal crossing divisor $E=\sum_i E_i$. We can write
$$K_Y=\mu^*K_X+\sum_i a_i E_i$$ where $a_i=a(E_i, X)$ is called the discrepancy of $E_i$ over $X$ in birational geometry. Recall that by definition, $X$ has klt singularities if for any (equivalently for all) log resolution of singularities we have $a_i>-1$ for all $i$.

The class of singular $\bQ$-Fano varieties that we will consider is defined as follows.
\begin{defn}\label{defn-admissible}
We say that a quasi-projective variety $X$ has {\it admissible singularities} if $X$ is $\bQ$-Gorenstein and
there exists a log resolution $\mu: Y\rightarrow X$ with a simple normal crossing exceptional divisor $E=\sum_i E_i$ such that
\begin{enumerate}
\item[(1)] if we write $K_Y=\mu^*K_X+\sum_i a_i E_i$, then $-1<a_i\le 0$;

\item[(2)] there exist $\theta_i\in \bQ_{>0}$ such that $-\sum_i \theta_i E_i$ is $\mu$-ample.
\end{enumerate}
Such a log resolution will be called an {\it admissible} log resolution.

If such an admissible resolution also satisfies $a_i=0$ for any $i$, then we say it is a {\it crepant admissible} resolution.

\end{defn}

The following are some examples of admissible singularities:
\begin{enumerate}
\item[(i)] smooth points.
\item[(ii)] $2$-dimensional klt singularities, which are the same as $2$-dimensional isolated quotient singularities.
\item[(iii)] If $S$ is a smooth Fano manifold, then the affine cone $C(S, K_S^{-k})={\rm Spec}_{\bC}\left(\bigoplus_{l=0}^{+\infty}H^0(S, K_S^{-kl})\right)$ has admissible singularity if and only if $k\ge 1$.
\item[(iv)] $\bQ$-factorial singularities that admit crepant resolutions.
\item[(v)] Product of above examples of admissible singularities.
\end{enumerate}

\begin{rem}\label{rem-admissible}
Log resolutions of singularities are not unique. If $X$ is $\bQ$-factorial, then the condition (2) in Definition \ref{defn-admissible} is always satisfied for any log resolution of singularities.

On the other hand, there are now canonical ways to resolve the singularities (see \cite{BM97}). Such canonical resolutions of singularities are obtained by a finite sequence of blow-ups with smooth centers $\sigma_j: X_{j+1}\rightarrow X_j, X_0=X$, with the property that for any local embedding $X|_U\hookrightarrow \bC^N$ this sequence of blow-ups is induced by the embedded desingularization of $X|_U$. For these canonical resolutions of singularities the condition (2) in Definition \ref{defn-admissible} is indeed satisfied (see e.g. \cite[Lemma 2.2]{CMM17}).

\end{rem}

Now we can state our main result.

\begin{thm}\label{thm-main}
An admissible $\bQ$-Fano variety admits a K\"{a}hler-Einstein metric if it is K-polystable.
\end{thm}

Note that the reverse direction was proved by Berman (\cite{Berm15}). We refer to Section \ref{sec-Kstab} for more details about K-polystability. Next we sketch the proof of the above theorem.

Under the assumption of K-polystability, we want to construct a family of conical K\"{a}hler-Einstein metrics on an admissible resolution and prove that this family converges in the Gromov-Hausdorff topology to a K\"{a}hler-Einstein metric on the admissible $\bQ$-Fano variety. The family of conical K\"{a}hler-Einstein metrics is obtained by solving some Monge-Amp\`{e}re equation set up in the following manner.

Assume $\mu: M\rightarrow X$ is an admissible resolution satisfying:
\begin{equation}\label{eq-KMKX}
K_M=\mu^*K_X+\sum_i a_i E_i, \text{ with } a_i\in (-1,0].
\end{equation}
By condition (2) in the definition \ref{defn-admissible}, we can choose $\theta_i\in \bQ_{>0}$ such that $L_{1}=\mu^*K_X^{-1}-\sum_i \theta_iE_i$ is an ample $\bQ$-Cartier divisor.
For any $\epsilon \in [0, 1]$, define
$$L_\epsilon\,\overset{\triangle}{=}\,\mu^*K_X^{-1}-\epsilon\sum_i\theta_i E_i.$$
By rescaling $\theta_i$, we can assume that $\theta_i\le (1+a_i)/2$ (for a technical reason see \eqref{eq-choosethetam}).
Moreover we can choose $m$ sufficiently large $(m\ge 2)$ and fix a sufficiently general smooth divisor $H\in |m L_1|$ such that $H+\sum_i E_i$ is simple normal crossing. For any $t\in (0,1)$, we want to solve the conical K\"{a}hler-Einstein equation for $\omega_\epsilon\in 2\pi c_1(L_\epsilon)$:

\begin{equation}
Ric(\omega_\epsilon)=t\omega_\epsilon+\frac{1-t}{m}\{H\}+\sum_i (-a_i+t \epsilon\theta_i+(1-t)\theta_i)\{E_i\}  \tag{$*_{(\epsilon,t)}$}.
\end{equation}
Note that this equation corresponds to the following decomposition of $-K_M$ in \eqref{eq-KMKX}:
\begin{eqnarray*}
-K_M&=&t (\mu^*(-K_X)-\epsilon \sum_i \theta_i E_i)+(1-t)\sum_i (\mu^*(-K_X)-\sum_i \theta_i E_i)\\
&&\hskip 2cm +\sum_i (-a_i+t\epsilon\theta_i+(1-t)\theta_i) E_i.
\end{eqnarray*}

For the simplicity of notations, we will denote the divisor
$$B_{(\epsilon,t)}=\frac{1-t}{m} H+(-a_i+t\epsilon\theta_i+(1-t)\theta_i )E_i.$$
Because $-a_i\in [0,1)$, when $\epsilon\ll 1$ and $t$ sufficiently close to $1$, $B_{(\epsilon,t)}$ is an effective divisor with simple normal crossing support and the coefficients of $B_{(\epsilon,t)}$ are strictly less than $1$. In particular $(M, B_{(\epsilon,t)})$ has log terminal singularities. We will carry out the proof in several steps:

{\bf Step 1:}
We prove that for any $t\in (0,1)$, there exists $\epsilon^*(t)>0$ such that for any $\epsilon\in (0,\epsilon^*(t))$, the pair $(M, B_{(\epsilon,t)})$ is uniformly log-K-stable with a positive slope constant $\delta=\delta(t)>0$ (independent of $\epsilon$). This is achieved by using the valuative criterions
for K-stability developed in \cite{Fuj16, Fuj17a, Li15,LX16}.

{\bf Step 2:}
By a recent work of Tian-Wang (\cite{TW19}), we know that there exists a strong conical K\"{a}hler-Einstein metric, which will be denoted by $\omega_{(\epsilon,t)}$ on $(M, B_{(\epsilon,t)})$. In particular, $\omega_{(\epsilon,t)}$ is smooth on the open set $U=M\setminus (H\cup \cup_i E_i)$ and has conical singularities with appropriate cone angles along the components of the simple normal crossing divisor $B_{(\epsilon,t)}$. Moreover $U$ is geodesically convex.

Moreover by adapting Berman-Boucksom-Jonsson's method to the log setting \emph{with a smooth ambient space}, we will prove that the log-Ding-energy of $(M, B_{(\epsilon,t)})$ is proper such that the leading coefficient of properness is uniform with respect to $\epsilon>0$ for a fixed $t\in (0,1)$. As a consequence of this properness, the existence of the strong conical K\"{a}hler-Einstein metric on $(M, B_{(\epsilon,t)})$ can also be obtained using the work in \cite{GP16} (see also \cite{JMR16}).

{\bf Step 3:}
Using a scaling trick,
we prove that the log-Ding-energy of $(M, B_{(\epsilon,t)})$ is uniformly proper for sufficiently small $\epsilon$ in the following sense. Choose a smooth reference metric $e^{-\psi_\epsilon}$ on
$L_\epsilon$
which is an interpolation of the Hermitian metrics $e^{-\psi_0}$ on $L_0$ and $e^{-\psi_1}$ on $L_1$. There exist $\delta=\delta(t)>0, C_1=C_1(t),C_2=C_2(t)>0$ such that for any $\epsilon \in (0, \epsilon^*(t))$ and $\vphi\in PSH(L_\epsilon)\bigcap \left( \psi_\epsilon+L^\infty(M)\right)$, we have:
\begin{equation}
\cD_{B_{(\epsilon,t)}}(\vphi)\ge \delta J_{\psi_\epsilon} (\vphi)- C_1 \epsilon \|\vphi-\psi_\epsilon\|_\infty-C_2.
\end{equation}
This, combined with the uniform Sobolev constant estimate of $\omega_{(\epsilon,t)}$, implies that $\vphi_{(\epsilon,t)}-\psi_{\epsilon}$ has a uniform $L^\infty$-bound. As a consequence, as $\epsilon\rightarrow 0$, we obtain a weak conical K\"{a}hler-Einstein metric $\omega_{(0,t)}$ on $(X, \frac{1-t}{m}H_X)$ for any fixed $t\in (0,1)$ where $H_X$ is the divisor on $X$ satisfying $\mu^*H_X=H+m\sum_i \theta_i E_i$ (see \eqref{eq-HX}).

{\bf Step 4:}
We prove that the metric completion of $\left(X^{\reg}, \omega_{(\epsilon,t)}|_{X^\reg}\right)$ is homeomorphic to $X$. Moreover, as $\epsilon \rightarrow 0$, $(M, \omega_{(\epsilon,t)})$ converges to a metric space $(X, d_{(0,t)})$ in the Gromov-Hausdorff toplogy.
This is in turn proved in the following steps.

\begin{enumerate}
\item
One can develop an extension of Cheeger-Colding and Cheeger-Colding-Tian's theory for conical K\"{a}hler-Einstein metrics to get a limit space $(X_{(0,t)}, d_{(0,t)})$. As a particular result, we have the following regularity results for the Gromov-Hausdorff limit of conical K\"{a}hler-Einstein metrics.

\begin{thm}[\cite{TW18}]
Let $(Y, d_Y)$ be a Gromov-Hausdorff limit of a sequence $(M, \omega_i)$ of conical K\"{a}hler-Einstein metrics as above. Then we have a decomposition $Y=\cR \cup \cS$. $\cR$ is open in $Y$ and has a smooth manifold structure equipped with a smooth K\"{a}hler-Einstein metric. The singular set has a decomposition $\cS=\cup_{k=1}^{n} \cS_{2n-2k}$ where $\cS_{2n-2k}$ consists of the points whose metric tangent cones do not split $\bR^{2n-2k+1}$ factor. $\cS_{2n-2k}$ satisfies ${\rm codim}_{\bR}(\cS_{2n-2k})\ge 2k$.

\end{thm}

We remark that the proof of the above result is different with that in \cite{Tia15}.  We don't approximate conical K\"{a}hler-Einstein metrics by smooth K\"{a}hler-Einstein metrics with uniform Ricci lower bound since in our setting the latter is not true for some reason involving cohomological classes. Instead, as already pointed out in \cite{Tia15}, one can develop similar argument to those used by Cheeger-Colding and Cheeger-Colding-Tian, in our conical setting.
\item
Based on uniform $L^\infty$-estimate obtained in {\bf Step 3},
we apply the gauge fixing technique used in \cite{RZ11, Son14, NTZ15} and prove that $(X_{(0,t)}, d_{(0,t)})$ is the metric completion of the strong conical K\"{a}hler-metric structure $(M\setminus E, (1-t)m^{-1}H, \omega_{(0,t)})$. Moreover, the identity maps
$${\rm id}: (M\setminus E, \omega_{(\epsilon,t)}) \rightarrow (M\setminus E, \omega_{(0,t)})$$
are Gromov-Hausdorff approximations, furthermore, as $\epsilon\rightarrow 0$, they converge to a surjective map: $\overline{\id}: X_{(0,t)}\rightarrow X$.

\item We show that $\overline{\id}$ is injective and hence $\overline{\id}$ is a homeomorphism.  This is essentially achieved by proving a version of Tian's partial $C^0$-estimate conjecture in our setting.

This part of the proof is also motivated by the arguments in \cite{Tia15, Son14, NTZ15}. A new input here is the gradient estimate of the potential function on the regular part, which does not seem to follow from the methods in \cite{Son14, NTZ15}.

\end{enumerate}

{\bf Step 5}
By adapting the arguments used in \cite{Li12, Tia12, DS14, Tia15, CDS15}, one can show that as $t_i\rightarrow 1$, $(X, \omega_{(0,t_i)})$ converges to a normal $\bQ$-Fano variety $(X_{\infty}, d_{\infty})$ with a K\"{a}hler-Einstein metric. Moreover, by adapting the method in \cite{Tia15} (see also \cite{CDS15}), we can show that the automorphism group of $X_\infty$ is reductive. By Luna's slice theorem from algebraic geometry, there exists a special test configuration that degenerates $X$ to $X_\infty$ and has zero Futaki invariant. By the K-polystability of $X$, we conclude that $X\cong X_\infty$.

As a consequence of the above arguments, we also get the following result for the weak conical K\"{a}hler-Einstein metric on admissible $\bQ$-Fano variety.
\begin{thm}
Assume that an admissible $\bQ$-Fano variety admits a weak conical K\"{a}hler-Einstein metric $\omega_{KE}$. Then $X$ is homeomorphic to the metric completion of the $(X^{\reg}, \omega_{KE}|_{X^\reg})$.
Moreover, $\omega_{KE}$ has continuous local K\"{a}hler potentials.
\end{thm}
Similar results in the canonically polarized and Calabi-Yau cases were proved by Jian Song in \cite{Son14}.

The main results in this paper were announced by the second named author in June 2017 at a workshop in Orsay, France.
After the submission of this paper, there have been several new results on the Yau-Tian-Donaldson conjecture for Fano varieties, which are related to this paper. In a revised version of \cite{BBJ15}, Berman-Boucksom-Jonsson have expanded their approach to prove the uniform version of Yau-Tian-Donaldson conjecture for general twisted K\"{a}hler-Einstein metrics on smooth Fano manifolds with discrete automorphism groups, including the log Fano case considered in the appendix of this paper. Based on the perturbation process in this paper and some non-Archimedean estimates, in \cite{LTW19} we proved the uniform version of Yau-Tian-Donaldson conjecture for all Fano varieties with discrete automorphism groups. In \cite{Li19} the first named author of this paper has further extended the results to work for all Fano varieties by dropping the discreteness assumption on automorphism groups (see also \cite{His19}).

We end this introduction by sketching the organization of this paper. In the next section, we collect various analytic tools including Sobolev inequality for conical metrics and some pluripotential theory that will be used later. We also recall the definition of K-stability and uniform K-stability, and state a log version of Berman-Boucksom-Jonsson's (log-BBJ) properness result whose proof will be sketched in the appendix by modifying Berman-Boucksom-Jonsson's argument.  In Section \ref{sec-alg}, we will prove {\bf Step 1} which is purely algebraic. Combined with the log-BBJ, we also achieve {\bf Step 2}. In Section \ref{sec-MA} we use a scaling trick to carefully analyze the properness properties of log-Ding-energy with respect to different parameters of $(\epsilon, t)$. As a consequence we carry out {\bf Step 3} and prove (in Theorem \ref{thm-mainweakKE}) the existence of weak conical K\"{a}hler-Einstein metrics on $\left(X, \frac{1-t}{m}H_X\right)$ (see \eqref{eq-HX} for $H_X$) for all $t\in (0,1)$ (resp. $t\in (0,1]$) if $X$ is K-semistable (resp. uniformly K-stable). In Section \ref{sec-conicX}, we carry out {\bf Step 4}. In particular, we use the techniques from \cite{NTZ15,Son14} to prove a partial $C^0$-estimate based on the uniform $L^\infty$-estimate we derived in {\bf Step 3} when $t\in (0,1)$. Finally in Section \ref{sec-STC}, we resort to the techniques developed in \cite{Tia15, CDS15} in order to carry out {\bf Step 5} and complete the proof of our main result.

\bigskip

\noindent
{\bf Acknowledgement:}
C. Li is partially supported by  NSF (Grant No. DMS-1405936 and DMS-1810867) and an Alfred P. Sloan research fellowship. G. Tian is partially supported by
NSF (Grant No. DMS-1309359,1607091) and NSFC (Grant No. 11331001). F. Wang is partially supported by NSFC (Grant No.11501501).
The authors would like to thank Chenyang Xu for correcting our definition of admissible singularities and pointing out Remark \ref{rem-admissible}.

\section{Preliminaries}

\subsection{Bochner formula}

Assume $(L, h)$ is a line bundle over a smooth manifold $M$ with a Hermitian metric $h$. Let $\Theta=\Theta(h)$ denote the Chern curvature of $h$: $\Theta=-\sddb\log h$. Assume $\omega$ is a K\"{a}hler
metric on $M$. Let $Ric(\omega)=\sqrt{-1} \sum_{i,j} R_{i\bar{j}}dz^i\wedge d\bar{z}^j$ denote its Ricci form. The following lemmas are well-known and can be obtained by direct calculations (see e.g. \cite{NTZ15})

\begin{lem}[Weitzenb\"{o}ch formulas]
Let $\xi\in \Gamma(M, T^{*(0,1)}M\otimes L)$. Then we have the following Wentzenb\"{o}ch formula:
\begin{equation}\label{eq-Wentz}
(\bar{\partial}^*\bar{\partial}+\bar{\partial}\bar{\partial}^*)\xi=\bar{\nabla}^*\bar{\nabla}\xi+(\Theta+Ric(\omega))(\xi, \cdot).
\end{equation}
\end{lem}

\begin{lem}[Bochner formulas]\label{lem-Bochner}
For any $\zeta\in H^0(M,L)$, we have the following formula:
\begin{equation}\label{eq-lapnorm}
\Delta |\zeta|^2=|\nabla \zeta|^2- |\zeta|^2\cdot {\rm tr}_{\omega}\Theta,
\end{equation}
where $|\zeta|^2=|\zeta|_h^2$ and $|\nabla \zeta|^2=|\nabla^h \zeta|_{h\otimes \omega}^2$. Moreover, we have the following Bochner formula
\begin{eqnarray}\label{eq-lapgrad}
\Delta_\omega |\nabla \zeta|^2&=&g^{k\bar{l}}\left( g^{i\bar{j}}\zeta_{,i}\overline{\zeta_{,j}}\right)_{k\bar{l}}\nonumber\\
&=& g^{k\bar{l}} g^{i\bar{j}}\left(\zeta_{,ik\bar{l}}\overline{\zeta_{,j}}+\zeta_{,i}\overline{\zeta_{,j\bar{k}l}}+\zeta_{,ik}\overline{\zeta_{,jl}}+\zeta_{,i\bar{l}}\overline{\zeta_{,j\bar{k}}}\right)\nonumber\\
&=&|\nabla \nabla \zeta|^2+|\bar{\nabla}\nabla \zeta|^2-\left[({\rm tr}_{\omega}\Theta)_i \zeta\overline{\zeta_{,i}}+c.c.\right]\nonumber\\
&&\hskip 5mm +R_{i\bar{j}}\zeta_{,i}\overline{\zeta_{,j}}-2\Theta_{i\bar{j}}\zeta_{,i}\overline{\zeta_{,j}}-|\nabla \zeta|^2 {\rm tr}_{\omega}(\Theta).
\end{eqnarray}

\end{lem}

\subsection{Sobolev constants}
By a strong conical K\"{a}hler metric on a log smooth pair $(M, B)$, we mean a K\"{a}hler current $\omega$ with bounded local potentials that also satisfies the following conditions:
\begin{enumerate}
\item[(1)]
$\omega$ is smooth on $M\setminus B$;
\item[(2)]
If in a holomorphic coordinate neighborhood $U_p$ of $p\in M$, $B=\sum_i (1-\beta_i)\{z_i=0\}$ with simple normal crossing support and $0<\beta_i\le 1$, then
$\left.\omega\right|_{U_p}$ is quasi-isometric to the following model conical metric:
\[
\sum_i \frac{\sqrt{-1} dz_i\wedge d\bar{z}_i}{|z_i|^{2(1-\beta_i)}}.
\]
Moreover $\omega$ is H\"{o}lder continuous (in the sense of \cite{Don12a,GP16,JMR16,Tia17}).
\item[(3)]
$M\setminus B$ is geodesically convex with respect to the metric structure induced by $\omega|_{M\setminus B}$.

\end{enumerate}
Note that the condition (3) follows from the previous two conditions. But we write it explicitly to emphasize.
We will need to following estimates of Sobolev constants.
\begin{prop}[{see \cite[Proposition 3]{GP16}}, \cite{MR12}]\label{prop-Sob}
Assume $(M, \omega)$ is a strong conical K\"{a}hler metric on a log smooth pair $(M, B)$ satisfying $Ric(\omega)\ge t \omega$ with $t>0$. Then the Sobolev inequality holds: there exists $C=C(M,t, \vol(\omega), \dim M)>0$ such that for any $f\in W^{1,2}(M, \bR)$, the following inequality holds:
\begin{equation}
\left(\int_M |f|^{\frac{2n}{n-1}}\omega^n\right)^{\frac{n-1}{n}}\le C \left(\int_M |\nabla f|_\omega^2 \omega^n+\int_M f^2\omega^n\right).
\end{equation}
\end{prop}
\begin{proof}
(see \cite[Remark 6.5]{JMR16} and \cite{MR12}) Because the set $M\setminus B$ is geodesically convex, we can apply the same proof in standard Riemannian geometry to get the diameter bound (as in Myers' theorem) and positive volume lower bound (as in Bishop-Gromov volume comparison). By using a cut-off function, one can verify the proof of C.Croke \cite{Cro80} and P.Li \cite{Li80} still applies. 

Alternatively, one can use the general result by Haj\l asz-Koskela. This states that if $M$ is a metric measure space which has a volume doubling
constant and satisfies a weak local Poincar\'{e} inequality, then there is a global Sobolev inequality with the constant depending only on the volume,
diameter bound, Poincar\'{e} constant and doubling constant. The Poincar\'{e} constant follows from a standard integration by parts argument and the doubling constant can be obtained by using the quasi-isometry of the strong conical K\"{a}hler metric with the model conical K\"{a}hler metric.
\end{proof}

\subsection{Energy functions}

Let $(M,B)$ be a fixed klt log-Fano pair. Fix $t\in \bR_{>0}$. We assume that the $\bR$-line bundle $L:=-t^{-1}(K_M+B)$ is semi-ample and big, in the sense that $2\pi c_1(L)\in H^{1,1}(M, \bR)$ contains a smooth  real closed positive (1,1)-form $\chi=\sddb\psi$ of positive volume. We consider $e^{-\psi}$ as a smooth Hermitian metric on the $\bR$-line bundle $L$. We consider the following spaces:
\begin{align}
&PSH(\chi)=\left\{\text{ u.s.c. function $u$ on } M; \quad \chi_u:=\chi+\sddb u\ge 0 \right\};\\
&\cH(\chi)=PSH(\chi)\cap C^\infty(M);\\
&PSH_\infty(\chi)=PSH(\chi)\cap L^\infty(M);\\
&PSH(L):=PSH([\chi])=\left\{\vphi=\psi+u; u\in PSH(\chi)\right\};\\
&PSH_\infty(L):=PSH_\infty([\chi])=\left\{\vphi=\psi+u; u\in PSH_\infty(\chi)\right\}.
\end{align}
Then $PSH([\chi])$ (resp. $PSH_\infty([\chi])$ is equal to the space of positively curved (resp. bounded positively curved) Hermitian metrics $\{e^{-\vphi}\}$ on the $\bR$-line bundle $L$.
Denote $V=c_1(L)^n=\left([\chi/(2\pi)]^n\right)>0$.
For any $\vphi\in PSH([\chi])$ such that $\vphi-\psi\in C^\infty(M)$, we have the following well-studied functionals:
\begin{eqnarray}
E_{\psi}(\vphi)&=&\frac{1}{n+1} \frac{1}{V}\sum_{i=0}^n \int_M (\vphi-\psi) (\sddb\psi)^{n-i}\wedge (\sddb\vphi)^{i},\label{eq-Ephi}\\
J(\vphi)&=&J_{\psi}(\vphi):=-E_{\psi}(\vphi)+\frac{1}{V}\int_M (\vphi-\psi)(\sddb\psi)^n, \label{eq-Jphi}\\
I(\vphi)&=&I_{\psi}(\vphi):=\frac{1}{V} \int_M (\vphi-\psi)\left((\sddb\psi)^n-(\sddb\vphi)^n\right), \label{eq-Iphi}\\
\nonumber\\
L_B(\vphi)&=&L_B(t, \vphi)=-\log\left(\frac{1}{V}\int_M \frac{e^{-t\vphi}}{|s_B|^2}\right),\label{eq-Lphi}\\
\cD_B(\vphi)&=&\cD_{B,\psi}(t, \vphi)=-E_{\psi}(\vphi)+\frac{1}{t} L_B(t, \vphi),\label{eq-Dphi} \\
\nonumber\\
H_B(\vphi)&=&H_{B,\psi}(t, \vphi)=\frac{1}{V}
\int_M \log \frac{|s_B|^2 (\sddb\vphi)^n}{e^{-t\psi}}(\sddb\vphi)^n, \label{eq-Hphi}\\
\cM_{B}(\vphi)&=&\cM_{B,\psi}(t, \vphi)=\frac{1}{t}H_{B,\psi}(t,\vphi)- (I_\psi-J_\psi)(\vphi). \label{eq-Mphi}
\end{eqnarray}
Note that in \eqref{eq-Hphi}, we think of $e^{-t\psi}|s_B|^{-2}$ as a singular Hermitian metric on $-K_M$ defining a measure on $M$ with an $L^p (p>1)$ density function.

The above functionals can be considered as functionals for $u=\vphi-\psi\in \cH(\chi)$ so that they are originally defined on the space of smooth $\chi$-psh functions in $[\chi]$. By the recent development of pluripotential theory, we know that they can all be extended to be functionals on a bigger space $\cE^1$ of finite energy $\chi$-psh functions. Following Guedj-Zeriahi \cite{GZ07}, we denote:
\begin{align}
&\cE=\cE(M, \chi)=\left\{u\in PSH(M, \chi); \int_M \chi_u^n=\int_M \chi^n\right\};\\
&\cE^1=\cE^1(M,\chi)=\left\{u\in \cE(M, \chi); \int_M |u| \chi_u^n<\infty\right\};\\
&\cE^1_{\rm norm}=\cE^1_{\rm norm}(M, \chi)=\left\{u\in \cE^1(M, \chi); \sup_M u=0\right\}.
\end{align}
Note that $\cE^1$ contains all bounded $\chi$-psh functions. By Darvas' work \cite{Dar14}, $\cE^1$ can be characterized as the metric completion of $\cH(\chi)$ under a Finsler metric $d_1$, where the Finsler metric $d_1$ is defined as follows. For any $u_0, u_1\in \cH$, the $d_1$ distance between $u_0$ and $u_1$ is defined as:
\[
d_1(u_1, u_2)=\inf_{u(s)} \left[\int_0^1 \left(\int_M |\dot{u}| \chi_u^n\right) ds\right].
\]
The $u(s)$ under the above infimum runs over the set of all smooth curves of metrics in $\cH(\chi)$ that connects $u_0$ and $u_1$. Darvas also proved that the infimum is achieved when $u(s)$ is the $C^{1,\bar{1}}$-geodesic segment connecting $u_0$ and $u_1$ (as obtained by X.Chen and is not contained in $C^\infty([0,1]\times M)$ in general according to Lempert-Vivas).

Following \cite{BBEGZ}, we endow $\cE^1$ with the strong topology. Then it is known that $u_j\rightarrow u$ in $\cE^1$ under the strong topology if and only if $I_{\psi+u}(\psi+u_j)\rightarrow 0$, if and only if $d_1(u_j, u)=0$.  Moreover in this case
$\sup(u_j)\rightarrow \sup(u)$ by Hartogs' lemma for plurisubharmonic functions.

The following compactness result is very important in the variational approach to solving Monge-Amp\`{e}re equations using pluripotential theory.
\begin{thm}[{\cite[Theorem 2.17]{BBEGZ}}]\label{thm-BBEGZ}
Assume $\chi$ is a smooth K\"{a}hler form.
Let $p>1$ and suppose $\mu=f \chi^n$ is a probability measure with $f\in L^p(X, \chi^n)$. For any $C>0$, the following set is compact in the strong topology:
\[
\left\{
u\in \cE^1; \quad
\sup_{M}u=0, \quad\frac{1}{V}\int_M \log\frac{\chi_{u}^n}{\mu}\chi^n_{u}<C
\right\}.
\]
\end{thm}

Let $(\cM, \cB, \cL)$ be a semi-ample test configuration of $(M,B,L)$ that dominates the product test configuration $(M,B,L)\times \bC$ via $\rho: \cM\rightarrow M\times\bC$. We denote by $\phi$ (resp. $\phi_\triv$) the non-Archimedean metric associated to $(\cM,\cB,\cL)$ (resp. $(M,B,L)\times\bC$).
We recall the corresponding non-Archimedean version of the energy functionals.
If $F$ is an energy appearing in \eqref{eq-Ephi}-\eqref{eq-Mphi}, then $F^\NA$ denotes its non-Archimedean version, in the sense that for any semi-ample test configuration $(\cM,\cB,\cL)$ if $\vphi(s)$ is the geodesic ray associated to $(\cM, \cB, \cL)$, then we have:
\begin{equation}
\lim_{s\rightarrow+\infty}\frac{F(\vphi_s)}{s}=F^\NA(\phi).
\end{equation}
We follow the notations in \cite{BHJ15, BBJ15}. By $(\bar{\cM}, \bar{\cB}, \bar{\cL}) \rightarrow \bP^1$ we mean the natural equivariant compactification of $(\cM,\cB,\cL)\rightarrow \bP^1$. We also denote:
\begin{eqnarray*}
K^{\log}_{(\bar{\cM},\bar{\cB})/\bP^1}=(K_{\bar{\cM}}+\bar{\cB})+\cM_{0,\rm red}-\pi^*(K_{\bP^1}+\{0\}).
\end{eqnarray*}
The non-Archimedean functional we will consider are the following:
\begin{eqnarray}
E^{\NA}(\phi)&=&\frac{(\phi^{n+1})}{(n+1)V}=\frac{\left(\bar{\cL}^{\cdot n+1}\right)}{(n+1)V},\label{eq-ENA} \\
J^\NA(\phi)&=&\frac{1}{V}\phi\cdot \phi_\triv^n-\frac{(\phi^{n+1})}{(n+1)V} \nonumber \\
 &=&\frac{1}{V}\left(\bar{\cL}\cdot \rho^*(L\times\bP^1)^{\cdot n}\right)-\frac{\left(\bar{\cL}^{\cdot n+1}\right)}{(n+1)V}, \label{eq-JNA} \\
 I^\NA(\phi)&=&\frac{1}{V}\left(\phi\cdot \phi_\triv^n-\phi^{n+1}+\phi_\triv\cdot (\phi^n)\right)\nonumber\\
&=&\frac{1}{V}\left(\bar{\cL}\cdot \rho^*(L\times\bP^1)^{\cdot n}\right)-\frac{1}{V}\left(\bar{\cL}^{\cdot n+1}\right)+\frac{1}{V}\left(\rho^*(L\times\bP^1)\cdot \bar{\cL}^{\cdot n}\right). \label{eq-INA}
\end{eqnarray}
Assume $L=-t^{-1} (K_M+B)$, we define:
\begin{eqnarray}
L_B^\NA(\phi)=L_B^\NA(t, \phi)&=& \inf_v\left(A_{(M,B)}(v)+t\; (\phi-\phi_\triv)(v)\right) \\
\cD^\NA(\phi)=\cD_B^\NA(t, \phi)&=&-E^\NA(\phi)+\frac{1}{t} L_B^\NA(\phi). \label{eq-DNA}
\end{eqnarray}
\begin{eqnarray}
H_B^\NA(\phi)
&=&\frac{1}{V} K^{\log}_{(\bar{\cM},\bar{\cB})/\bP^1}\cdot \bar{\cL}^{\cdot n}-\frac{1}{V} K^{\log}_{(M,B)\times\bP^1/\bP^1}\cdot \bar{\cL}^{\cdot n} \label{eq-HNA}\\
\cM_B^\NA(\phi)&=&\cM_B^\NA(t, \phi)
=\frac{1}{t}H_B^\NA(\phi)-(I^\NA-J^\NA) \nonumber\\
&=&\frac{1}{tV} \left(K^{\log}_{(\bar{\cM},\bar{\cB})/\bP^1}\cdot \bar{\cL}^{\cdot n}\right)+\frac{n}{(n+1)V}\left(\bar{\cL}^{n+1}\right).\label{eq-MNA}
\end{eqnarray}
To get the second identity in \eqref{eq-MNA}, we used $K^{\log}_{(M,B)\times\bP^1/\bP^1}=K_{(M,B)\times\bP^1/\bP^1}=-(K_M+B)\times \bP^1=t L\times\bP^1$.
We also need to recall the log-version of Tian's CM weight (see \cite{Tia97, Don12a, Wan12}):
\begin{eqnarray}
\CM_B(\phi)=\CM_B(t, \phi)&=&\frac{n}{(n+1)V} \left(\bar{\cL}^{\cdot n+1}\right)+\frac{1}{tV} \left(\bar{\cL}^{\cdot n}\cdot K_{(\bar{\cM},\bar{\cB})/\bP^1}\right)\label{eq-CM}.
\end{eqnarray}
\begin{rem}
\begin{enumerate}
\item
For special test configurations, the CM weight coincides with the Futaki invariant studied in \cite{Fut83, DT92}.
\item
Note the following inequality:
\begin{equation}
\CM_B(t,\phi)-
\cM^\NA_B(t, \phi)=\frac{1}{tV}\left(K_{(\bar{\cM},\bar{\cB})}-K^{\log}_{(\bar{\cM},\bar{\cB})}\right)\cdot \bar{\cL}^{\cdot n}\ge 0.
\end{equation}
This gives the fact that the CM weight is bigger than the asymptotic expansion of (log-)Mabuchi energy, with identity holds if and only if $\cM_0$ is reduced (see \cite{Tia97,BHJ16}).
\end{enumerate}
\end{rem}

\subsection{Existence of conical K\"{a}hler-Einstein metrics}

We will need the following log smooth version of the main result in \cite{BBJ15}:
\begin{thm}[see \cite{BBJ15}]\label{thm-BBJ}
Let $M$ be a smooth projective variety and $(M, B)$ be a log pair with klt singularities. Assume that $-(K_M+B)$ is an ample $\bQ$-Cartier divisor. Fix $t\in \bQ_{>0}$ and let $L=-t^{-1}(K_M+B)$. Using the notations from the above section, the following conditions are equivalent:
\begin{enumerate}
\item There exist $\delta_\cD \in (0,1)$ and $C_\cD>0$ such that $\cD_B \ge \delta_\cD J-C_\cD$ on $\cE^1$.
\item There exist $\delta_\cM\in (0,1)$ and $C_\cM>0$ such that $\cM_B \ge \delta_\cM J-C_\cM$ on $\in \cE^1$.
\item There exists $\delta^\NA \in (0,1)$ such that $\cD_B^\NA\ge \delta^\NA J^\NA$ for any test configuration $(\cM, \cB, \cL)$.
\item There exists $\delta^\NA \in (0,1)$ such that  $\cM_B^\NA\ge \delta^\NA J^\NA$ for any test configuration $(\cM, \cB, \cL)$.
\item There exists $\delta^\NA\in (0,1)$ such that $\CM_B^\NA\ge \delta^\NA J^\NA$ for any test configuration $(\cM, \cB, \cL)$.
\end{enumerate}
Moreover, if the above conditions hold true, the slope constants $(\delta_\cD, \delta_\cM, \delta^\NA)$ can be chosen satisfying 
\begin{equation}\label{eq-deltarel}
 \delta_\cM \ge \delta_\cD \ge 1-\left(\frac{1}{1+\delta_\cM}\right)^{1/n}, \quad \quad \delta^\NA\ge \delta_\cM\ge \frac{\delta^\NA}{2}.
\end{equation}
\end{thm}
\begin{rem}\label{rem-scaling}
By the rescaling property of the energy functionals, it is easy to verify that the above statement is independent of the scaling parameter $t$.  We include the parameter $t$ for the convenience of the later argument comparing energies for different K\"{a}hler classes. \end{rem}
Theorem \ref{thm-BBJ} can be proved by modifying argument in \cite{BBJ15} and replacing various quantities by their log versions.
We emphasize here that because our ambient space $M$ is smooth, there is no essential difficulty in carrying out the same argument as in \cite{BBJ15}.  For the reader's convenience, we will give a sketch of its proof after Berman-Boucksom-Jonsson in Appendix \ref{app-BBJ}. From its proof, we see that $\frac{\delta^\NA}{2}$ in \eqref{eq-deltarel}, which is a lower bound of $\delta_\cM$,  can be replaced by any positive number smaller than $\delta^\NA$ by possibly adjusting the other constant $\delta_\cM$.

The following existence result for the log version of Yau-Tian-Donaldson conjecture is recently obtained by Tian-Wang (\cite{TW19}).
\begin{thm}[\cite{TW19}]\label{thm-TW}
Let $(M, B)$ be a log smooth klt pair. Assume $-K_{(M,B)}:=-(K_M+B)$ is ample and $(M, B, -K_{(M,B)})$ is log-K-polystable, then $(M,B)$ admits a strong conical K\"{a}hler-Einstein
metric.
\end{thm}

We refer to \cite{Tia94, Don12a, JMR16, GP16, BBEGZ} for more details on conical K\"{a}hler-Einstein metrics.

\subsection{K-polystability and uniform K-stability}\label{sec-Kstab}
We recall the definitions of (log-)K-stability and uniform K-stability (see \cite{Tia97, Don02, LX12, BHJ15, Der16}).

\begin{defn}
Let $(M,B)$ be a log Fano pair with klt singularities. Fix $t\in \bQ_{>0}$ and denote $L=-t^{-1}(K_M+B)$.

\begin{enumerate}
\item
The pair $(M, B, L)$ is log-K-semistable if for any test configuration $(\cM, \cB, \cL)$ of $(M, B, L)$ we have $\CM(\cM, \cB, \cL)\ge 0$.
\item
$(M, B, L)$ is K-polystable if it is log-K-semistable and $\CM(\cM, \cB, \cL)=0$ for a normal test configuration $(\cM, \cB, \cL)$ iff $(\cM, \cB, \cL)$ is a product test configuration.

\item
The pair $(M, B, L)$ is uniformly log-K-stable with a slope constant $\delta=\delta^\NA>0$, if for any test configuration $(\cM, \cB, \cL)$ of $(M, B, L)$ we have:
\begin{equation}\label{eq-CMuK}
{\rm CM}(\cM, \cB, \cL)\ge \delta \cdot J^\NA(\cM, \cL).
\end{equation}
For simplicity of language, we will some times just say (uniformly) K-(semi)stable if the log pair is clear.
\end{enumerate}
\end{defn}
In the following, we will use valuative critierions of K-stability developed in \cite{Fuj16, Fuj17a, Li15, LX16}. For our convenience, we will use the following notation:
for any divisorial valuation $\ord_F$ over $X$, denote:
\[
\Phi_{(M,B)}(F)=\frac{A_{(M, B)}(F) (-K_M-B)^n}{\int_0^{+\infty}\vol_M(-K_M-B-xF)dx}, \quad \Phi_M(F)=\Phi_{(M,\emptyset)}(F).
\]
We also denote:
\[
\tilde{\delta}(M,B)=\inf_F\Phi_{(M,B)}(F), \quad \tilde{\delta}(M)=\tilde{\delta}(M, \emptyset),
\]
where $F$ ranges over all divisorial valuations over $M$.
\begin{thm}\label{thm-val}
Assume $(M,B)$ is a log Fano pair with log terminal singularities. Then we have the following criterion:

{\rm (1) (\cite{Li15,LX16}, \cite{Fuj16})} $(M,B)$ is log-K-semistable if and only if $\tilde{\delta}(M,B)\ge 1$.

{\rm (2) (\cite{Fuj17a}, see also \cite{BJ17})} $(M,B)$ is uniformly log-K-stable if and only if $\tilde{\delta}(M,B)>1$. Moreover, we can choose $\delta$ in \eqref{eq-CMuK} to be
$\frac{\tilde{\delta}-1}{n}$.
\end{thm}

\section{Uniform log-K-stability of log-smooth pairs}\label{sec-alg}

Let $X$ be an admissible $\mathbb{Q}$-Fano variety with an admissible resolution $\mu: M\rightarrow X$:
\[
K_M=\mu^* K_X+\sum_i a_i E_i \text{ with } a_i\in (-1,0].
\]
Letting $b_i=-a_i$, we have:
\[
K_M^{-1}=\mu^*K_X^{-1}+\sum_i b_i E_i
\]
with $b_i\in [0, 1)$. There exist $\theta=\{\theta_i\}$ such that $L_\epsilon:=\mu^*K_X^{-1}-\epsilon \sum_i \theta_i E_i$ is ample for any $\epsilon\in (0,1]$.
On $M$, for any $t\in (0,1]$, we would like to solve the conical KE equation in $c_1(L_\epsilon)$:
\begin{equation}\label{eq-teKE}
Ric(\omega)= t\omega+\frac{1-t}{m}\{H\}+\sum_i (b_i+(t\epsilon+(1-t)) \theta_i)\{E_i\}.
\end{equation}
Here $\omega\in 2\pi c_1(L_\epsilon)$.
For simplicity, we denote
\begin{align}
&E_\epsilon=\epsilon \sum_i\theta_iE_i,\\ 
&B
=B_{(\epsilon,t)}=(1-t)m^{-1} H+\sum_i (b_i+t\epsilon\theta_i+(1-t)\theta_i )E_i.\label{eq-Bept}
\end{align}
By taking class on both sides of \eqref{eq-teKE}, we get the numerical equivalence:
\begin{equation}\label{eq-num}
-K_M\equiv t L_\epsilon+B=t (L_0-E_\epsilon)+B.
\end{equation}

To solve \eqref{eq-teKE}, we need K-polystability of the pair $(M, B_{(\epsilon,t)})$. We will prove the (uniform) log-K-stability of
$(M, B_{(\epsilon,t)})$ using the criterions in Theorem \ref{thm-val}. To apply Fujita's criterion of uniform K-stability, for any divisorial valuation $\ord_F$ over $M$, we consider the quantity:
\begin{equation}\label{eq-betaF}
\Phi(\epsilon, t):=\Phi_{(M,B_{(\epsilon,t)})}(F)=\frac{A_{(M, B_{(\epsilon,t)})}(F) (-K_M-B_{(\epsilon,t)})^n}{\int_0^{+\infty}\vol_M(-K_M-B_{(\epsilon,t)}-xF)dx}.
\end{equation}

Because $-K_M-B=t L_\epsilon$, by change of variables, the integral in the denominator of \eqref{eq-betaF} is equal to:
\[
\int_0^{+\infty}\vol_M(t L_\epsilon-xF)dx=t^{n+1}\int_0^{+\infty}\vol_M(L_\epsilon-xF)dx.
\]

By \eqref{eq-num} we also have:
$(-K_M-B)^n=t^n(\mu^*K_X^{-1}-\sum_i \epsilon \theta_i  E_i)^n=L_\epsilon^n$.  So we have:
\begin{equation}
\Phi(\epsilon,t)=\Phi_{(M,B_{(\epsilon,t)})}(F)=\frac{A_{(M,B_{(\epsilon,t)})}(F) \cdot L_\epsilon^n}{t\int_0^{+\infty}\vol_M(L_\epsilon-x F)dx}.
\end{equation}
Using the identity $A_{(M,B)}(F)=A_M(F)-\ord_F(B)$ and \eqref{eq-Bept}, we can calculate:
\begin{eqnarray}\label{eq-Phit1}
\frac{\Phi(\epsilon,t)}{\Phi(\epsilon,1)}-1&=&\frac{1}{t}\frac{A_{(M,B_{(\epsilon,t)})}(F)}{A_{(M,B_{(\epsilon,1)})}(F)}-1= \frac{1}{t}\frac{A_{(M,B_{(\epsilon,t)})}(F)-t A_{(M,B_{(\epsilon,1)})}(F)}{A_{(M,B_{(\epsilon,1)})}(F)}\nonumber\\
&=&\frac{1-t}{t}\frac{A_M(F)-\frac{1}{m}\ord_F(H)-\ord_F(\sum_i (b_i+\theta_i)E_i)}{A_{(M,B_{(\epsilon,1)})}(F)}\nonumber\\
&=&\frac{1-t}{t}\frac{A_{(M,B_{(\epsilon,1)})}(F)-\ord_F(-B_{(\epsilon,1)}+m^{-1}H+\sum_i (b_i+\theta_i)E_i)}{A_{(M,B_{(\epsilon,1)})}(F)}\nonumber\\
&=&\frac{1-t}{t}\left(1-\frac{\ord_F\left(m^{-1}H+(1-\epsilon)\sum_i \theta_i E_i\right)}{A_{(M,B_{(\epsilon,1)})}(F)}\right)\nonumber\\
&\ge &\frac{1-t}{t}\left(1-\frac{1}{\lct_{(M,B_{(\epsilon,1)})}\left(m^{-1}H+(1-\epsilon)\sum_i \theta_i E_i\right)}\right).
\end{eqnarray}
The last inequality follows from the following defining inequality of log canonical threshold: for any effective divisor $D$, we have
\begin{eqnarray*}
\lct_{(M,B)}\left( D \right)&=&\sup \left\{\alpha>0; (M, B+\alpha D) \text{ is log canonical} \right\}\\
&=&\inf_F \frac{A_{(X,B)}(F)}{\ord_F(D)}\le \frac{A_{(X,B)}(F)}{\ord_F(D)},
\end{eqnarray*}
where the above $\inf$ ranges over all divisorial valuations over $M$. We notice that:
\begin{eqnarray*}
B_{(\epsilon,1)}+\alpha (m^{-1}H+(1-\epsilon)\sum_i\theta_i E_i)&=& \alpha m^{-1}H+\sum_i (b_i+(\epsilon+\alpha(1-\epsilon))\theta_i) E_i
\end{eqnarray*}
Since $H\cup \cup_i E_i$ has simple normal crossing support, we get:
\[
\alpha:=\lct_{(M,B_{(\epsilon,1)})}\left(m^{-1}H+(1-\epsilon)\theta_i E_i\right)=\min\left\{m, \frac{\frac{1-b_i}{\theta_i}-\epsilon}{1-\epsilon}\right\}.
\]
So as long as we choose $(\theta_i, m, H)$ such that
\begin{eqnarray}\label{eq-choosethetam}
&\theta_i\le (1-b_i)/2\quad m\ge 2, \text{ and }\nonumber \\
&H\in |-mL_1| \text{ is smooth such that } H+\sum_i E_i \text{ is simple normal crossing},
\end{eqnarray}
then $\alpha\ge 2$ and we get from \eqref{eq-Phit1} that:
\begin{equation}
\Phi(\epsilon,t)\ge \Phi(\epsilon,1)\left(1+(t^{-1}-1)\frac{1}{2}\right).
\end{equation}
Next we need to estimate $\Phi(\epsilon,1)$ in terms of $\Phi(1, 0)$:
We have:
\begin{eqnarray*}
\frac{\Phi(\epsilon,1)}{\Phi(0,1)}&=& \frac{A_{(M,B_{(\epsilon,1)})}(F)}{A_{(M,B_{(0,1)})}(F)}\cdot\frac{\int_0^{+\infty}\vol_M(L_0-xF)dx}{\int_0^{+\infty}\vol_M(L_\epsilon-xF)dx}\cdot\frac{(\mu^*K_X^{-1}-E_\epsilon)^n}{(\mu^*K_X^{-1})^n}\\
&=:&R_1\cdot R_2\cdot R_3.
\end{eqnarray*}
Recall that $B_{(\epsilon,t)}=\sum_i (b_i+\epsilon \theta_i)E_i=B_{(0,1)}+E_\epsilon$. Because $\cup_i E_i$ has simple normal crossings and, for any $\alpha>0$,
$B_{(0,1)}+\alpha E_\epsilon=\sum_i (b_i+\alpha \epsilon\theta_i)E_i$, we have:
\begin{eqnarray*}
R_1&=&\frac{A_{(M,B_{(\epsilon,t)})}(F)}{A_{(M,B_{(0,1)})}(F)}=\frac{A_{(M,B_{(0,1)})}(F)-\ord_F(\sum_i \epsilon \theta_i E_i)}{A_{(M, B_{(0,1)})}(F)}\\
&\ge& 1-\lct^{-1}_{(M, B_{(0,1)})}(E_\epsilon)=1-\left(\min_i \frac{1-b_i}{\epsilon \theta_i}\right)^{-1}\\
&=&1-\max_i \frac{\epsilon \theta_i}{1-b_i}.
\end{eqnarray*}
The second ratio $R_2\ge 1$: $E_\epsilon$ is effective and hence
\begin{eqnarray*}
\int^{+\infty}_0 \vol_M(\mu^*K_X^{-1}-E_\epsilon-xF)dx &\le& \int^{+\infty}_{0}\vol_M(\mu^*K_X^{-1}-xF)dx.
\end{eqnarray*}
Combining the above estimates, we get:
\begin{eqnarray}\label{eq-lbPsi}
\Phi(\epsilon,t)&\ge&\Phi(\epsilon,1)\left(1+(t^{-1}-1)\frac{1}{2}\right)\nonumber\\
&\ge& R_1\cdot R_2\cdot R_3 \cdot \tilde{\delta}(X)(1+(t^{-1}-1)2^{-1})\nonumber\\
&\ge& \left(1-\epsilon\max_i \frac{\theta_i}{1+a_i}\right)\left(1+(t^{-1}-1)2^{-1})\right)\frac{(\mu^*K_X^{-1}-E_\epsilon)^n}{K_X^{-n}}\cdot\tilde{\delta}(X)\nonumber\\
&=:&\tilde{\delta}(\epsilon,t).
\end{eqnarray}
Now we can state and prove the main result of this section, which may be of independent interest.
\begin{prop}\label{prop-tstable}
{\rm (1)}
Assume that $X$ is K-semistable. For any $t\in (0,1)$, there exists $\epsilon^*=\epsilon^*(t)>0$ such that for any $0<\epsilon\le \epsilon^*$, $(M, B_{(\epsilon,t)})$ is uniformly K-stable with a slope constant
\begin{equation}\label{eq-tdelta}
\delta=(t^{-1}-1)/(4n).
\end{equation}

{\rm (2)}
Assume that $X$ is uniformly K-stable. 
There exists $\epsilon^*_1>0$ such that for any $0<\epsilon\le \epsilon^*_1$, $(M, B_{(\epsilon,1)})$ is uniformly K-stable with a slope constant $\delta=\frac{1}{2n}(\tilde{\delta}(X)-1)$.
\end{prop}

\begin{proof}
By \eqref{eq-lbPsi}, we get:
\begin{equation}
\tilde{\delta}(M, B_{(\epsilon,t)})=\inf_F \Phi_{(M,B_{(\epsilon,t)})}(F)\ge \tilde{\delta}(\epsilon,t).
\end{equation}
Notice that $\tilde{\delta}(\epsilon, t)$ in \eqref{eq-lbPsi} satisfies:
\begin{equation}\label{eq-limdel}
\lim_{\epsilon\rightarrow 0+}\tilde{\delta}(\epsilon,t)=(1+(t^{-1}-1)2^{-1})\tilde{\delta}(X)
\end{equation}
Assume $X$ is K-semistable, then by Theorem \ref{thm-val}.(1) the right-hand-side of \eqref{eq-limdel} is greater than or equal to $(1+(t^{-1}-1)(1-m^{-1}))$ which is strictly bigger than $1$ for $t\in (0,1)$ and $m\ge 2$. So there exists $\epsilon^*>0$ such that for any $\epsilon\in (0,\epsilon^*]$, we have
\[
\tilde{\delta}(\epsilon, t)\ge 1+\frac{(t^{-1}-1)2^{-1}}{2}.
\]
We then get the first statement statement by applying Theorem \ref{thm-val}.(2).

If $X$ is uniformly K-stable, then
\[
\lim_{\epsilon\rightarrow 0+}\tilde{\delta}(\epsilon, 1)=\tilde{\delta}(X)>1.
\]
The second statement follows by applying Theorem \ref{thm-val}.(2).

\end{proof}

\begin{rem}
While completing this paper, we noted a recent openness result of K. Fujita in \cite[Theorem 5.7]{Fuj17a}, which is related to but different with the above proposition. The authors thank K. Fujita for helpful comment on this.
\end{rem}

A key feature of the above proposition that will be important to us is that the slope constant $\delta$ does not depend on $\epsilon$ as long as $\epsilon$ is sufficiently small.

\begin{prop}\label{prop-semiexist}
Assume that $X$ is K-semistable.
For any $t\in (0,1)$, there exists $\epsilon\ll 1$, such that there exists a strong conical KE on the log smooth pair $(M, B_{(\epsilon,t)})$ where $B_{(\epsilon,t)}$ is defined as in \eqref{eq-Bept}.
\end{prop}

\begin{proof}
Because $(M, B_{(\epsilon,t)})$ is uniformly log-K-stable, and in particular log-K-polystable, this follows from Tian-Wang's recent work in \cite{TW19}. Note one could also use Theorem \ref{thm-BBJ} to get the properness of log-Mabuchi-energy and hence the existence of strong conical KE on $(M, B_{(\epsilon,t)})$ using the method in \cite{GP16} (see also \cite{JMR16}).
\end{proof}

\section{Uniform $L^\infty$-estimates}\label{sec-MA}
{\bf Notations:}
Fix $\theta$ sufficiently small such that $L_1:=\mu^*K_X^{-1}-\sum_i \theta_i E_i$ is ample. Choose a smooth Hermitian metric $e^{-\psi_1}$ on $L_1$ such that $\chi_1:=\sddb\psi_1> 0$.
Fix $m\gg 1$ sufficiently divisible such that $mK_X^{-1}$ is a very ample line bundle. Choose a basis of $|-mK_X|$ and get a Kodaira embedding $\iota: X\hookrightarrow \cP^N$. We will denote both the Hermitian metric on $K_X^{-1}$ and $\mu^*K_X^{-1}$ by
$e^{-\psi_0}$ or $h_\FS$ the $1/m$ times the pull back of the Fubini-Study metric via the map $\iota$ and $\mu\circ\iota$.
We also use $\omega_{\FS}$ to denote both the Chern curvature of $e^{-\psi_0}$ on  $X$ and $M$.

Because $L_0=\mu^*K_X^{-1}=L_1+\sum_i \theta_iE_i$, by using the $\partial\bar{\partial}$-lemma, it is easy to see that there exists Hermitian metric $\kappa_i$ on the line bundle determined by $E_i$ such that if we denote by $e^{-\kappa_\theta}$ or simply by $e^{-\kappa}$ the Hermitian metric $e^{-\sum_i \theta_i \kappa_i}$ on $\sum_i \theta_i E_i$ then the following identity holds for Hermitian metrics on $L_0$:
\[
e^{-\psi_0}=e^{-\psi_1-\kappa}
\]
For simplicity, we denote
\begin{equation}
\eta:=\sddb \kappa.
\end{equation}
Then we have
\[
\chi_0:=\omega_{\FS}=\sddb\psi_0=\sddb\psi_1+  \sddb\kappa=\sddb\psi_1+\eta \ge 0.
\]
For any $\epsilon\in [0, 1]$, we then have a smooth Hermitian
metric $e^{-\psi_\epsilon}$ on $L_\epsilon=\mu^*K_X^{-1}-\epsilon \sum_i \theta_iE_i$:
\begin{equation}\label{eq-psieps}
\psi_\epsilon=\psi_0-\epsilon \kappa=(1-\epsilon)\psi_0+\epsilon (\psi_0-\kappa)=(1-\epsilon)\psi_0+\epsilon \psi_1.
\end{equation}
Its Chern curvature is a positive $(1,1)$-form on $M$:
\[
\chi_\epsilon:=\sddb\psi_\epsilon=(1-\epsilon)\chi_0+\epsilon \chi_1\ge 0.
\]
In particular $\chi_\epsilon>0$ if $\epsilon>0$ since $\chi_1>0$.

Let $PSH([L_\epsilon])$ denote the set of possibly singular Hermitian metrics $e^{-\vphi}$ n $L_\epsilon$  satisfying $\sddb\vphi\ge 0$. For any $\vphi\in PSH([L_\epsilon])$, denote $\omega_\vphi=\sddb \vphi$.
We want to solve for $\vphi=\vphi_\epsilon$ such that $\omega_\vphi=\omega_{\vphi_\epsilon}$ satisfies:
\begin{eqnarray}\label{eq-}
Ric(\omega_\vphi)-t\omega_\vphi-\frac{1-t}{m}\{H\}-\sum_i (b_i+t\epsilon \theta_i+(1-t)\theta_i)\{E_i\}=0
\end{eqnarray}
Geometrically,
the cone angle along $H$ is $2\pi t\in (0,2\pi]$. The cone angle along $E_i$ is $2\pi \beta_i$ where $\beta_i:=1-(b_i+t\epsilon \theta_i+(1-t)\theta_i)$.
Notice that as $\epsilon\rightarrow 0+$,  $\beta_i$ approaches
\begin{equation}
1-(b_i+(1-t)\theta_i)\in \left(\frac{1+t}{2}(1-b_i), 1-b_i\right)\subset (0,1].
\end{equation}
for any $t\in (0,1]$. On the other hand, for $0<\epsilon\ll 1$,
$\beta_i\in (0,1)$.

The left-hand-side can be written as:
\begin{eqnarray*}
-\sddb \log \omega_\vphi^n-t \sddb\vphi-\frac{1-t}{m}\log |s_H|^2-\log \prod_i |s_{i}|^{2(b_i+t\epsilon \theta_i+(1-t)\theta_i)}
\end{eqnarray*}
So we get the Monge-Amp\`{e}re equation that corresponds to the above equation:
\begin{equation}\label{eq-sCMA}
(\sddb\vphi)^n=\frac{e^{-t\vphi}}{|s_H|^{2(1-t)/m}\prod_i |s_i|^{2(1-\beta_i)}}.
\end{equation}
Both sides are considered as measures on $M$. We can re-write this into a more familiar form. Let $u=\vphi-\psi_\epsilon$, then $u$ is a globally defined $\chi_\epsilon$-psh function on $M$.
Then we have
\begin{eqnarray}\label{eq-MAref}
(\chi_\epsilon+\sddb u)^n&=&\omega_\vphi^n=e^{-t (\vphi-\psi_\epsilon)}\frac{e^{-t\psi_\epsilon}}{|s_H|^{2(1-t)/m}\prod_i |s_i|^{2(b_i+t\epsilon \theta_i+(1-t)\theta_i)}}.
\end{eqnarray}
Note that $e^{-(\psi_0+\sum_i b_i\kappa_i)}$ is a smooth Hermitian metric on $\mu^*(-K_X)+\sum_i b_i E_i=-K_M$ and hence corresponds to a smooth
volume form $\Omega$. Then we can rewrite the right-hand-side of \eqref{eq-MAref} as:
\begin{eqnarray}\label{eq-Omegept}
&&\frac{e^{-t\psi_\epsilon}}{|s_H|^{2(1-t)/m}\prod_i |s_i|^{2(b_i+t\epsilon\theta_i+(1-t)\theta_i)}}\nonumber\\
&=&e^{-(\psi_0+\sum_i b_i\kappa_i)}  \frac{e^{(1-t)(\psi_0-\sum_i \theta_i\kappa_i)} }{|s_H|^{2(1-t)/m}}
\frac{e^{\sum (t\epsilon+(1-t))\theta_i \kappa_i}}{|s_\theta|^{2(t\epsilon+1-t)}} \frac{e^{\sum_ib_i\kappa_i}}{\prod_i |s_i|^{2b_i}}\nonumber\\
&=&\frac{\Omega}{\|s_H\|_{m\psi_1}^{2(1-t)/m}\|s_\theta\|_{\kappa_\theta}^{2(t \epsilon+1-t)}\prod_{i}\|s_i\|_{\kappa_i}^{2b_i}}=:\Omega(\epsilon, t).
\end{eqnarray}
If we define the following constant:
\begin{equation}\label{eq-gamma}
\gamma=\gamma(\epsilon, t)=\frac{1}{t}\log\frac{\int_M \Omega(\epsilon, t)}{\int_M \chi_\epsilon^n}.
\end{equation}
then
\[
\int_M e^{-t\gamma}\Omega(\epsilon, t)=\int_M \omega_{\vphi}^n=\int_M \chi_\epsilon^n=(2\pi L_\epsilon)^n.
\]
So by the equation \eqref{eq-MAref}, there exists $p\in M$ such that $u(p)-\gamma=(\vphi-\psi_\epsilon)(p)-\gamma=0$. Hence we have:
\begin{equation}
 \|\vphi-\psi_\epsilon\|_\infty-\gamma\le {\rm osc}(\vphi-\psi_\epsilon)\le 2\|\vphi-\psi_\epsilon\|_\infty+2\gamma.
\end{equation}
On other other hand, it is easy to see that $\gamma$ is uniformly bounded with respect to $(\epsilon, t)$. So there exists a constant $C>0$ independent of $\epsilon,t$, such that
\begin{equation}\label{eq-osc2C0}
\|\vphi-\psi_\epsilon\|_{\infty}-C\le {\rm osc}(\vphi-\psi_\epsilon)\le 2\|\vphi-\psi_\epsilon\|_\infty+C.
\end{equation}

If $u\in PSH(L_\epsilon, \chi_\epsilon)$, then
\[
\chi_\epsilon+\sddb u=\chi- \epsilon \sddb\kappa+\sddb u\ge 0.
\]
For any $\epsilon^*\ge \epsilon$, we then get:
\begin{eqnarray}\label{eq-s1u}
\chi_{\epsilon^*}+\sddb u&=& \chi_\epsilon+\sddb u-(\epsilon^*-\epsilon) \sddb\kappa\ge -(\epsilon^*-\epsilon)\sddb \kappa.
\end{eqnarray}
A key observation for us to proceed is that:
\begin{lem}
For any $\epsilon^*\in [0,1/2]$, $\sddb\kappa\le 2 \chi_{\epsilon^*}$.
\end{lem}
\begin{proof}
This follows immediately from the following identity:
\begin{equation}
\chi_{\epsilon^*}-\frac{1}{2}\sddb\kappa=\chi_0-\left(\epsilon^*+\frac{1}{2}\right)\sddb\kappa=\left(\frac{1}{2}-\epsilon^*\right)\chi_0+\left(\frac{1}{2}+\epsilon^*\right)\chi_1.
\end{equation}
\end{proof}
For simplicity of notations, we denote $\tau=\epsilon^*-\epsilon>0$. Then we get from the above lemma and \eqref{eq-s1u} that:
\begin{eqnarray*}
\chi_{\epsilon^*}+\sddb u\ge -(\epsilon^*-\epsilon)2 \chi_{\epsilon^*}=-2\tau \chi_{\epsilon^*}.
\end{eqnarray*}
So we can define the map:
\begin{eqnarray}\label{eq-Pepsilon}
P_\epsilon: PSH(L_\epsilon)&\longrightarrow& PSH(L_{\epsilon^*})\\
\vphi&\mapsto& \psi_{\epsilon^*}+\frac{1}{1+2\tau}(\vphi-\psi_\epsilon).\nonumber
\end{eqnarray}
Then we have:
\begin{equation*}
P_\epsilon(\vphi)=\vphi+\frac{\tau}{1+2\tau}\left[-2\vphi+2\psi_0-(1+2\epsilon^*) \kappa\right].
\end{equation*}
and, if we denote
\begin{equation}\label{eq-sigma}
\sigma=\frac{1}{1+2\tau}\sddb (-2\vphi+2\psi_0-(1+2\epsilon^*)\kappa), 
\end{equation}
then we have
\begin{eqnarray*}
\sddb P_\epsilon(\vphi)&=&\sddb \vphi+\tau\sigma.
\end{eqnarray*}
We will use the following lemma which is easily verified.
\begin{lem}\label{lem-intbyC0}
Let $I\times J\subset \bR\times\bR$. Assume that $\{\sigma_i=\sigma_i(\epsilon,t)\}_{(\epsilon, t)\in I\times J}$ for $i=1,\dots, n$  are $n$ families of real closed positive $(1,1)$-currents with bounded potentials. If there exists a compact set $K\subset H^{1,1}(M,\bR)$ such that $\{[\sigma_i]\}\subset K$ for $i=1,\dots, n$. Then there exists a constant $C=C(M,K)>0$ such that for any $f\in L^\infty(M, \bR)$, we have:
\begin{eqnarray*}
\left|\int_M f \sigma_1\wedge \dots \wedge \sigma_n\right| \le C\|f\|_{L^\infty}.
\end{eqnarray*}
\end{lem}
\begin{lem}\label{lem-compE}
There exists $C=C(n)>0$, such that for any $\vphi\in PSH(L_\epsilon)$, we have
\begin{equation}\label{eq-Ediff}
\left|E_{\psi_{\epsilon^*}}(P_\epsilon(\vphi))-E_{\psi_\epsilon}(\vphi)\right|\le C\tau \|\vphi-\psi_\epsilon\|_{\infty},
\end{equation}
\begin{equation}\label{eq-Jdiff}
\left|J_{\psi_{\epsilon^*}}(P_\epsilon(\vphi))-J_{\psi_\epsilon} (\vphi))\right|\le C \tau \|\psi_\epsilon-\vphi\|_{\infty}.
\end{equation}
\end{lem}

\begin{proof}
Denote $V_\epsilon=(L_\epsilon^{\cdot n})$.
By using the expression of $P_\epsilon(\vphi)$ in \eqref{eq-Pepsilon} and $E_\psi(\vphi)$ in \eqref{eq-Ephi}, we get
the identities:
\begin{eqnarray*}
E_{\psi_{\epsilon^*}}(P_\epsilon(\vphi))&=&\frac{1}{(n+1)V_{\epsilon^*}}\int_M (P_\epsilon(\vphi)-\psi_{\epsilon^*})\sum_i (\sddb\psi_{\epsilon^*})^{n-i}\wedge (\sddb P_\epsilon(\vphi))^i\\
&=&\frac{1}{(n+1)V_{\epsilon^*}}\int_M \frac{1}{1+2\tau}(\vphi-\psi_\epsilon) \sum_i (\sddb \psi_\epsilon-\tau\sddb\kappa)^{n-i}\wedge (\sddb\vphi+\tau\sigma)^i\\
&=&\frac{1}{(n+1)V_\epsilon}\int_M(\vphi-\psi_\epsilon)\sum_i (\sddb\psi_\epsilon)^{n-i}\wedge (\sddb\vphi)^i+\tau\cdot \Delta\\
&=&E_{\psi_\epsilon}(\vphi)+\tau\cdot \Delta.
\end{eqnarray*}
Here $\Delta$ is of the form:
\begin{equation}
\Delta= \sum_k \int_M \left(\pm (\vphi-\psi_\epsilon)\right) \sigma^{(k)}_1\wedge \dots\wedge \sigma^{(k)}_n,
\end{equation}
which by Lemma \ref{lem-intbyC0} is uniformly bounded independent of $\epsilon$ because the classes of $\sigma^{(k)}_i$ are uniformly bounded.

By the expression of $J$-energy in \eqref{eq-Jphi} and Lemma \ref{lem-compE}, we get:
\begin{eqnarray*}
\left|J_{\psi_{\epsilon^*}}(P_\epsilon(\vphi))-J_{\psi_\epsilon}(\vphi)\right|&\le& \left|E_{\psi_{\epsilon^*}}(P_\epsilon(\vphi))-E_{\psi_\epsilon}(\vphi)\right|+\left|{\bf II}\right|\\
&\le& C\tau \|\psi_\epsilon-\vphi\|_{\infty}+\left|{\bf II}\right|.
\end{eqnarray*}
where
\[
{\bf II}=\frac{1}{V_{\epsilon^*}}\int_M (P_\epsilon(\vphi)-\psi_{\epsilon^*})(\sddb\psi_{\epsilon^*})^n-\frac{1}{V_\epsilon}\int_M (\vphi-\psi_\epsilon)(\sddb\psi_\epsilon)^n.
\]
To estimate $\left|\bf II\right|$, we note that:
\begin{eqnarray*}
\int_M (P_\epsilon(\vphi)-\psi_{\epsilon^*})(\sddb\psi_{\epsilon^*})^n=\int_M \frac{1}{1+2\tau}(\vphi-\psi_\epsilon)(\sddb \psi_\epsilon-\tau \sddb \kappa)^n.
\end{eqnarray*}
The estimate for $J$ follows by similar argument for $E$ by applying Lemma \ref{lem-intbyC0}.

\end{proof}

To state the next result, recall that we have denoted:
\[
B=B_{(\epsilon,t)}=\frac{1-t}{m}H+\sum_i (b_i+(1-t)\theta_i+t\epsilon \theta_i)E_i.
\]
The Ding energy associated to the pair $(X, B)$ is denoted by
\begin{equation}\label{eq-recallD}
\cD_{B_{(\epsilon,t)}}(t, \vphi)=-E_{\psi_\epsilon}(\vphi)+L_{B_{(\epsilon,t)}}(t, \vphi).
\end{equation}

\begin{lem}\label{lem-compD}
Notations as above. Denote $\tau=\epsilon^*-\epsilon$ for $0<\epsilon<\epsilon^*\le\frac{1}{2}$.
There exists constant $C>0$ independent of $\epsilon$ such that if $\vphi\in PSH_\infty(L_\epsilon)$ then we have
\begin{equation}\label{eq-Detphi}
\cD_{B_{(\epsilon, t)}}(t, \vphi)-\cD_{B_{(\epsilon^*,t)}}(t, P_\epsilon(\vphi))\ge -C\tau\|\psi_\epsilon-\vphi\|_{\infty}-C.
\end{equation}
\end{lem}
\begin{proof}
By \eqref{eq-Detphi} and Lemma \ref{lem-compE}, we just need to compare $L(t, \epsilon^*, P_\epsilon(\vphi))$ and $L(t, \epsilon, \vphi)$.
We have:
\begin{eqnarray*}
L_{B_{(\epsilon^*,t)}}(t, P_\epsilon(\vphi))&=&-\frac{1}{t}\log\int_M \frac{e^{-tP_\epsilon(\vphi)}}{|s_H|^{2(1-t)/m}\prod_i|s_i|^{2(b_i+t\epsilon^*_i\theta_i+(1-t)\theta_i)}}=:-\frac{1}{t}\log\int_M G .
\end{eqnarray*}
We re-combine the numerator and denominator of the integrand as follows:
\begin{eqnarray*}
e^{-t P_\epsilon(\vphi)}&=&e^{-t (\frac{1}{1+2\tau}(\vphi-\psi_\epsilon)+\psi_{\epsilon^*})}=e^{-t\vphi}e^{+t \frac{2\tau}{1+2\tau}(\vphi-\psi_\epsilon)}e^{t(\psi_\epsilon-\psi_{\epsilon^*})}.
\end{eqnarray*}
\begin{eqnarray*}
\frac{1}{\prod_i |s_i|^{2 t\epsilon^*_i\theta}}=\frac{1}{\prod_i |s_i|^{2t\epsilon \theta_i}|s_i|^{2t\tau\theta_i}}=\frac{1}{|s_\theta|^{2t\tau}\prod_i |s_i|^{2t\epsilon \theta_i}}
\end{eqnarray*}
Notice that by \eqref{eq-psieps}, we have $\psi_\epsilon-\psi_{\epsilon^*}=(\epsilon^*-\epsilon)\kappa=\tau\kappa$. Moreover over $M$ there exists a constant $C>0$ independent of $t$ and $\tau$ such that
\[
\frac{e^{t\tau\kappa}}{|s_\theta|^{2t\tau}}=\frac{1}{\|s_\theta\|^{2t\tau}}\ge C.
\]
So we get the estimate:
\begin{eqnarray*}
G&=&e^{\frac{2\tau t}{1+2\tau}(\vphi-\psi_\epsilon)}\frac{e^{-t\vphi}}{|s_H|^{2(1-t)/m}\prod_i|s_i|^{2(b_i+t\epsilon \theta_i+(1-t)\theta_i)}}\frac{e^{t\tau\kappa}}{|s_\theta|^{2t\tau}}\\
&\ge& C e^{-\frac{2\tau t}{1+2\tau}\|\vphi-\psi_\epsilon\|_{\infty}}\frac{e^{-t\vphi}}{|s_H|^{2(1-t)/m}\prod_i |s_i|^{2(b_i+t\epsilon \theta_i+(1-t)\theta_i)}}
\end{eqnarray*}
and hence:
\begin{eqnarray}\label{eq-estLB}
L_{B_{(\epsilon^*,t)}}(t, P_\epsilon(\vphi))&\le&-\frac{1}{t}\log\int_M \frac{e^{-t\vphi}}{|s_H|^{2\frac{1-t}{m}}\prod_i |s_i|^{2(b_i+t\epsilon \theta_i+(1-t)\theta_i)}}+\frac{2\tau t}{1+2\tau}\|\vphi-\psi_\epsilon\|_{\infty}+\log C\nonumber\\
&\le& L_{B_{(\epsilon,t)}}(t, \vphi)+C \tau\|\vphi-\psi_\epsilon\|_{\infty}+\log C.
\end{eqnarray}
The inequality \eqref{eq-Detphi} follows \label{eq-estLB}, \eqref{eq-Ediff} and the definition of $\cD_{B_{(\epsilon,t)}}$ (see \eqref{eq-recallD}).

\end{proof}
\begin{prop}[Uniform properness]\label{prop-uniproper}
Assume $X$ is K-semistable.
Fix $t\in (0,1]$. There exist $\epsilon^*=\epsilon^*(t)>0$, $\delta^*=\delta^*(t)>0$ and a constant $C>0$, such that for any $\epsilon\in (0,\epsilon^*]$ and any $\vphi\in PSH(L_\epsilon)$, the following inequality holds:
\begin{equation}\label{eq-tproper}
\cD_{B_{(\epsilon, t)}}(t, \vphi)\ge \delta^* J_{\psi_\epsilon}(\vphi)-C \epsilon^* \|\vphi-\psi_\epsilon\|_\infty-C.
\end{equation}
\end{prop}
\begin{proof}
If $X$ is K-semistable, then by Proposition \ref{prop-tstable}, there exists $\epsilon^*=\epsilon^*(t)>0$ such that $(M, B_{(\epsilon,t)})$ is uniformly K-stable for $\epsilon\in (0,\epsilon^*]$ with the slope constant $\delta=(t^{-1}-1)/(4n)>0$.
By Theorem \ref{thm-BBJ} and \eqref{eq-deltarel}, if we choose
\begin{equation}\label{eq-delta*}
\delta^*:=\delta^*(t)=1-\left(\frac{1}{1+\frac{(t^{-1}-1)}{8n}}\right)^{1/n}>0,
\end{equation}
then there exists a constant $C>0$ such that for any $\vphi\in PSH(L_{\epsilon^*})$,
\begin{eqnarray}\label{eq-*proper}
\cD_{B_{(\epsilon^*,t)}}(t, \vphi)\ge \delta^* J_{\psi_{\epsilon^*}}(\vphi)-C.
\end{eqnarray}
For any $\vphi\in PSH(L_\epsilon)$, $P_\epsilon(\vphi)\in PSH(L_{\epsilon^*})$ (see \eqref{eq-Pepsilon}) and hence:
\begin{eqnarray}\label{eq-*properPphi}
\cD_{B_{(\epsilon^*,t)}}(t, P_\epsilon(\vphi))\ge \delta^* J_{\psi_{\epsilon^*}}(P_\epsilon(\vphi))-C.
\end{eqnarray}
By the above lemmas, the following estimates hold:
\begin{eqnarray}\label{eq-tproper1}
\cD_{B_{(\epsilon, t)}}(t, \vphi)&\ge& \cD_{B_{(\epsilon^*,t)}}(t,P_\epsilon(\vphi))-C\tau \|\vphi-\psi\|-C\quad \text{ (Lemma \ref{lem-compD})}\nonumber\\
&=&\delta^* J_{\psi_{\epsilon^*}}(P_\epsilon(\vphi))-C \tau \|\vphi-\psi_\epsilon\|-C \quad \text{ \eqref{eq-*properPphi}}\nonumber\\
&=&\delta^* J_{\psi_\epsilon}(\vphi)-C\delta^*\tau \|\vphi-\psi_\epsilon\|-C\tau \|\vphi-\psi_\epsilon\|-C. \text{ (\eqref{eq-Jdiff})}
\end{eqnarray}
where the constant $C=C(n,\epsilon^*)$ appeared above is independent of $\epsilon$. So we get \eqref{eq-tproper} since $\tau=\epsilon^*-\epsilon\le \epsilon^*$. 
\end{proof}

\begin{prop}\label{prop-Ricproper}
Assume that $X$ is K-semistable. There exist a constant $\epsilon^*=\epsilon^*(t)>0$, $\hdelta=\hdelta(t)>0$ and a constant $C>0$, such that if $\epsilon\in (0,\epsilon^*]$ and
$\omega_\vphi\in 2\pi c_1(L_\epsilon)$ is a strong conical metric on the log smooth pair $(M, B_{(\epsilon,t)})$ that satisfies $Ric(\omega_\vphi)\ge \frac{1}{2} \omega_\vphi$, then we have:
\[
\cD_{B_{(\epsilon,t)}}(t, \vphi)\ge \hdelta \|\vphi-\psi_\epsilon\|_\infty-C.
\]
\end{prop}
\begin{proof}
Using Green's formula for $\chi_{\epsilon^*}$, we get:
\begin{equation}
\sup_M (P_\epsilon(\vphi)-\psi_{\epsilon^*})\le \frac{1}{V_{\epsilon^*}}\int_M (P_\epsilon(\vphi)-\psi_{\epsilon^*})\chi_{\epsilon^*}^n+C(\epsilon^*).
\end{equation}
Using the expression of $P_\epsilon(\vphi)$ in \eqref{eq-Pepsilon}, we know that this implies
\begin{eqnarray*}
\sup_M (\vphi-\psi_\epsilon)&\le&\frac{1}{V_{\epsilon^*}}\int_M (\vphi-\psi_\epsilon)\chi_{\epsilon^*}^n+(1+2\tau)C(\epsilon^*)\\
&\le&\frac{1}{V_{\epsilon^*}}\int_M(\vphi-\psi_\epsilon)\chi_{\epsilon^*}^n+(1+2\epsilon^*)C(\epsilon^*).
\end{eqnarray*}
Using the identity $\chi_{\epsilon^*}=\chi_\epsilon-\tau \sddb\kappa$, it is easy to verify that
\begin{eqnarray*}
\left|\frac{1}{V_{\epsilon^*}}\int_M(\vphi-\psi_{\epsilon})\chi_{\epsilon^*}^n-\frac{1}{V_\epsilon}\int_M(\vphi-\psi_\epsilon)\chi_{\epsilon}^n\right|\le C\tau \|\vphi-\psi_\epsilon\|_{\infty}.
\end{eqnarray*}
So we get:
\begin{eqnarray}\label{eq-supbd}
\sup_M(\vphi-\psi_\epsilon)\le \frac{1}{V_\epsilon}\int_M (\vphi-\psi_\epsilon)\chi_\epsilon^n+C\tau\|\vphi-\psi_\epsilon\|_{\infty}+C(\epsilon^*)(1+2\epsilon^*).
\end{eqnarray}

On the other hand, if $Ric(\omega_\vphi)\ge \frac{1}{2}\omega_\vphi$, then there is a uniform Sobolev constant for the metric $\omega_\vphi$ by Proposition \ref{prop-Sob}. So by Moser's iteration, there is a uniform constant $C$ such that:
\begin{equation}\label{eq-infbd}
-\inf_M (\vphi-\psi_{\epsilon})\le -\frac{C}{V_\epsilon}\int_M(\vphi-\psi_\epsilon)\omega_\vphi^n+C.
\end{equation}
Combing \eqref{eq-supbd}, \eqref{eq-infbd} and \eqref{eq-osc2C0}, it is easy to see that there exists a constant $C>0$ such that:
\begin{equation}
{\rm osc}(\vphi-\psi_\epsilon)\le \frac{C}{V_\epsilon} \int_M (\vphi-\psi_\epsilon)(\chi_\epsilon^n-\omega_\vphi^n)+C\tau\|\vphi-\psi_\epsilon\|_{\infty}+C.
\end{equation}
Since $\|\vphi-\psi_\epsilon\|_\infty\le \osc(\vphi-\psi_\epsilon)$, we get if $\epsilon^*\ll 1$, then:
\begin{equation}
\|\vphi-\psi_\epsilon\|_\infty\le C I_{\psi_\epsilon}(\vphi)+C(\epsilon^*, \tau)\le C (n+1) J_{\psi_\epsilon}(\vphi)+C(\epsilon^*).
\end{equation}
Now we use \eqref{eq-tproper} to get:
\begin{equation}
\cD_{B_{(\epsilon, t)}}(t, \vphi)\ge \left((C(n+1))^{-1}\delta^*-C\epsilon^*\right)\|\vphi-\psi_\epsilon\|_\infty-C.
\end{equation}
Notice $\delta^*$ does not depend on $\epsilon^*$. So if we choose $\epsilon^*\ll 1$, then the wanted inequality holds with $\hat{\delta}=(C(n+1))^{-1}\delta^*/2$.

\end{proof}

\begin{prop}\label{prop-C0est}
There exists a constant $C=C(X,t)>0$ that is independent of $\epsilon$ such that, the solution $\vphi_{(\epsilon,t)}$ to the Monge-Amp\`{e}re equation \eqref{eq-sCMA} (or equivalently the solution $u=\vphi_{(\epsilon,t)}-\psi_\epsilon$ to the equation \eqref{eq-MAref}) satisfies the following uniform $L^\infty$-estimate:
\begin{equation}
\|u\|_{\infty}=\|\vphi_{(\epsilon,t)}-\psi_\epsilon\|_{\infty}\le C.
\end{equation}
\end{prop}
\begin{proof}
It's known that $\vphi_{(\epsilon,t)}$ is the minimizer of $\cD_{B(\epsilon,t)}(t,\vphi)$ among all $\vphi\in \cE^1$. In particular,
\[
\cD_{B_{(\epsilon,t)}}(t, \vphi_{(\epsilon,t)})\le \cD_{B_{(\epsilon,t)}}(t, \psi_\epsilon).
\]
We also know that $\omega_{(\epsilon,t)}=\chi_\epsilon+\sddb\vphi_{(\epsilon,t)}$ satisfies $Ric(\omega_{(\epsilon,t)})\ge  \frac{1}{2}\omega_{(\epsilon,t)}$ if $t\ge 1/2$. So by Proposition \ref{prop-Ricproper}, we just need to verify that $\cD_{B_{(\epsilon,t)}}(t,\psi_\epsilon)$ is uniformly bounded from above. This indeed holds (see \eqref{eq-MAref}-\eqref{eq-gamma}):
\begin{eqnarray*}
\cD_{B_{(\epsilon,t)}}(t,\psi_\epsilon)&=&-\log\left(\frac{1}{V}\int_M \frac{e^{-t\psi_\epsilon}}{|s_H|^{2(1-t)/m}\prod_i |s_i|^{2(b_i+t\epsilon\theta_i+(1-t)\theta_i)}}\right)\\
&=&-\log\left(\frac{1}{V}\int_M \Omega(\epsilon,t)\right)=-t\gamma(\epsilon,t)\le C.
\end{eqnarray*}
\end{proof}

To state the next result, we first need to modify $\chi_0$. Let $H'=H+m \sum_i \theta_i E_i$ is a holomorphic section of $ m(L_0-\sum_i \theta_i E_i)+m \sum_i \theta_i E_i=m \mu^*(K_X^{-1})$ which descends to a holomorphic section of $mK_X^{-1}$ which will be denoted by $H_X$.
\[
\hat{\chi}_0=\chi_0+c\sddb \|s_{H'}\|_{\FS}^{2\left(1-\frac{1-t}{m}\right)}.
\]
Then for $0<c\ll 1$, $\hat{\chi}_0$ is a conical K\"{a}hler metric on $(M\setminus E, (1-t)m^{-1}H)$. The bisectional curvature of $\hat{\chi}_0$ is bounded from above because
$\hat{\chi}_0$ is the pull-back of the $\omega_{\FS}+c\sddb \|s_{H'}\|_{\FS}^{2(1-(1-t)m^{-1})}$ on $\bP^N$ and the latter has the bisectional curvature bounded from above (see \cite[Appendix A]{JMR16}). We also modify $\chi_\epsilon$ to:
\begin{equation}
\hat{\chi}_\epsilon=\hat{\chi}_0-\epsilon \eta=\chi_0-\epsilon\eta+c\sddb \|s_{H'}\|_{\FS}^{2\left(1-\frac{1-t}{m}\right)}.
\end{equation}
Then again for $0<c\ll 1$, $\hat{\chi}_\epsilon$ is a conical K\"{a}hler metric on $(M, (1-t)m^{-1}H)$. It's easy to check that there exists $C>0$, which is independent of $\epsilon$, such that
$\hat{\chi}_0\le C\hat{\chi}_\epsilon$ as long as $c$ and $\epsilon$ are sufficiently small.

\begin{prop}\label{prop-C2lb}
There exists $C=C(X,V,t)>0$, independent of $\epsilon$, such that $\hat{\chi}_0\le C \omega_{(\epsilon,t)}$. As a consequence, there exists a constant $C>0$ independent of $\epsilon$, such that $\chi_0\le C\omega_{(\epsilon,t)}$.
\end{prop}
\begin{proof}
Let $f: M\rightarrow \cP^N$ be a holomorphic morphism induced given by $L_0=\mu^*(-K_X)$. For simplicity, denote $\omega=\omega_{(\epsilon,t)}$. Then $Ric(\omega)=t\omega$. By using Chern-Lu's inequality we have:
\begin{equation}
\Delta_{\omega} \log \tr_{\omega}(\hat{\chi}_0)=t- C_1 \tr_{\omega}(\hat{\chi}_0),
\end{equation}
where $C_1$ can be chosen to be the upper bound of $\hat{\chi}_0$ on $(M\setminus E)$.
Because $\omega=\hat{\chi}_0-\epsilon\eta+\sddb \hat{u}$ and $\hat{\chi}_0\le C_2 \hat{\chi}_\epsilon$ for some $C_2>0$ independent of $\epsilon$, we get:
\begin{equation}
n=\tr_{\omega}(\hat{\chi}_\epsilon)+\Delta \hat{u}\ge C_2 \tr_{\omega}(\hat{\chi}_0)+\Delta \hat{u}.
\end{equation}
Combining the above identities, we get:
\begin{equation}
\Delta_{\omega}\left(\log\tr_{\omega}(\hat{\chi}_0)-\lambda \hat{u}\right)\ge (\lambda C_2-C_1) \tr_{\omega}(\hat{\chi}_0)+t-\lambda n.
\end{equation}
By choosing $\lambda\gg 1$, we can assume $\lambda C_2-C_1=:C_3>0$. So at the maximum point $p$ of $(\log\tr_{\omega}(\hat{\chi}_0)-\lambda u)$, we have:
\[
\tr_{\omega}(\hat{\chi}_0)(p)\le \frac{\lambda n-t}{C_3}\le C_4.
\]
So for any $x\in M$, we have:
\[
\tr_{\omega}(\hat{\chi}_0)(x)\le tr_{\omega}(\hat{\chi}_0)(p) e^{\lambda (\hat{u}(x)-\hat{u}(p))} \le C_4 e^{\osc(\hat{u})}.
\]
The right-hand-side of the above is uniformly bounded by Proposition \ref{prop-C0est}.

The last statement follows from the inequality $\chi_0\le C \hat{\chi}_0$ for a constant $C>0$.

\end{proof}

\begin{cor}
For any relatively compact open subset $V\Subset (M\setminus E)$, there exists a constant $C=C(X,V,t)$ independent of $\epsilon$ such that , there exists a constant $C=C(X,V,t)$ such that
\begin{equation}
C^{-1} \hat{\chi}_0 \le \omega_{(\epsilon,t)}\le C \hat{\chi}_0.
\end{equation}
\end{cor}
\begin{proof}
Recall that $u=\vphi-\psi_\epsilon$ satisfies the equation:
\begin{eqnarray}
(\chi_\epsilon+\sddb u)^n&=&e^{-tu}\frac{\Omega}{\|s_H\|_{m\psi_1}^{2(1-t)/m}\|s_\theta\|_{\kappa_\theta}^{2(t \epsilon+(1-t)\theta_i)}\prod_{i}\|s_i\|_{\kappa_i}^{2b_i}},
\end{eqnarray}
\end{proof}
where $\Omega$ is smooth volume form on $M$. Because $\hat{u}=u-c \|s_H\|^{2(1-(1-t)m^{-1})}=:u-f$, we know that $\hat{u}$ satisfies the following equation:
\begin{equation}\label{eq-MAhatu}
\omega_{(\epsilon,t)}^n=(\hat{\chi}_\epsilon+\sddb \hat{u})^n=e^{-t\hat{u}}\frac{\Omega'}{\|s_{H}\|^{\frac{2(1-t)}{m}}\prod_i \|s_i\|_{\kappa_i}^{2(b_i+t\epsilon\theta_i+(1-t)\theta_i)}}
\end{equation}
where $\Omega'=\Omega=e^{-tf}\Omega$ is a non-degenerate volume form. On the other hand, because $\hat{\chi}_0$ has is a conical K\"{a}hler metric on $M\setminus E$ with cone angle
$2\pi (1-\frac{1-t}{m})$ along $H$, we have the identity:
\begin{equation}\label{eq-hchivol}
\hat{\chi}_0^n=\frac{\Omega''}{\|s_{H}\|^{\frac{2(1-t)}{m}}}
\end{equation}
where $\Omega''$ is a non-degenerate volume form on $M\setminus E$. Comparing \eqref{eq-MAhatu} and \eqref{eq-hchivol}, we see that over $V\Subset M\setminus E$, the ratio
$
\omega_{(\epsilon,t)}^n/\hat{\chi}_0^n
$
is uniformly bounded on $V$. The statement now follows from this together with the estimate in Proposition \ref{prop-C2lb}.

With the above local $C^2$ estimate, we can derive the $C^{2,\alpha}$ estimate (in the sense of \cite{Don12a,JMR16}) for conical K\"{a}hler-Einstein metrics following the arguments in \cite{Tia17, She15}.
\begin{prop}[see \cite{Tia17,She15}]
For any relatively compact open set $V\Subset (M\setminus E)$, there exists a constant $C=C(M,V,t)$ such that, for any $\alpha<\left(1-\frac{1-t}{m}\right)^{-1}-1$,
\begin{equation}
\|\omega_{(\epsilon,t)}\|_{C^{\alpha}}\le C.
\end{equation}
\end{prop}
Letting $\epsilon\rightarrow 0$, we can use the above uniform estimates to take limit of $u_{(t,\epsilon)}$ and obtain:
\begin{thm}\label{thm-mainweakKE}
Let $X$ be an admissible $\bQ$-Fano variety and use the above notations.  Then the following statements are true.

{\rm (1)}
If $X$ is uniformly K-stable, then $X$ admits a weak K\"{a}hler-Einstein metric $\omega_{(0,1)}$ which is smooth on $X^{\reg}$.

{\rm (2)}
Assume that $X$ is K-semistable. Let $\mu: M\rightarrow X$ be an admissible resolution. Choose $m\gg 1$ sufficiently divisible and let $H\in \left|m(\mu^*K_X^{-1}-\sum_i \theta_i E_i)\right|$ be a smooth divisor such that $H+\sum_i E_i$ is simple normal crossing. Put
\begin{equation}\label{eq-HX}
H'\,=\,H\,+\,m\,\sum_i \,\theta_i E_i\,=\,\mu^* H_X.
\end{equation}
Then for any $t\in (0,1)$, $(X, \frac{1-t}{m} H_X)$ has a weak conical K\"{a}hler-Einstein metric $\omega_{(0,t)}$. Moreover, $\omega_{(0,t)}$ is a strong conical K\"{a}hler-Einstein metric on $(X^{\reg}, \frac{1-t}{m} \left.H_X\right|_{X^{\rm reg}})$.

\end{thm}

\section{Conical metric structure on $X$}\label{sec-conicX}

\subsection{Gromov-Hausdorff compactness and gauge fixing}

In this section, we assume $X$ is K-semistable. Fix $t\in (0,1)$. By Proposition \ref{prop-semiexist}, $(M, B_{(\epsilon,t)})$ admits a strong conical K\"{a}hler-Einstein metric $\omega_{(\epsilon,t)}$ for $\epsilon$ sufficiently small.

\begin{prop}[Diameter bound]
For any $t\in (0,1)$, $\diam(M, \omega_{(\epsilon,t)})\le C$.
\end{prop}
\begin{proof}
By Theorem B in \cite{GP16}, $\omega_{(\epsilon,t)}$ is a strong conical K\"{a}hler-Einstein metric.
The regular part of $(M, \omega_{(\epsilon,t)})$, whose metric completion is $M$, is geodesically convex. The usual proof of Myer's theorem applies.
\end{proof}

As $\epsilon_i\rightarrow 0$, by possibly taking a subsequence, $(M, \omega_{(\epsilon_i, t)})$ converges to a compact metric space, which will be denoted by $(X_{(0,t)}, d_{(0,t)})$ or simplify by $(X_t, d_t)$, in the Gromov-Hausdorff topology. Moreover, we have:
\begin{enumerate}
\item
$M\setminus E$, $\omega_{(\epsilon_j,t)}$ converges in $C^{2,\alpha,\beta}$ norm to $\omega_{(0,t)}$ and $\omega_{(0,t)}$ is a strong conical K\"{a}hler-Einstein metric on $(M\setminus E, (1-t)m^{-1}H)$. We will denote by $g_{(0,t)}$ the associated metric tensor.
\item
 On the open set $M\setminus (E\cup H)$, $\omega_{(\epsilon_j,t)}$ converges to $\omega_{(0,t)}$ smoothly and $\omega_{(0,t)}$ is a smooth K\"{a}hler-Einstein metric.
\end{enumerate}

Let $\cS$ denote the set of points where at least one metric tangent cone is not $\bR^n$ and $\cS_k$ denote the set of points where no tangent cone splits $\bR^{k+1}$.
Denote by $\cR=X_t\setminus \cS$ the set of regular points.
By Theorem 1.6 in \cite{TW18} of the extension of Cheeger-Colding-Tian's theory to the conical K\"{a}hler
setting, we can get
\begin{enumerate}
\item $\dim_{\bR}(\cS)\le 2n-2$;
\item $\cS_{2k-1}=\cS_{2k}$.
\end{enumerate}

Define the set:
\[
S_B=\left\{x\in X_t; \text{ there exist } \{y_j\}_{j=1}^{+\infty}\subset \supp(B) \text{ with } y_j\stackrel{d_{\rm GH}}{\longrightarrow} x, \text{ as } j\rightarrow+\infty\right\}.
\]

\[
S_E=\left\{x\in X_t; \text{ there exist } \{y_j\}_{j=1}^{+\infty}\subset E \text{ with } y_j\stackrel{d_{\rm GH}}{\longrightarrow} p, \text{ as } j\rightarrow +\infty\right\}.
\]

\[
S_H=\left\{x\in X_t; \text{ there exists } \{y_j\}_{j=1}^{+\infty}\subset H \text{ with } y_j\stackrel{d_{\rm GH}}{\longrightarrow} p, \text{ as } j\rightarrow +\infty\right\}.
\]
The following almost gauge fixing theorem can be proved in the same way as the work of Rong-Zhang:
\begin{prop}[{\cite[Theorem 4.1]{RZ11}}]\label{prop-agf}
There is a continuous surjection $$f_t: \overline{(M\setminus \supp(B), d_{g_{(0,t)}})}\rightarrow (X_t, d_t)$$ such that $f_t: (M\setminus \supp(B), d_{g_{(0,t)}})\rightarrow (X_t\setminus S_B, d_t)$ is a homeomorphism and a local isometry. In other words, for any $y\in M\setminus \supp (B)$, there is an open neighborhood of $y$, $U\subset M\setminus \supp(B)$, such that $f_t: (U, d_{g_{(0,t)}}|_U)\rightarrow (f_t(U), d_t|_{f(U)})$ is an isometry.
\end{prop}

Because $\omega_{(\epsilon_j,t)}$ converges to a smooth K\"{a}hler-Einstein metric $\omega_{(0,t)}$ on $M\setminus \supp(B)=M\setminus (E\cup H)$, we have $X_t\setminus S_B \subseteq X_t \setminus \cS$, or equivalently $S_B\supseteq \cS$. On the other hand,  we have: 
\begin{prop}\label{prop-cSSB}
$S_B \subseteq \cS$. As a consequence, $\cS=S_B$.
\end{prop}
\begin{proof}
Prove by contradiction. Suppose there exist $x\in \cS_B\cap \cR$ and $y_j\in B\subset (M, \omega_{(\epsilon_j,t)})$ converging to $x\in (X_t,d_t)$. Then for any sufficiently small $\epsilon>0$, there exists $r_0>0$ sufficiently small such that:
\[
\Vol(B_{d_t}(x, r_0))>(1-\epsilon) \Vol(B_{Euc}(0,r_0)).
\]
By Colding's volume convergence result, if $j$ is sufficiently large, then we have:
\begin{equation}\label{eq-volBzj}
\Vol(B_{g_j}(y_j, r))\ge \Vol(B_{Euc}(0,r))(1-2\epsilon).
\end{equation}
On the other hand, if $y_j\in \supp(B)$, then by the Bishop-Gromov volume comparison, for any $r>0$ we have:
\[
\Vol(B_{g_j}(y_j, r))\le \Vol(B_{Euc}(0,r)) \beta_B,
\]
where $\beta_B=\max\{1-(1-t)m^{-1}, 1-b_i-t\epsilon \theta_i-(1-t)\theta_i\}<1$ since $t\in (0,1)$ and $\theta_i>0$. This contradicts \eqref{eq-volBzj} if $\epsilon$ is sufficiently small.

\end{proof}

\begin{prop}
$(X_t, d_t)$ is isometric to the metric completion $\overline{(M\setminus B_{\red}, d_{g_{(0,t)}})}$.
\end{prop}

\begin{proof}
We have the following identities:
\begin{eqnarray*}
\overline{(M\setminus \supp(B), d_{g_{(0,t)}})} &\stackrel{\rm isom}{\cong}& \overline{(f_t(M\setminus \supp(B)), d_t|_{f_t(M)})} \quad (\text{Proposition } \ref{prop-agf})\\
&=&\overline{(X_t\setminus S_B, d_t)}\quad (\text{Proposition } \ref{prop-agf}) \\
&=&\overline{(X_t\setminus \cS, d_t)} \quad (\text{Proposition } \ref{prop-cSSB})   \\
&=&\overline{(\cR, d_t)}=(X_t, d_t).
\end{eqnarray*}

\end{proof}

By arguing as in \cite{NTZ15}, we get a gauge fixing theorem
\begin{prop}[see {\cite[Lemma 3.10]{NTZ15}}]\label{prop-gf}
The identity map ${\rm id}: (M\setminus \supp(B), d_{g_{(0,t)}}) \rightarrow (M, \omega_{(\epsilon_j,t)})$ gives a Gromov-Hausdorff approximation representing the convergence $(M, \omega_{(\epsilon_j,t)})\rightarrow (X_{t}, d_{t})$. As a consequence, the identity map ${\rm id}$ extends to an isometry denoted by $$\overline{\rm id}: \overline{(M\setminus \supp(B), d_{g_{(0,t)}})}\rightarrow (X_{(0,t)}, d_{(0,t)}). $$
Moreover, we have $\overline{{\rm id}}(M\setminus \supp(B))=\cR$.
\end{prop}
\begin{proof}
Note that $\omega_{(\epsilon_j,t)}$ converges to $\omega_{(0,t)}$ smoothly on $M\setminus \supp(B)$. So $\id: (M\setminus \supp(B), d_{g_{(0,t)}})\rightarrow (M\setminus \supp(B), \omega_{(\epsilon_j,t)})$ gives a Gromov-Hausdorff approximation. In other words, there exists a metric structure $\mathfrak{d}_{(\epsilon_j,t)}$ on $(M\setminus \supp(B))\sqcup (M\setminus \supp(B))$ such that $\mathfrak{d}_{(\epsilon_j,t)}$ restricted to the two copies gives $d_{(g_{(0,t)})}$ and $d_{\omega_{(\epsilon_j,t)}}$ respectively and the Hausdorff distance between the copies of $M\setminus \supp(B)$ converges to 0 as $\epsilon_j\rightarrow 0$.

By Proposition \ref{prop-agf} and Proposition \ref{prop-cSSB} we know that $(M\setminus \supp(B), \omega_{(0,t)})$ is isometric to $(\cR, d_t)$, whose metric completion is $(X_t, d_t)$. Under this identification of isometry, $\id$ extends to give Gromov-Hausdorff approximation representing $(M, \omega_{(\epsilon_j,t)})\rightarrow (X_t, d_t)$ satisfying $\overline{\id}(M\setminus \supp(B))=\cR$, as claimed.
\end{proof}

The following is an immediate corollary as in \cite{NTZ15}.
\begin{cor}\label{cor-gauge}
Let $L'$ be an $\bR$-line bundle over $M$ with possibly singular Hermitian metric $h'$. Assume $h'$ is a smooth Hermitian metric on $(M\setminus B, L'|_{M\setminus B})$. Then for any $k\in \bZ$, the twisted line bundle $(L'\otimes K_{M}^{-\otimes k}), h'\otimes (\omega_{(\epsilon_i, t)})^k)$ converges smoothly to a limit Hermitian line bundle
$(L'\otimes K_{\cR}^{-k}, h'\otimes (\omega_{(0,t)})^k)$ on $\cR$.
\end{cor}

\subsection{Gradient estimate of the conical K\"{a}hler-Einstein potential}

If we write $\omega_{(0,t)}=\chi_0+\sddb u_{(0,t)}$, then $u_{(0,t)}$ is a bounded function determined up to a constant. Choosing a point $p\in X^{\rm reg}\setminus H_X$, for some small positive number $r$, we have
 $$B_{\omega_{(0,t)}}(p, 2r)\subset X^{\rm reg}.$$
 Recall we have the log resolution $\mu: M\rightarrow X$.
  Putting
 $$U=\mu^{-1}\left(X\setminus B_{\omega_{(0,t)}}(p, 2r)\right),$$ by the smooth convergence in the regular part, we know that
 $$B_{\omega_{(\epsilon,t)}}(p,r)\subset U^c.$$
Let $\Delta_{\epsilon}$ denote the Laplacian operator with respect to $\omega_{(\epsilon,t)}$. We solve the following Dirichlet problem:
\begin{equation}
\left\{
\begin{array}{l}
\Delta_\epsilon v_\epsilon=-\tr_{\omega_{(\epsilon,t)}}\chi_0+n \text{ on } U; \\
v_\epsilon=u_{(0,t)} \text{ on } \partial U.
\end{array}
\right.
\end{equation}

\begin{lem}\label{lem-C0vep}
$v_\epsilon$ is uniformly bounded with respect to $\epsilon$. In other words, there exists a constant $C>0$ independent of $\epsilon>0$ such that
$|v_\epsilon|_{L^\infty}\le C$.
\end{lem}
\begin{proof}
This follows the existence of suitable barrier functions under the assumption of bounded Ricci curvature. Since there exists $R>0$ such that $U\subset A_{\omega_{(\epsilon,t)}}(p; r,R)$, we have
function $\phi(x)=\phi(r(x))$ such that $$\Delta_\epsilon \phi\geq C$$ on $A_{\omega_{(\epsilon,t)}}(p; r,R)$. From  $\tr_{\omega_\epsilon}\chi_0\le C$, we get $n-C\le \Delta_\epsilon v_\epsilon\le n$ and
$$\Delta_\epsilon (v_\epsilon+\phi)\geq 0, \Delta_\epsilon (v_\epsilon-\phi)\leq 0.$$ So we can apply the maximal principle to prove the lemma.
\end{proof}
\begin{prop}\label{gradientp}
For some constant $C$, we have $$|\nabla u_{(0,t)}|_{X^{\rm reg}}\leq C.$$
\end{prop}
\begin{proof}
In the following calculation, the operators $\Delta_\epsilon$, $\nabla_\epsilon$ and the norms are with respect to $\omega_{(\epsilon,t)}$.
By Bochner's formula and Cauchy-Schwarz inequality, we have:
\begin{eqnarray*}
\Delta_\epsilon |\nabla_\epsilon v_\epsilon|^2&=&|\nabla_\epsilon \nabla v_\epsilon|^2+|\nabla_\epsilon \overline{\nabla_\epsilon} v_\epsilon|^2+2 \Re\big((\Delta_\epsilon v_\epsilon)_i (v_\epsilon)_{\bar{i}}\big)+Ric(\omega_{(\epsilon,t)})_{i\bar{j}}(v_\epsilon)_i(v_\epsilon)_{\bar{j}}\\
&\ge& \frac{t}{2}|\nabla_\epsilon v_\epsilon|^2-\frac{2}{t}|\nabla_\epsilon (\Delta_\epsilon v_\epsilon)|^2.
\end{eqnarray*}
Note that $Ric(\omega_{(\epsilon,t)}) \ge t\omega_{(\epsilon,t)}$ on $M$.
On the other hand, using Chern-Lu's inequality and the $C^2$-estimate in Proposition \ref{prop-C2lb}, we have:
\begin{eqnarray*}
\Delta_\epsilon (\tr_{\omega_\epsilon}\chi_0) &\ge & \frac{|\nabla_\epsilon (\tr_{\omega_\epsilon}\chi_0)|^2}{\tr_{\omega_\epsilon}\chi_0}+t\cdot  \tr_{\omega_\epsilon}(\chi_0)-C_1 \tr_{\omega_\epsilon}(\chi_0)^2\\
&=& C_2 |\nabla_\epsilon (\tr_{\omega_\epsilon}\chi_0)|^2- C_3=C_2 |\nabla_\epsilon (\Delta_\epsilon v_\epsilon)|^2-C_3.
\end{eqnarray*}
So we get the inequality:
\begin{eqnarray*}
\Delta_\epsilon (|\nabla_\epsilon v_\epsilon|^2)&\ge& \frac{t}{2}|\nabla_\epsilon v_\epsilon|^2-C\Delta_\epsilon (\tr_{\omega_\epsilon}\chi_0)-C\\
\end{eqnarray*}
If we let
\begin{equation}\label{eq-auxf}
f=|\nabla_\epsilon v_\epsilon|^2+C \tr_{\omega_\epsilon}\chi_0,
\end{equation}
since $\tr_{\omega_\epsilon}\chi_0$ is bounded, we get:
\begin{equation}\label{equ}
\Delta_\epsilon f\ge \frac{t}{2} f-C.
\end{equation}
At any boundary point $q\in \partial U$, we can choose a neighborhood $\Omega$ of $q$ such that $$\omega_{(0,t)}\leq C\omega_{(\epsilon,t)}\text{ in } \Omega$$

On the other hand, we have
$$ \left\{
\begin{aligned}
\Delta_\epsilon (v_\epsilon-u_{(0,t)})& =  h& \text{ in } U\bigcap \Omega\\
 (v_\epsilon-u_{(0,t)})|_{T} & =  0, &
\end{aligned}
\right.
$$
where $T=\partial U\bigcap \Omega$ and
$
h=-\tr_{\omega_{(\epsilon,t)}}\chi_0+n-\tr_{\omega_{(\epsilon,t)}}(\omega_{(0,t)}-\chi_0)
$
is uniformly bounded in $\Omega$. By the boundary estimate (Theorem 4.16 in [GT]), we know that $$|v_\epsilon-u|_{C^{1,\alpha}(\Omega')}\leq C$$ for any $\Omega'\subset\subset (\Omega\cap U)\bigcup T$. So we get
$|\nabla_\epsilon v_\epsilon|_{\partial U}\leq C$ and hence $f$ in \eqref{eq-auxf} satisfies $\left.f\right|_{\partial U}\le C$ with $C$ independent of $\epsilon$.
From (\ref{equ}), by Maximum principle, we get
\begin{equation}\label{eq-gradvep}
|\nabla_\epsilon v_\epsilon|_{U}\leq C.
\end{equation}
As $\epsilon\rightarrow 0$, by shrinking $U$ slightly, we can assume that $(U, \omega_{(\epsilon,t)})$ converges to an open set $\widehat{U}=X_{(0,t)}\setminus B_{\omega_{(0,t)}}(p, 2r)$ in $(X_{(0,t)},d_{(0,t)})$ in the Gromov-Hausdorff topology. Because of the uniform estimate \eqref{eq-gradvep}, the Lipschitz function $v_\epsilon$ converges to $v$, which satisfies
$$ \left\{
\begin{aligned}
\Delta v& =  -\tr_{\omega_{(0,t)}}\chi_0+n\text{ in } \widehat{U}\bigcap \mathcal{R}\\
 v  & =  u_{(0,t)} \text{ on } \partial \widehat{U}.
\end{aligned}
\right.
$$
Moreover, by Lemma \ref{lem-C0vep} and the uniform estimate \eqref{eq-gradvep}, we get the estimate:
\begin{equation}
\|v\|_{L^\infty(\widehat{U})}< \infty, \quad
\|\nabla v\|_{L^\infty(\widehat{U}\cap X^\reg)}<+\infty
\end{equation}
We now claim that $u_{(0,t)}=v$. If this is true, then the proposition is proved since $v$ is smooth in $B_{\omega_{(0,t)}}(p, 2r)$.

To verify the claim, take a cut-off function $\rho_\delta$ for the singular set $\cS\subset X_{(0,t)}$ and compute
\begin{align}
0&=\int_{\widehat{U}}{\rm div}\left(\rho_\delta^2(u_{(0,t)}-v)\nabla(u-v)\right)dv\\
&=2\int_{\widehat{U}}(u_{(0,t)}-v)\rho_\delta\langle \nabla \rho_\delta, \nabla (u_{(0,t)}-v)\rangle+\int_{\widehat{U}}\rho_\delta^2|\nabla (u_{(0,t)}-v)|^2\notag\\
&\geq -\frac{1}{2}\int_{\widehat{U}}\rho_\delta^2|\nabla(u_{(0,t)}-v)|^2-C\int_{\widehat{U}}|\nabla \rho_\delta|^2+\int_{\widehat{U}}\rho_\delta^2|\nabla (u_{(0,t)}-v)|^2 \notag\\
&=\frac{1}{2}\int_{\widehat{U}}\rho_\delta^2|\nabla(u_{(0,t)}-v)|^2-C\int_{\widehat{U}}|\nabla \rho_\delta|^2\notag.
\end{align}
Taking $\delta\rightarrow 0$, we get:
$$\int_{\widehat{U}\cap \cR}|\nabla (u_{(0,t)}-v)|^2=0.$$
So we indeed have $u_{(0,t)}=v$.
\end{proof}

\subsection{A priori estimates for holomorphic sections}

In this section, we prove some $L^\infty$ estimate and gradient estimate for holomorphic sections of $k L:=k (\mu^*K_X^{-1})$ for $k$ sufficiently divisible.
We will consider two Hermitian metrics on $k L$. The first one is a singular Hermitian metric constructed using $h_{(\epsilon,t)}$ on $L_\epsilon$. The second one is (up to a scaling) the pull back of the Fubini-Study metric under the composition $M\rightarrow X\rightarrow \bP^{N_m}$ which is morphism associated to the line bundle $m L=\mu^*(-mK_X)$.

We first consider the singular Hermitian metric. Write:
\begin{eqnarray*}
\mu^*K_{X}^{-k}&=&k(\mu^*K_X^{-1}- \sum_i \epsilon \theta_i E_i)+\sum_i k\epsilon \theta_i E_i.
\end{eqnarray*}

Recall that $h_{(\epsilon,t)}=e^{-\vphi_{(\epsilon,t)}}$ is the Hermitian metric on $L_\epsilon=L-\sum_i \theta_i E_i$ whose Chern curvature is equal to $\omega_{(\epsilon,t)}$.

{\bf Notation:} For all the discussions in this section, $t\in (0,1)$ is fixed. So for simplicity, we will just write the subscript $\epsilon$ for $(\epsilon,t)$.

We define the following singular Hermitian on $L:=\mu^*K_X^{-1}$:
\[
\hat{h}_\epsilon=e^{-\hvphi_\epsilon}:=h_\epsilon\frac{1}{\prod_i |s_i|^{2\epsilon \theta_i}}=\frac{e^{-\vphi_\epsilon}}{\prod_i |s_i|^{2\epsilon\theta_i}}
\]
Then the Chern curvature current $\Theta(\hat{h}_\epsilon)$ of $\hat{h}_\epsilon$ satisfies:
\begin{eqnarray}\label{eq-ThetaRic}
\Theta(\hat{h}_\epsilon^k)+Ric(\omega_\epsilon)&=&-\sddb \log \hat{h}_\epsilon^k+Ric(\omega_\epsilon)\nonumber\\
&=& k\omega_\epsilon+\sum_i k\epsilon \theta_i \{E_i\}+t \omega_\epsilon+\frac{1-t}{m}\{H\}+\sum_i (b_i+t\epsilon \theta_i+(1-t)\theta_i) \{E_i\}\nonumber\\
&\ge & (k+t)\omega_\epsilon.
\end{eqnarray}

Using the Weitzenboch formula \eqref{eq-Wentz}, it is easy to get:
\begin{lem}
Assume $k\epsilon \theta_i+b_i+t\epsilon \theta_i+(1-t)\theta_i<1$. Then for $\xi\in \Gamma(M, T^{0,1}M\otimes L^k)$, we have the inequality:
\begin{equation}
\int_M (|\bar{\partial}\xi|^2+|\bar{\partial}^*\xi|^2)\omega_\epsilon^n\ge (k+t)\int_M |\xi|^2\omega_\epsilon^n.
\end{equation}
\end{lem}
Under the assumption $k\epsilon\theta_i+b_i+t\epsilon \theta_i+(1-t)\theta_i<1$, the singular Hermitian metric on $kL-K_M$ given by:
\[
e^{-k\hvphi_\epsilon}\omega_\epsilon^n= \frac{e^{-(k+t)\vphi_\epsilon}}{|s_H|^{2\frac{1-t}{m}}\prod_i |s_i|^{2(k\epsilon\theta_i+b_i+t\epsilon \theta_i+(1-t)\theta_i)}}
\]
has a trivial multiplier ideal sheaf and its curvature is a K\"{a}hler current on $M$ by \eqref{eq-ThetaRic}. So we get the solvability of $\bar{\partial}$-equation with an $L^2$-estimate:
\begin{prop}[$L^2$-estimate]\label{prop-L2est}
Assume $k\epsilon \theta_i+b_i+t\epsilon \theta_i+(1-t)\theta_i<1$. Then there exists a constant $C$ independent of $\epsilon$, such that for $\xi\in \Gamma(T^{0,1}M\otimes L^k)$ with $\bar{\partial}\xi=0$, we can find a solution
$\bar{\partial}\zeta=\xi$ which satisfies:
\begin{equation}
\int_M |\zeta|_{\hat{h}_\epsilon^k}^2 \omega_\epsilon^n\le \frac{C}{k}\int_M |\xi|^2_{\hat{h}_\epsilon^k\otimes \omega_\epsilon}\omega_\epsilon^n.
\end{equation}
\end{prop}


The following convergence follows from local elliptic estimates for holomorphic sections and the convergence with a fixed gauge (see Corollary \ref{cor-gauge}).
\begin{prop}
Assume $k\epsilon \theta_i+b_i+t\epsilon \theta_i+(1-t)\theta_i<1$.
Let $\zeta_j$ be a sequence of holomorphic sections of $L^k$, $k\ge 1$, satisfying:
\[
\int_M |\zeta_j|_{\hat{h}_{\epsilon_j}^k}^2\omega_{\epsilon_j}^n\le 1.
\]
Then as $\epsilon_j\rightarrow 0$, by possibly passing to a subsequence if necessary, $\zeta_j$ converges to a locally bounded holomorphic section
$\zeta_\infty$ of $L^k$ over $\cR=M\setminus \supp(B)$.
\end{prop}

A priori, it is not clear that $|\zeta_\infty|_{h_{(0,t)}^k}$ is globally bounded. In the following, we also need the boundedness of $|\nabla^{h^k_{(0,t)}} \zeta_\infty|$. To prove these boundedness,
we follow the approach of \cite{NTZ15} to replace $\hat{h}_\epsilon$ by $h_{\FS}$ and consider the norms $|\zeta|_{h_\FS}^2$ and $|\nabla^{h^k_\FS} \zeta|_{h_{\FS}^k\otimes \omega_\epsilon}$. Recall that, by Proposition \ref{prop-C2lb}, there exists $C$ independent of $\epsilon$, such that the curvature of $h_{\FS}^k$ satisfies:
\begin{equation}\label{eq-Thetaomega}
\Theta=\Theta(h_\FS^k)=k \omega_{\FS}=k\chi_0\le C k \omega_\epsilon.
\end{equation}
The following proposition is motivated by \cite[Proposition3.17]{NTZ15}.
\begin{prop}\label{prop-FSest}
There exists a constant $C>0$ independent of $\epsilon$, such that
for any $\zeta\in H^0(M, L^k)$, we have the following $L^\infty$ and gradient estimates:
\begin{equation}\label{eq-C0FS}
\sup_M |\zeta|^2_{h^k_\FS}\le C\; k^n \int_M |\zeta|^2_{h^k_\FS}\omega_\epsilon^n;
\end{equation}
\begin{equation}\label{eq-C1FS}
|\nabla^{h_\FS}\zeta|^2_{h^k_\FS\otimes\omega_\epsilon}\le C\; k^{n+1}\int_M |\zeta|^2_{h^k_\FS}\omega_\epsilon^n.
\end{equation}

\end{prop}
\begin{proof}
For simplicity, we will denote by $|\zeta|^2$ (resp. $|\nabla\zeta|^2$) the norm $|\zeta|_{h_\FS}^2$ (resp. $|\nabla^{h^k_\FS} \zeta|_{h_{\FS}^k\otimes \omega_\epsilon}$).

Substituting $h=\hat{h}_\FS^k$ and $\omega=\omega_\epsilon$ into \eqref{eq-lapnorm}, we get:
\begin{equation}
\Delta_\epsilon|\zeta|^2=|\nabla \zeta|^2-k |\zeta|^2\cdot (\tr_{\omega_\epsilon}\chi_0)\ge -k C |\zeta|^2,
\end{equation}
where we used $\tr_{\omega_\epsilon}\chi_0\le C$. Because we have the uniform Sobolev constant by Proposition \ref{prop-Sob}, the estimate in \eqref{eq-C0FS} follows from the standard Moser iteration.

Substituting $h=\hat{h}_\FS^k$ and $\omega=\omega_\epsilon$ into \eqref{eq-lapgrad}, we get:
\begin{eqnarray}
\Delta_\omega |\nabla \zeta|^2&=&|\nabla \nabla \zeta|^2+|\bar{\nabla}\nabla \zeta|^2-\left[({\rm tr}_{\omega}\Theta)_i \zeta\overline{\zeta_{,i}}+c.c.\right]\nonumber\\
&&\hskip 5mm +R_{i\bar{j}}\zeta_{,i}\overline{\zeta_{,j}}-2\Theta_{i\bar{j}}\zeta_{,i}\overline{\zeta_{,j}}-|\nabla \zeta|^2 {\rm tr}_{\omega}(\Theta)\nonumber\\
&=&|\nabla \nabla \zeta|^2+|\bar{\nabla}\nabla \zeta|^2-k\left[(\tr_{\omega_\epsilon}\chi_0)_i\zeta\overline{\zeta_i}+c.c.\right]\nonumber\\
&&\hskip 5mm+ t|\nabla\zeta|^2-2 k (\chi_0)_{i\bar{j}} \zeta_i\overline{\zeta_j}-k |\nabla \zeta|^2 \tr_{\omega_\epsilon}(\chi_0)\nonumber\\
&\ge&|\nabla\nabla \zeta|^2+|\bar{\nabla}\nabla \zeta|^2-k \left[(\tr_{\omega_\epsilon}\chi_0)_i \zeta\overline{\zeta_i}+c.c.\right]-Ck |\nabla\zeta|^2.
\end{eqnarray}
The last inequality is because of \eqref{eq-Thetaomega}. Choose a cut-off function $\rho=\rho_\delta$ such that $\|\nabla \rho\|_{L^2(\omega_\epsilon)}\rightarrow 0$ as $\delta\rightarrow 0$. Then we can calculate:
\begin{eqnarray*}
\int_M \rho^2 \left|\nabla |\nabla \zeta|^p\right|^2&=&\frac{p^2}{4(p-1)}\int_M \rho^2 \nabla_i (|\nabla\zeta|^{2(p-1)}) \nabla_{\bar{i}}|\nabla \zeta|^2\\
&=&\frac{p^2}{4(p-1)}\int_M \rho^2 |\nabla \zeta|^{2(p-1)}(-\Delta) |\nabla \zeta|^2-2\rho |\nabla \zeta|^{2(p-1)}\nabla_i \rho \nabla_{\bar{i}}|\nabla \zeta|^2\\
&\le &\frac{p^2}{4(p-1)} \int_M \left[-\rho^2(|\nabla\nabla \zeta|^2+|\bar{\nabla}\nabla \zeta|^2)+k\rho^2 \left((\tr_{\omega}\chi_0)_i \zeta\overline{\zeta_i}+c.c.\right)\right]|\nabla \zeta|^{2(p-1)}\\
&& +Ck\rho^2 |\nabla \zeta|^{2p}-2\rho |\nabla \zeta|^{2(p-1)}\nabla_i \rho \nabla_{\bar{i}}|\nabla \zeta|^2.
\end{eqnarray*}
We estimate as follows:
\begin{eqnarray*}
&&k\int_M \rho^2 (\tr_{\omega}\chi_0)_i\zeta\overline{\zeta_i}|\nabla \zeta|^{2(p-1)}\\
&=&-k\int_M \rho^2 (\tr_{\omega}\chi_0)\left[ |\nabla \zeta|^{2p}+\zeta \overline{\zeta_{i\bar{i}}}|\nabla \zeta|^{2(p-1)}+(p-1)\zeta\overline{\zeta_i}|\nabla \zeta|^{2(p-2)}\nabla_i |\nabla\zeta|^2 \right]\\
&&-Ck \int_M 2\rho (\tr_{\omega}\chi_0) (\nabla_i \rho) \zeta \overline{\zeta_i} |\nabla \zeta|^{2(p-1)}\\
&\le&Ck\int_M \rho^2 |\nabla \zeta|^{2p}+ C(p-1)^2 k^2\int_M \rho^2 |\zeta|^2 |\nabla \zeta|^{2(p-1)}+\frac{1}{8}\int_M \rho^2 |\nabla\zeta|^{2(p-1)} (|\nabla\nabla \zeta|^2+|\bar{\nabla}\nabla \zeta|^2)\\
&&+ Ck \int \left(|\nabla \rho|^2 |\zeta|^2 |\nabla \zeta|^{2(p-1)}+\rho^2 |\nabla \zeta|^{2p}\right).
\end{eqnarray*}
Moreover, we have:
\begin{eqnarray*}
-2\int_M \rho |\nabla \zeta|^{2(p-1)}\nabla_i \rho\nabla_{\bar{i}}|\nabla \zeta|^2
&\le& \frac{1}{4}\int_M \rho^2 |\nabla \zeta|^{2(p-1)}(|\bar{\nabla}\nabla \zeta|^2+|\nabla\nabla \zeta|^2)+
\int_M |\nabla \rho|^2|\nabla \zeta|^{2p}.
\end{eqnarray*}
Summing up the estimates we get:
\begin{eqnarray*}
\int_M \rho^2 \left|\nabla |\nabla \zeta|^p\right|^2&\le& C p^3k \int_M \left(\rho^2 |\nabla \zeta|^{2p}+k\rho^2 |\zeta|^2 |\nabla\zeta|^{2(p-1)}\right.\\
&&\hskip 2cm+\left.  |\nabla \rho|^2 |\zeta|^2 |\nabla \zeta|^{2(p-1)}+|\nabla \rho|^2 |\nabla \zeta|^{2p}
\right).
\end{eqnarray*}
Since $|\zeta|^2$ and $|\nabla \zeta|^2$ are bounded, we let $\delta\rightarrow 0$ to get:
\begin{eqnarray*}
\int_M  \left|\nabla |\nabla \zeta|^p\right|^2&\le& C p^3k \int_M \left( |\nabla \zeta|^{2p}+k |\zeta|^2 |\nabla\zeta|^{2(p-1)}\right).
\end{eqnarray*}

Applying the Sobolev inequality for $(M, \omega_{(\epsilon,t)})$, we get:
\begin{equation}
\left(\int_M (|\nabla \zeta|^p)^{2n/(n-1)}\right)^{\frac{n-1}{n}}\le  Cp^3 k\int_M \left( |\nabla \zeta|^{2p}+k |\zeta|^2 |\nabla \zeta|^{2(p-1)}\right).
\end{equation}

By H\"{o}lder's inequality, we have:
\begin{eqnarray}\label{eq-Holder106}
k\int_M |\zeta|^2 |\nabla \zeta|^{2(p-1)}&\le& k\left( \int |\nabla \zeta|^{2p}\right)^{1-\frac{1}{p}} \left(\int |\zeta|^{2p} \right)^{\frac{1}{p}}\nonumber \\
&=& \left\| |\nabla \zeta|^2\right\|_{L^p}^p \frac{k\left\| |\zeta|^2\right\|_{L^p}}{\left\| |\nabla \zeta|^2\right\|_{L^p}}.
\end{eqnarray}
So we get the estimate:
\begin{equation}\label{eq-iteration*}
\left\| |\nabla \zeta|^2\right\|_{L^{\frac{np}{n-1}}}\le (C p^3 k)^{\frac{1}{p}}\left\| |\nabla \zeta|^2\right\|_{L^p} \left(1+\frac{k\left\| |\zeta|^2\right\|_{L^p}}{\left\| |\nabla \zeta|^2\right\|_{L^p}}\right)^{\frac{1}{p}}
\end{equation}
Now the estimate \eqref{eq-C1FS} can be proved by the following iteration argument (cf. \cite[Proof of Proposition 3.18]{NTZ15}): put $p_j=2\nu^{j}$, $j\ge 0$ where $\nu=\frac{n}{n-1}$. Write $\|\cdot\|_{p}$ for $\|\cdot \|_{L^p}$ $(p\in [1, \infty])$. Then we have:

{\it Case 1:} $\||\nabla \zeta|^2\|_{p_j}\ge k \|\zeta\|_{p_j}$ for all $j\ge 0$. Then \eqref{eq-iteration*} becomes:
\begin{equation}
\||\nabla \zeta|^2\|_{p_{j+1}}\le (C k )^{\frac{1}{p_j}} p_j^{\frac{3}{p_j}} \||\nabla \zeta|^2\|_{p_j}.
\end{equation}
By standard argument of Moser iteration, we get
\begin{eqnarray}\label{eq-Moser1}
\||\nabla \zeta|^2\|_{\infty}&\le& \left(\prod_{j=0}^{+\infty} (Cp_j^3)^{\frac{1}{p_j}}\right) k^{\sum_{j=0}^{+\infty} \frac{1}{p_j}}\||\nabla \zeta|^2\|_{2}\nonumber \\
&\le& Ck^{\frac{n}{2}} \||\nabla \zeta|^2\|_{2}.
\end{eqnarray}
On the other hand, we have:
\begin{eqnarray*}
\||\nabla \zeta|^2\|_2^2=\int_M |\nabla \zeta|^{4}\le \||\nabla\zeta|^2\|_{\infty}\int_M |\nabla \zeta|^2.
\end{eqnarray*}
Using \eqref{eq-Moser1}, we get:
\begin{eqnarray*}
\||\nabla \zeta|^2\|_{L^\infty}&\le& C k^{\frac{n}{2}}\||\nabla \zeta|^2\|_{2}\\
&\le& C k^{\frac{n}{2}}\||\nabla \zeta|^2\|_\infty^{\frac{1}{2}}\||\nabla \zeta|\|_{1}^{1/2}\\
&\le&\frac{1}{2}\||\nabla \zeta|^2\|_{L^\infty}+Ck^n \||\nabla \zeta|^2\|_1,
\end{eqnarray*}
which implies $\||\nabla \zeta|^2\|_{\infty}\le Ck^n \||\nabla \zeta|^2\|_1$.
On the other hand, we have
\begin{eqnarray*}
\||\nabla \zeta|^2\|_{1}=\int_M |\nabla \zeta|^2=\int_M k|\zeta|^2 \tr_{\omega_\epsilon}\chi_0\le Ck \int_M |\zeta|^2.
\end{eqnarray*}
So we get the wanted estimate $\||\nabla \zeta|^2\|_{L^\infty}\le C k^{n+1}\||\zeta|^2\|_1$.

{\it Case 2:} There exists $j_0\ge 0$ such that $\||\nabla \zeta|^2\|_{p_j}\ge k\||\zeta|^2\|_{p_j}$ for all $j>j_0$, but
$\||\nabla \zeta|^2|\|_{p_{j_0}}< k\||\zeta|^2\|_{p_{j_0}}$. Then from \eqref{eq-Holder106}, we have:
$$
\int k |\zeta|^2|\nabla \zeta|^{2(p_{j_0}-1)}\le k^{p_{j_0}}\||\zeta|^2\|^{p_{j_0}}_{p_{j_0}}.
$$
The same iteration as case 1 shows that:
\begin{eqnarray*}
\||\nabla \zeta|^2\|_{\infty}&\le & C k^{\frac{n}{p_{j_0+1}}}\||\nabla \zeta|\|_{p_{j_0+1}}\\
&\le & C k^{\frac{n}{p_{j_0+1}}}   \left(C p_{j_0}^3 k \int_M (|\nabla \zeta|^{2p_{j_0}}+k |\zeta|^2 |\nabla \zeta|^{2(p_{j_0}-1)}\right)^{\frac{1}{p_{j_0}}}\\
&\le & Ck^{\frac{n}{p_{j_0}}} k \cdot \||\zeta|^2\|_{p_{j_0}}\le C k^{\frac{n}{p_{j_0}}+1}\||\zeta|^2\|_{\infty}^{\frac{p_{j_0}-1}{p_{j_0}}}\||\zeta|^2\|_1^{\frac{1}{p_{j_0}}}\\
&\le & C k^{\frac{n}{p_{j_0}}+1}\left(k^n\||\zeta|^2\|_1\right)^{\frac{p_{j_0}-1}{p_{j_0}}}\||\zeta|^2\|_1^{\frac{1}{p_{j_0}}}\\
&=&C k^{n+1} \||\zeta|^2\|_1.
\end{eqnarray*}

{\it Case 3:} $\||\nabla \zeta|^2\|_{p_j}\le k \||\zeta|^2\|_{p_j}$ for infinitely many $j\ge 0$. Then by letting $p_j\rightarrow +\infty$ we get:
$\||\nabla \zeta|^2\|_{L^\infty}\le k \||\zeta|^2\|_{L^\infty}\le C k^{n+1}\||\zeta|^2\|_{1}$.
\end{proof}
From Proposition \ref{prop-C0est} and Proposition \ref{gradientp}, we know that $|\zeta|_{h^k_{(0,t)}}$ and $|\nabla^{h^k_{(0,t)}} \zeta|$ are both bounded.

\subsection{Isomorphism of GH limit with $X$}\label{sec-isom}

Choose $m\in \bZ_{>0}$ such that $mL=m\mu^*(-K_X)$ is a genuine line bundle. Let $\Phi^{\ell}: M\rightarrow \cP^N$ be the morphism defined by an orthonormal basis of $(L^{m\ell}, h^{m\ell}_\FS)$ for any $\ell$ sufficiently large. For any $t\in (0,1)$ and $\epsilon\in (0, \epsilon^*(t))$, the map
$$
\Phi^{\ell}_{(\epsilon, t)}=\Phi^\ell: (M, \omega_{(\epsilon, t)})\rightarrow (\Phi(M)\cong X, \omega_\FS)
$$
is Lipschitz with a uniform Lipschitz constant. As $\epsilon\rightarrow 0$ with $t$ fixed, by taking a subsequence $\Phi^{\ell}_{(\epsilon_i,t)}$ converges to a Lipschitz map
\begin{equation}
\Phi^{\ell}_{(0,t)}=\lim_{\epsilon_i\rightarrow 0} \Phi^{\ell}_{(\epsilon_i, t)}: (X_{(0,t)}, d_{(0,t)})\rightarrow (\Phi^\ell(M)\cong X, \omega_\FS).
\end{equation}

Given the estimates in the previous subsection, we can follow similar argument in \cite{NTZ15} to prove
\begin{prop}\label{prop-homeo}
$\Phi^{\ell^*}_{(0,t)}$ is a homeomorphism for some $\ell^*>n+1$. Hence $(M, \omega_{(\epsilon,t)})$ converges to $(X, d_{(0,t)})$ (in the Gromov-Hausdorff topology) which is the metric completion of the conical K\"{a}hler-Einstein metric $(X^\reg, \omega_{(0,t)})$.
\end{prop}

For the reader's convenience, we sketch the steps of the proof.

{\bf Step 1:} Construction of local approximating holomorphic sections. Let $p\in X_{(0,t)}$ be any point. Let $r_j\rightarrow 0$ be a decreasing sequence of radius and $\cC_p=\lim_{j\rightarrow+\infty} (X_{(0,t)},
r_j^{-1}d_{(0,t)}, p)$ be a tangent cone at $p$. By Theorem 1.6 in \cite{TW18}, we have:

\noindent
{$\bf T_1.$} $\cC_p$ is smooth outside a closed subcone $\cS_p$ of complex codimension at least $1$ which is the singular set of $\cC_p$;

\noindent
{$\bf T_2.$} There is a K\"{a}hler Ricci-flat cone metric $\omega_p$ of the form $\sddb \rho^2$ on $\cC_p\setminus \cS_p$, where $\rho$ denotes the distance function from the vertex $o$ of $\cC_p$;

\noindent
{$\bf T_3.$} Denote by $L_p$ the trivial bundle $\cC_p\times \bC$ over $\cC_p$ equipped with the Hermitian metric $e^{-k \rho^2}|\cdot|^2$. The curvature of this Hermitian metric is given by $\omega_p$.

For any $\epsilon>0$, define:
\[
V(p; \epsilon)=\{y\in \cC_p; y\in B_{\epsilon^{-1}}(o, \omega_p)\setminus \overline{B_\epsilon(o, \omega_p)}, d(y, \cS_p)>\epsilon \}.
\]
For any $\epsilon>0$ and $\delta>0$, we can choose $j_0=j_0(\epsilon, \delta)$ such that $r_{j_0}\le \epsilon^2$, and for each $j\ge j_0$, there is a diffeomorphism $\phi_j: V(p; \frac{\epsilon}{4})\rightarrow X_{(0,t)}\setminus \cS$ where $\cS$ is the singular set of $X_{(0,t)}$, satisfying:

(i) $d(p, \phi_j(V(p; \epsilon))< 10 \epsilon r_j$ and $\phi_j(V(p, \epsilon))\subset B_{(1+\epsilon^{-1})r_j}(p)$;

(ii) The K\"{a}hler metric $\omega_{(0,t)}$ on $X_{(0,t)}\setminus \cS$ satisfies
\[
\|r_j^{-2}\phi_j^*\omega_{(0,t)}-\omega_p\|_{C^6(V; \frac{\epsilon}{2})}\le \delta,
\]
where the norm is defined in terms of the metric $\omega_p$.

The following lemmas are crucial for us:
\begin{lem}[{\cite[Lemma 5.7]{Tia15}}]
Given $\epsilon>0$ and any sufficiently small $\delta>0$, there are a sufficiently large $j$, a diffeomorphism $\phi_j: V(p; \frac{\epsilon}{4})\rightarrow X_{(0,t)} \setminus \cS$ with with properties above, and an isomorphism $\psi_j$ from the trivial bundle $\cC_p\times\bC$ onto $\cL^{k_j}$ (for some $k_j$ sufficiently divisible) over $V(p; \epsilon)$ commuting with $\phi_j$ and satisfying:
\begin{equation}
|\psi_j(1)|_{h_{(0,t)}}^2=e^{- \rho^2} \text{ and } \|\nabla\psi_j\|_{C^6(V(p; \epsilon))}\le \delta,
\end{equation}
where $\nabla$ denotes the covariant derivative with respect to the metric $h_{(0,t)}$ and $e^{-\rho^2}|\cdot|^2$.
\end{lem}

\begin{lem}[{\cite[Lemma 5.8]{Tia15}}]
Let $\cS_p$ be the singular set of tangent cone $\cC_p$. There is a smooth function $\gamma_{\bar{\epsilon}}$ on $B_o(\bep)\subset \cC_p$ such that the following hold:
\begin{enumerate}
\item $\gamma_\bep\equiv 1$ if ${\rm dist}(y, \cS_p)\ge \bep$;
\item $0\le \gamma_{\bep}\le 1$ and $\gamma_\bep\equiv 0$ near $\cS_p$;
\item the metrics converges smoothly in $B_o(\bep^{-1})\setminus \{\gamma_\bep\equiv 0\}$;
\item $|\nabla \gamma_\bep|\le C=C(\bep)$;
\item $\int_{B_o(\bep^{-1})}|\nabla \gamma_\bep|^2\omega_p^n\le \bep$.
\end{enumerate}
\end{lem}

Assuming the two lemmas, one can find for any $\bep$ with $0<\epsilon< \bep$ such that ${\rm Supp}(\gamma_\bep)\subset V(p; \epsilon)$. Then $\tau=\psi_j(\gamma_\bep 1)$ extends to a compactly supported smooth section $\cL^{k_j}$ on $\cR\subset X_{(0,t)}$ which satisfies that $\tau$ is holomorphic on $\phi_j(V(p; \bep))$ and
\[
\int_{X_{(0,t)}}|\bar{\partial}\tau|^2_{h_{(0,t)}^{k_j}\otimes k_j\omega_{(0,t)}} (k_j\omega_{(0,t)})^n\le \bep.
\]

{\bf Step 2:} Existence of holomorphic peak sections on $M$.

Let $p\in X_{(0,t)}$. Suppose $p_i\in M$ satisfies $p_i\stackrel{d_{\rm GH}}{\longrightarrow} p$ under the Cheeger-Gromov convergence. By the smooth convergence on the regular set $\phi_j(V(p; \epsilon))$, the approximating holomorphic section $\tau$ on $\cL^{k_j}$ is pulled back to a family of smooth section of $L^{k_j}$ on $M$, denoted $\tau_i$. By Proposition \ref{prop-gf}, this can be done via a family of smooth maps $f_i\equiv\id: \cR=X_{(0,t)}\setminus \cS=M\setminus \supp(B) \rightarrow M$ representing the Gromov-Hausdorff convergence $(M, \omega_{(\epsilon_i,t)})\rightarrow (X_{(0,t)}, d_{(0,t)})$. For simplicity, we will just denote $\omega_{\epsilon_i}=\omega_{(\epsilon_i,t)}$ and $\hat{h}_{\epsilon_i}=\hat{h}_{(\epsilon_i,t)}$ which satisfies:
\begin{equation}
{\rm Supp}(\tau_i)\subset f_i(\phi_j(V(p; \epsilon)))\subset B_{k_j\omega_{\epsilon_i}}(p_i, 2\sqrt{k_j}\bep);
\end{equation}

\begin{equation}
\left||\tau_i|^2_{\hat{h}_{\epsilon_i}^{k_j}}(x)-e^{- d^2_{k_j\omega_{\epsilon_i}}}\right|\le \bep, \text{ on } f_i(\phi_j(V(p; \bep))),
\end{equation}
and
\[
C^{-1}\le \int_M |\tau_i|^2_{\hat{h}_{\epsilon_i}^{k_j}}(k_j\omega_{\epsilon_i})^n\le C
\]
for some constant $C$ depending on the volume ratio of the tangent cone $\cC_p$ and
\[
\int_M |\bar{\partial} \tau_i|_{\hat{h}_{\epsilon_i}^{k_j}\otimes (k_j\omega_{\epsilon_i})}^2 (k_j\omega_{\epsilon_i})^n \le 2\bep
\]
for any $i$ sufficiently large. We can assume that
\[
f_i(\phi_j(V(p; 2\bep^{1/4n}))\cap B_{k_j\omega_{\epsilon_i}}(p_i, 4\bep^{1/4n})\neq \emptyset,
\]
\[
B_{k_j\omega_{\epsilon_i}}(p'_i, \frac{1}{2}\bep^{1/4n})\subset f_i(\phi_j(V(p; \bep))),
\]
for some $p'_i\in M$ with $d_{k_j\omega_{\epsilon_i}}(p'_i, p_i)\le 2\bep^{1/4n}$.

By the $L^2$-estimate in Proposition \ref{prop-L2est}, there exists a smooth section $v_i$ solving $\bar{\partial}v_i=\bar{\partial}\tau_i$ with
\begin{equation}
\int_M |v_i|_{\hat{h}_{\epsilon_i}^{k_j}}^2 (k_j\omega_{\epsilon_i})^n \le C\cdot \bep
\end{equation}
for some $C$ independent of $i$. Noticing that $v_i$ is holomorphic on $B_{k_j\omega_{\epsilon_i}}(p'_i, \frac{1}{2}\bep^{1/4n})$, by the standard elliptic estimate, we have:
\begin{equation}
|v_i|^2_{\hat{h}_{\epsilon_i}^k}(p'_i)\le C \bep^{-1/2}\int_{B_{k_j\omega_{\epsilon_i}}(p'_i, \frac{1}{2}\bep^{1/4n})}|v_i|^2_{\hat{h}_{\epsilon_i}^{k_j}}(k_j\omega_{\epsilon_i})^n\le C\bep^{1/2}.
\end{equation}
Therefore, $\sigma_i=\tau_i-v_i$ defines a holomorphic section of $L^{k_j}$.

By the elliptic estimate, we have:
\begin{equation}
|\sigma_i|^2_{\hat{h}_{\epsilon_i}^{k_j}}(p'_i) \ge e^{-k_j d^2_{k_j\omega_{\epsilon_i}}(p_i,p'_i)}-\bep-C \bep^{1/4}\ge \frac{1}{2}
\end{equation}
once $\bep$ is chosen sufficiently small, and
\begin{equation}
|\sigma_i|^2_{\hat{h}^{k_i}_{\epsilon_i}}\le C \bar{\epsilon}^{\frac{1}{2}} \text{ on } M\setminus B_{\omega_{\epsilon_i}}(p_i, 2\epsilon_i).
\end{equation}
Moreover, we have:
\begin{equation}
C^{-1}\le \int_M |\sigma_i|^2_{\hat{h}^{k_j}_{\epsilon_i}} (k_j\omega_{\epsilon_i})^n\le C
\end{equation}
and
\begin{equation}
\int_{M\setminus B_{\omega_{\epsilon_i}}(p_i, 2\bar{\epsilon})}|\sigma_i|^2_{\hat{h}^{k_j}_{\epsilon_i}} (k_j \omega_{\epsilon_i})^n\le C\cdot \bar{\epsilon}.
\end{equation}

Passing to a subsequence if necessary, the sequence of points $p'_i$ converge to a point $p'\in \cR$ with $d_t(p, p')\le 2 k_j^{-1/2}\bep^{1/4n}$, the section $\sigma_i\in H^0(M, L^k)$ converges
to a holomorphic section $\sigma_\infty\in H^0(\cR; \cL^k)$ such that:
\[
|\sigma_\infty|_{h_{(0,t)}^{k_j}}(p')\ge \frac{1}{2};
\]\begin{equation}
|\sigma_\infty|_{h^{k_j}_{(0,t)}}\le C\cdot \bar{\epsilon}^{\frac{1}{2}} \text{ on } M\setminus B_{d_{(0,t)}}(p, 2\epsilon);
\end{equation}
\begin{equation}
C^{-1}\le \int_{\cR} |\sigma_\infty|^2_{h_{(0,t)}^{k_j}}(k_j\omega_{(0,t)})^n\le C;
\end{equation}
and
\begin{equation}
\int_{\cR\setminus B_{d_{(0,t)}(p, 2\bar{\epsilon})}}|\sigma_\infty|^2_{h_{(0,t)}^{k_j}}(k_j\omega_{(0,t)})^n\le C\cdot \bar{\epsilon}.
\end{equation}
By the gradient estimate in Proposition \ref{prop-FSest}, $|\nabla^{\FS} \sigma_\infty|_{h_{\FS}^{k_j}\otimes\omega_{(0,t)}}\le C k_j^{1/2}$. By the gradient estimate of the potential $u_{(0,t)}$ from Proposition \ref{gradientp}, we get $|\nabla^{h_{(0,t)}}\sigma_\infty|_{h^{k_j}_{(0,t)}\otimes\omega_{(0,t)}}\le C k_j^{1/2}$. So if $\bar{\epsilon}$ is chosen sufficiently small, then we get:
\begin{equation}
|\sigma_\infty|_{h^{k_j}_{(0,t)}}(p)\ge \frac{1}{2}-C \bep^{1/4n}\ge \frac{1}{4}.
\end{equation}

{\bf Step 3:}
There exists $\ell^*>n+1$, such that $\Phi^{\ell^*}_{(0,t)}$ is injective and a local homeomorphism. Hence $\Phi^{\ell^*}_{(0,t)}$ is a homeomorphism. This part of the argument is exactly the same as that in \cite[p.1719-1721]{NTZ15} to which we refer for details. The idea is that for any pair of points $p,q\in X$, we can find peak sections almost centered at $p,q$ and show that $d_{\FS}(\Phi_{(0,t)}^{\ell}(p),\Phi_{(0,t)}^{\ell}(q))$ has a definite lower bound for some $\ell=\ell_{p,q}$. Then the effective finite generation from \cite[Proposition 7]{Li12} allows us to find a uniform $\ell^*$ that is independent of the pair.

\section{Algebraic limit and special test configuration}\label{sec-STC}

Up to this point, we have shown that there is a conical K\"{a}hler-Einstein metric $\omega_{(0,t)}$ on the pair $(X, \frac{1-t}{m}H_X)$ that is a Gromov-Hausdorff limit of $\omega_{(\epsilon, t)}$ on $(M, B_{(\epsilon, t)})$.  This allows us to get a Gromov-Hausdorff limit $X_\infty$ of a subsequence $(X, \omega_{(0,t_i)})$ by letting $t_i\rightarrow 1$. Using the regularity result in \cite{TW18} and similar arguments as in \cite{Tia15, CDS15}, we can get:
\begin{thm}[{see \cite[Theorem 4.3]{Tia15}}]\label{structure}
As $t_i\rightarrow 1$
$(X, \omega_{(0,t_i)})$ subsequentially converges to $(X_\infty,\omega_\infty)$ in the $C^\infty$ topology outside a closed subset $\bar{\cS}\cup H_\infty$, where $\bar{\cS}$ is of real codimension at least $4$, and $H_X$ converges to $H_\infty$ in the Gromov-Hausdorff topology. Moreover $\cS=\bar{\cS}$ and $H_\infty$ is a divisor of $K_{X_\infty}^{-m}$ in the regular part.
\end{thm}
For the tangent cone, we have
\begin{cor}
Let $(C_x,x,\omega_x)=\lim_{i\rightarrow \infty} (X,x,\frac{\omega_\infty}{r_i^2})$ be a tangent cone of $X$ at $x$ and $C_x=\mathcal R(C_x)\bigcup \mathcal S(C_x)$ be the decomposition of $C_x$ into regular part and singular part. Then $\mathcal S(C_x)$ is of real codimension at least 4, and $(X,x,\frac{\omega_\infty}{r_i^2})$ converges to $(C_x,x,\omega_x)$ in the $C^\infty$ topology in $\mathcal R(C_x)$.
\end{cor}

\subsection{Algebraic structrure on $X_\infty$}

On the singular variety $X$, we also have the estimates for the holomorphic sections.
\begin{lem}\label{gra}
There exists a constant $C>0$, such that
for any $\zeta\in H^0(X, K_X^{-k})$, we have the following $L^\infty$ and gradient estimates:
\begin{equation}\label{eq-C0}
\sup_X |\zeta|^2_{h^k_{(0,t)}}\le C\; k^n \int_X |\zeta|^2_{h^k_{(0,t)}}\omega_{(0,t)}^n;
\end{equation}
\begin{equation}\label{eq-C1}
|\nabla^{h_{(0,t)}}\zeta|^2_{h^k_{(0,t)}\otimes\omega_{(0,t)}}\le C\; k^{n+1}\int_X |\zeta|^2_{h^k_{(0,t)}}\omega_{(0,t)}^n.
\end{equation}
\begin{proof}
Let $\gamma_\epsilon$ be a cut-off function supported in $X\setminus (\bar{\mathcal{S}}\bigcup H_\infty)$ on $(X,\omega_{(0,t)})$ satisfying
 $$\int_X |\nabla \gamma_\epsilon|^2\leq \epsilon.$$ Since
  $$\Delta |\zeta|^2_{h^k_{(0,t)}}\geq -k|\zeta|^2_{h^k_{(0,t)}} \text{ outside } \bar{\mathcal S}\bigcup H_\infty, $$
   multiplying both sides by
$\gamma^2_\epsilon |\zeta|^{2(p-1)}_{h^k_{(0,t)}}$ we get
 $$-\gamma^2_\epsilon |\zeta|^{2(p-1)}_{h^k_{(0,t)}} \Delta |\zeta|^2_{h^k_{(0,t)}}\leq k\gamma^2_\epsilon |\zeta|^{2p}_{h^k_{(0,t)}}.$$

For $|\zeta|^2_{h^k_{(0,t)}}$ is bounded, as in Lemma 3.1 in \cite{TW19}, we get (\ref{eq-C0}) by the Moser iteration.
 From Proposition \ref{gradientp}, we know that $|\nabla^{h_{(0,t)}}\zeta|^2_{h^k_{(0,t)}}$ is bounded. So (\ref{eq-C1}) is proved similarly.
\end{proof}

\end{lem}
On $(X,\omega_{(0,t)})$, we have an $L^2$ metric on $H^0(X, K_X^{-k})$ and we can define the Bergman kernel $\rho_k(X,\omega_{(0,t)})$.
We have the following partial $C^0$ estimate.
\begin{prop}\label{partial}
There exist an integer $l>0$ and $c_l>0$ such that $\rho_l(X,\omega_{(0,t_i)})\ge c_{l}>0$ for $t_i$ as above.
\end{prop}
\begin{proof}
To prove the partial $C^0$ estimate, we need the following ingredients:
\begin{itemize}
\item compactness of $(X, \omega_{0,t})$: this follows from Proposition \ref {prop-homeo} and Theorem 1.6 in \cite{TW18}
\item solution of the $\bar\partial-$equation on $(X, \omega_{(0, t)})$: this is a consequence of Proposition \ref{prop-L2est} by taking $\epsilon\rightarrow 0$
\item gradient estimate of holomorphic sections: this is Lemma \ref{gra}
\item existence of the cut-off function $\gamma_{\bar \epsilon}$ with small $L^2$-norm of the gradient: with the help of Theorem  \ref{structure}, as the proof of Lemma 5.8 in \cite{Tia15} we can get the existence of $\gamma_{\bar \epsilon}$.
\end{itemize}
Given the above ingredients, the proposition can be proved, in a similar way to the proof of Proposition \ref{prop-homeo}, following the arguments in \cite[Section 5]{Tia15} and \cite[Proposition 4.18]{TW19}.
\end{proof}

Now, as a consequence of the partial $C^0$-estimates, the following proposition can be proved by using the same argument as in \cite{DS14, Tia15} (see also Section \ref{sec-isom}):
\begin{prop}
$X_\infty$ is a $\mathbb{Q}$-Fano variety.
\end{prop}
Indeed, using the orthogonal basis of $H^0(X,K_X^{-l})$, we have holomorphic maps:
$$\Phi_t: X \rightarrow \Phi_t(X) \subset \mathbb{CP}^N.$$ Let $X'_\infty$ be the Chow limit of $\Phi_t(X)$ as subvarieties of $\bP^N$. From Proposition \ref{partial}, we know that $\Phi_{t_i}$ is uniformly Lipschitz. As a consequence, by letting $t_i\rightarrow 1$, we get a Lipschitz map:
\[
\Phi_\infty=\lim_{i\rightarrow +\infty} \Phi_{t_i}: X_\infty \rightarrow X'_\infty \subset \bP^N.
\]
As pointed out in Section \ref{sec-isom}, one can then use the partial $C^0$-estimates to prove that $\Phi_\infty$ is indeed a homeomorphism.

\subsection{Weak K\"{a}hler-Einstein metric on $X_\infty$}
For each $t\in (0,1)$ and $\epsilon<\epsilon^*(t)$, the conical K\"{a}hler-Einstein metric on $(M, B_{(\epsilon,t)})$ corresponds to
the following complex Monge-Amp\`{e}re equation on $M$
\begin{equation}
(\sddb\psi)^n=\frac{e^{-t\psi}}{|s_H|^{2(1-t)/m}\prod_i |s_i|^{2(1-\beta_i)}},
\end{equation}
where $1-\beta_i=b_i+t\epsilon \theta_i+(1-t)\theta_i$. Equivalently, we have (see \eqref{eq-MAref}-\eqref{eq-Omegept})
\begin{equation}
(\chi_\epsilon+\sddb u)^n=e^{-t u}\Omega(\epsilon,t),
\end{equation}
where the volume form $\Omega(\epsilon, t)$ is given by:
\begin{eqnarray}
\Omega(\epsilon, t)
&=&\frac{e^{-t\psi_\epsilon}}{|s_H|^{2(1-t)/m}\prod_i |s_i|^{2(b_i+t\epsilon\theta_i+(1-t)\theta_i)}}\nonumber\\
&=&\frac{e^{-(\psi_0-\sum_i b_i\kappa_i)}}{\|s_H\|_{m\psi_1}^{2(1-t)/m}\|s_\theta\|_{\kappa_\theta}^{2(t \epsilon+1-t)}\prod_{i}\|s_i\|_{\kappa_i}^{2b_i}}.
\end{eqnarray}

As we have proved in Proposition \ref{prop-homeo}, as $\epsilon\rightarrow 0$, $(M, \omega_{(\epsilon,t)})$ converges to $(X, d_{(0,t)})$ in the Gromov-Hausdorff topology and $(X, d_{(0,t)})$ is compatible with the weak conical K\"{a}hler-Einstein metric $\omega_{(0,t)}$ on $X$ in the sense that $(X, d_{(0,t)})$ is the metric completion of $(X^{\reg}, \omega_{(0,t)})$. Over $X$, $\omega_{(0,t)}$ satisfies the following Monge-Amp\`{e}re equation:
\begin{equation}
(\sddb\psi)^n=\frac{e^{-t\psi}}{|s_{H_X}|^{2(1-t)/m}},
\end{equation}
or equivalently:
\begin{equation}
(\chi_0+\sddb u)^n=e^{-tu}\frac{e^{-t\psi_0}}{\|s_{H_X}\|_{m\psi_0}^{2(1-t)/m}}
\end{equation}

The potential function $u$ is Lipschitz and in particular bounded. As $t\rightarrow 1$, $\omega_{(0,t)}$ converges to $\omega_\infty$ on $X_\infty$. $\omega_\infty=\sddb\psi_\infty$ satisfies the K\"{a}hler-Einstein equation:
\[
(\sddb \psi_\infty)^n=e^{-\psi_\infty}.
\]

Now one can construct the special test configuration as in \cite{Tia15, CDS15}. On the one hand, we can embed $X$ and $X_\infty$ into a common projective space $\bP^N$ such that the Hilbert point of $X_\infty$ is in the closure of the orbit of the Hilbert point of $X$ under the action of $G:=PGL(N+1,\bC)$. On the other hand, the existence of the limit K\"{a}hler-Einstein metric implies that $Aut(X_\infty, -K_\infty)$ is reductive (see \cite{Tia15}). For the reader's convenience, we sketch an argument of the Matsushima type (see \cite{Mat57}) for proving reductivity that is due to Tian  and refer the details to \cite[pp. 44-53]{Tia15}. Note that an alternative approach (as pointed out in \cite{CDS15}) is to use the uniqueness result of Berndtsson and its possible generalization to the singular setting as exposed in \cite[Appendix C]{BBEGZ}.

By the previous section, we can already embed $X_\infty$ as a normal projective variety into a projective space $\bP^N$. $PGL(N+1, \bC)$ acts on the corresponding Hilbert scheme. Let $G_\infty$ be the stabilizer of $PGL(N+1, \bC)$-action at the Hilbert point corresponding to $X_\infty$.

Let $Z$ be a holomorphic vector field of $\bC\bP^N$ that is tangent to $X_\infty$. Let $W$ be the real or imaginary part of $Z$ and $\sigma_s$ be the one-parameter subgroup of automorphisms generated by $W$. Then we have:
\[
\sigma_s^*\omega_\infty=\omega_\infty+\sddb\xi_s.
\]
Let $\theta_\infty=u+\sqrt{-1}v$ where $u$ or $v$ is equal to $\left.\frac{\partial\xi_s}{\partial s}\right|_{s=0}$ according to whether $W$ is the real part or the imaginary part of $Z$. With the same argument as in \cite[pp. 45-50]{Tia15}, we know that $\theta_\infty$ is a bounded function on $X_\infty$ and moreover:
\[
\int_{X_\infty}|\nabla \theta_\infty|^2\omega_\infty^n<\infty \quad \text{ and } \int_{X_\infty} \theta_\infty \omega_\infty^n=0.
\]
Let $\zeta$ be any function on $X_\infty$ that can be extended to be a smooth function in a neighborhood of $X_\infty$ in $\bC\bP^N$. Then using the K\"{a}hler-Einstein equation and the change of variable formula, we have:
\[
\int_{X_\infty} \zeta\circ \sigma_t^{-1} \omega_\infty^n=\int_{X_\infty} \zeta e^{-\xi_t} \omega_\infty^n,
\]
which is equivalent to
\begin{equation}\label{eq-inteigen}
\int_{X_\infty}\left( \int_0^s W(\zeta)\circ\sigma_\tau d\tau\right)\omega^n_\infty=\int_{X_\infty} \zeta\left(\int_0^s \dot{\xi} d\tau\wedge \sigma_\tau^*\omega_\infty^n\right).
\end{equation}
Dividing both sides of \eqref{eq-inteigen} by $s$ and letting $s$ tends to $0$, we deduce:
\begin{equation}
\int_{X_\infty} Z(\zeta) \omega_\infty^n=\int_{X_\infty} \zeta\theta_\infty \omega_\infty^n.
\end{equation}
So we get, in the {\it weak} sense,
\begin{equation}\label{eq-eigen}
-\Delta_\infty \theta_\infty=\theta_\infty \text{ on } X_\infty.
\end{equation}
Using the diagonal approximation of $(X_\infty, d_\infty)$ by $(X_{(\epsilon_i,t_i)}, \omega_{(\epsilon_i,t_i)})$ (with $\epsilon_i=\epsilon_i(t_i)\rightarrow 0$), one can argue as in \cite[pp. 51-52]{Tia15} that there exist $\theta_i$ on $X_{(\epsilon_i,t_i)}$ such that
$\Delta_{(\epsilon_i,t_i)}\theta_i=\lambda_i\theta_i$ and $\lim_{i\rightarrow+\infty}\lambda_i=1$. Applying the Bochner identity, we get:
\[
\int_{X_\infty} |\nabla^{0,1}\bar{\partial}\theta_i|^2\omega_{(\epsilon_i,t_i)}^n\le (\lambda_i-1)\int_M |\partial \theta_i|^2\omega_{(\epsilon_i,t_i)}^n=\lambda_i(\lambda_i-1).
\]
It follows that $\nabla^{0,1}\bar{\partial}\theta_\infty=0$ and $\bar{\partial}\theta_\infty$ induces a holomorphic vector field $Z$ outside the singular set $\cS$ of $(X_\infty, d_\infty)$.
The imaginary part of $\theta_\infty$ corresponds to the imaginary part of $Z$ which is a Killing field. Because \eqref{eq-eigen} is a real equation, we conclude that the Lie algebra of $G_\infty$ is indeed the complexification of a Lie algebra of Killing fields. The rest of the argument is exactly as in \cite[pp. 53]{Tia15}.

Given the reductivity of $Aut(X_\infty, -K_\infty)$ the Luna Slice theorem (see \cite{Don12b}) there exists a one parameter subgroup of $G$ that specially degenerates $X$ to $X_\infty$. We will show below that the Futaki invariant of this test configuration is 0. But this contradicts the K-polystability of $X$ and completes the proof of Theorem \ref{thm-main}.

Finally we adapt the argument in \cite{DT92} to prove the vanishing of Futaki invariant (see also \cite{Tia15,CDS15}). At first we have
$$\omega_{t_i}=\frac{1}{l}\omega_{FS}-\frac{1}{l}\sqrt{-1}\partial\bar\partial\log \rho_l(X,\omega_{t_i}).$$ Since $\rho_l(X,\omega_{t_i})$ has a definite lower bound and uniform Lipschitz, we have a limit function $\rho_l(X_\infty,\omega_\infty)$ which is just the Bergman kernel of $-lK_\infty$ on $X_\infty$. We choose a bounded hermitian metric $h_\infty$ on $-K_\infty$ with $R(h_\infty)=\omega_\infty$ in the regular part.

For simplicity, denoting $$\bar\omega_0=\frac{1}{l}\omega_{FS}, \phi_\infty=-\frac{1}{l}\log \rho_l(X_\infty,\omega_\infty),$$ we define
$$\omega_s=\bar\omega_0+s\sqrt{-1}\partial \bar\partial \phi_\infty=s\omega_\infty+(1-s)\bar\omega_0.$$ The Ricci potential $f_0$ of $\bar\omega_0$ is given by
$$\sqrt{-1}\partial\bar\partial f_0={\rm Ric}\,\bar\omega_0-\bar\omega_0.$$ Denoting by $Z$ the holomorphic vector filed generating the $\mathbb{C}^*$ action of test configuration, we know that (see \cite{Fut83, DT92}):
$$Fut_{X_\infty}(Z)=\int_{X_\infty}Z(f_0)\omega_0^n.$$
Put
$$f_s=-\log\frac{\omega_s^n}{\bar\omega_0^n}-s\phi_\infty+f_0$$ which is the Ricci potential of $\omega_s$. Denoting by $\psi$ the potential of $Z$ with respect to $\bar\omega_0$: $$i_Z\bar\omega_0=\sqrt{-1}\bar\partial \psi, $$ then we know that
$$i_Z\omega_s=\sqrt{-1}\bar\partial\theta_s, \text{ for }\theta_s=(1-s)\psi+s\theta_\infty=\psi+sZ(\phi_\infty).$$
From $\frac{d}{ds}f_s=-\Delta_s \phi_\infty-\phi_\infty$, we have:
\begin{align}\label{invariant}
\frac{d}{ds}\int_{X_\infty}Z(f_s)\omega_s^n&=\int_{X_\infty}[Z(-\Delta_s\phi_\infty-\phi_\infty)
+Z(f_s)\Delta_s\phi_\infty]\omega_s^n\notag\\
&=\int_{X_\infty}[(\Delta_s\theta_s+Z(f_s))\Delta_s\phi_\infty-Z(\phi_\infty)]\omega_s^n\notag\\
&=\int_{X_\infty}[\Delta_s\theta_s+\theta_s+Z(f_s)]\Delta_s\phi_\infty\omega_s^n.
\end{align}
The integration by parts is guaranteed as in \cite{TW19}. A direct computation shows that $\bar\partial(\Delta_s\theta_s+\theta_s+Z(f_s))=0$, so $\Delta_s\theta_s+\theta_s+Z(f_s)$ is a bounded holomorphic function which must be constant. From (\ref{invariant}), we know that $$\frac{d}{ds}\int_{X_\infty}Z(f_s)\omega_s^n=0.$$
Since  $f_1=0$, we get $\int_{X_\infty}Z(f_0)\omega_0^n=0$.


\appendix
\section{Log version of Berman-Boucksom-Jonsson's result}\label{app-BBJ}

 A key result needed in the argument of \cite{BBJ15} is the following local regularization result of Demailly:
\begin{prop}[{\cite[Proposition 3.1]{Dem92}}]\label{prop-Demailly}
Let $\Phi$ be a psh function on a bounded pseudoconvex open set $U\subset \bC^{n+1}$.
For every $m>0$, let
$$\cH_{U}(m\Phi)=\left\{f\in \cO(U); \int_U |f|^2 e^{-m\Phi}d\lambda<+\infty\right\}$$
and let $\Phi_m=\frac{1}{m}\log\sum_l |f_l|^2$ where $\{f_l\}$ is an orthonormal basis of $\cH_U(m\Phi)$. Then there
are constants $C_1, C_2>0$ independent of $m$ such that for every $z\in U$

\begin{equation}\label{eq-PhivsPhim}
\Phi(z)-\frac{C_1}{m}\le \Phi_m(z).
\end{equation}

\end{prop}

\noindent {\bf Sketch of the proof of Theorem \ref{thm-BBJ}}\vskip 2mm

Fix $\ell \in \bZ_{>0}$ such that $L=-\ell (K_M+B)$ is an genuine line bundle over $M$ which is ample.
From now on, for simplicity, we denote $\cM_B(\ell^{-1}, \vphi)$ by $\cM_B(\vphi)$ and similarly for other functionals. By Remark \ref{rem-scaling}, we just need to prove Theorem \ref{thm-BBJ} with the parameter $t=\ell^{-1}$.

Because $\cM_B\ge \cD_B$ and $\CM_B^\NA\ge \cM_B^\NA\ge \cD_B^\NA$ we immediately get the implication $1 \Rightarrow 2$, and 3 $\Rightarrow$ 4 $\Rightarrow$ 5. The equivalence of 1 and 2, with $\delta_\cM$ and $\delta_\cD$ satisfying the first inequality in \eqref{eq-deltarel}, was proved in \cite[Corollary 3.5]{Berm13} (see \cite[Proposition 15]{Li12}).
The equivalence of 3 and 5 was proved in the case when $B=0$ in \cite{BBJ15} and the more general log version was proved in \cite{Fuj16}. The implication 1 $\Rightarrow$ 3, 2 $\Rightarrow$ 4 was proved in \cite{BHJ16}.

So one just needs to show $4 \Rightarrow 2$, which will be proved by contradiction. Denote $\delta=\delta^\NA$ and assume $2$ is not true. Then we can pick $\delta'\in (0, \delta)$ and a sequence $\{\vphi_j\}_1^\infty\in \cE^1$ such that:
\[
\cM_B(\vphi_j)\le \delta' J(\vphi_j)-j.
\]
We normalizes $\vphi_j$ such that $\sup (\vphi_j-\psi)=0$. The inequality $\cM_{B}\ge C-n J$ implies $J(\vphi_j)\rightarrow+\infty$, and hence $E(\vphi_j)\le -J(\vphi_j)\rightarrow-\infty$.

{\bf Step 1: Constructing a geodesic ray in $\cE^1$}

Connect $\psi$ and $\vphi_j$ by a geodesic segment $\{\vphi_{j,s}\}$ parametrized so that $S_j=-E(\vphi_j)\rightarrow +\infty$. For any $s\in (0, S_j]$, we have $E(\vphi_{j,s})=-s$ and $\sup(\vphi_{j,s}-\psi)=0$.
So $J(\vphi_{j,s})\le \sup(\vphi_{j,s}-\psi)-E(\vphi_{j,s})=s \le S_j$ and $\cM_B(\vphi_j)\le \delta' S_j-j\le \delta' S_j$.
By \cite{BDL16}, $\cM_B$ is convex along geodesic segment. So
\begin{equation}
\cM_B(\vphi_{j,s})\le \frac{S_j-s}{S_j}\cM_B(\psi)+\frac{s}{S_j}\cM_B(\vphi_j)\le \delta' s+C.
\end{equation}
Using $\cM_B\ge H_B-n J$, we get $H_B(\vphi_{j,s})\le (\delta'+n)s+C$. So for any fixed $S>0$ and $s\le S$, the metrics $\vphi_{j,s}$ lie in the set:
\[
\mathcal{K}_S:=\{\vphi\in \cE^1; \sup(\vphi-\psi)= 0 \text{ and } H_B(\vphi)\le (\delta'+n) s+C \}.
\]
This is a compact subset of the metric space $(\cE^1, d_1)$ by Theorem \ref{thm-BBEGZ} from \cite{BBEGZ}. So, by arguing as in \cite{BBJ15}, after passing to a subsequence, $\{\vphi_{j,s}\}$ converges to a geodesic ray
$\{\vphi_s\}_{s\ge 0}$ in $\cE^1$, uniformly for each compact time interval. $\{\vphi_s\}$ satisfies $\sup(\vphi_s-\psi_0)=0$ and $E(\vphi_s)=-s$. Moreover,
\begin{equation}\label{eq-Dvphiupb}
\cD_B(\vphi_s)\le \cM_B(\vphi_s)\le \delta' s+C  \text{ for } s\ge 0.
\end{equation}
Denote $\bD=\{\tau\in \bC; |\tau|\le 1\}$ and $\bD^*=\bD\setminus \{0\}$.
$\{\vphi_s\}$ defines an $S^1$-invariant metric $\Phi$ on $p_1^*L$ over $M\times \bD^*$ such that the restriction of $\Phi$ to $M\times \{\tau\}$ is $\vphi_{\log |\tau|^{-1}}$.
Moreover, $\Phi$ is plurisubharmonic because $\Phi_j$ converges to $\Phi$ locally uniformly in $L^1$ topology, where $\Phi_j$ is the psh metric on $p_1^*L\rightarrow M\times \{e^{-S_j}\le |\tau|<1\}$ defined by the geodesic segment $\{\vphi_{j,s}\}$.

{\bf Step 2: Approximate geodesic ray by test configurations}

Because $\sup(\vphi_s-\vphi)= 0$ for any $s$, $\Phi$ extends to a psh metric on $p_1^*L$ over $M\times\bD$. Consider the multiplier ideals $\cJ(m\Phi)\subset \cO_{M\times\bC}$
which is defined over any $U\subset M\times\bC$ as
\begin{eqnarray}\label{eq-JmPhi}
\cJ(m\Phi)(U)&:=&\left\{f\in \cO_{M\times\bC}(U);  w(f)+A_{M\times \bC}(w)-m\; w(\Phi)\ge 0 \right. \nonumber\\
&&\hskip 2cm \left. \text{ for any divisorial valuation } w \text{ on } M\times\bC \right\}.
\end{eqnarray}
$\cJ(m\Phi)$ are co-supported on $\bC^*$-invariant proper subvarieties of the central fiber $M\times\{0\}$. Fix a very ample line bundle $H$ on $M$. Choose $m_0$ large enough such that
$A:=m_0L-K_M-(n+1)H$ is ample on $M$ and write:
\[
\cF:=\cO((m+m_0)p_1^*L)\otimes \cJ(m\Phi)=\cO( p_1^*(K_M+m L+A+(n+1)H)\otimes \cJ(m\Phi)).
\]
Then for any $i>0$, $j>0$, we have:
\begin{eqnarray*}
&&R^i(p_2)_* \left(\cF\otimes p_1^*H^{\otimes(j-i)}\right)\\
&=&R^i(p_2)_*(\cO(p_1^*(K_M+mL+A+(j+(n+1)-i)H)\otimes \cJ(m\Phi)))\\
&=&0,
\end{eqnarray*}
by the Nadel vanishing theorem.
By the relative version of Castelnuovo-Mumford criterion, $\cO((m+m_0)p_1^*L)\otimes\cJ(m\Phi)$ is $p_2$-globally generated. Because $\bD$ is Stein, $\cO((m+m_0)p_1^*L)\otimes\cJ(m\Phi))$ is generated by global sections on $M\times\bD$ if $m\ge 1$ and $(m+m_0)$ is sufficiently divisible.

Let $\mu_m: \cM_m\rightarrow M\times \bC$ denote the normalized blow-up of $M\times\bC$ along $\cJ(m\Phi)$, with exceptional divisor $E_m$ and set $\cL_m=\mu_m^*p_1^* L-\frac{1}{m+m_0}E_m$. Let $\cB_m$ be the strict transform of $B\times\bC$. Then $(\cM_m, \cB_m, \cL_m)$ is a normal semi-ample test configuration for $(M, B, L)$, inducing a non-Archimedean metric $\phi_m\in \cH^\NA$ given by:
\[
\phi_m(v)=\frac{1}{m+m_0}v\left(\cJ(m\Phi)\right)
\]
for each $\bC^*$-invariant divisorial valuation $v$ on $M\times\bC$.

Choose any $S^1$-invariant smooth psh metric $\Psi_m$ on $\cL_m$. The corresponding geodesic ray $\vphi_{m,s}$ satisfies (see \cite{Berm15, BHJ16}):
\begin{equation}
\lim_{s\rightarrow+\infty}\frac{L_B(\vphi_{m,s})}{s}=L_B^\NA(\phi_m) \text{ and } \lim_{s\rightarrow+\infty}\frac{1}{s}E(\vphi_{m,s})=E^\NA(\phi_m)=-J^\NA(\phi_m).
\end{equation}
The psh metric $\Phi_m$ on $p_1^*L$ over $M\times\bD$ induced by $\Psi_m$ has analytic singularities of type $\cJ(m\Phi)^{1/(m+m_0)}$. By Proposition \ref{prop-Demailly}, $\Phi_m$ is less singular than $\Phi$. Note that here we used the smoothness of the ambient space $M\times\bC$. By monotonicity of $E$, we get:
\begin{equation}
E^\NA(\phi_m)=\lim_{s\rightarrow+\infty} \frac{E(\vphi_{m,s})}{s}\ge \lim_{s\rightarrow+\infty}\frac{E(\vphi_s)}{s}=-1.
\end{equation}

{\bf Step 3: asymptotic expansion of $L_B$ along the geodesic ray}

The following theorem is the log version of \cite[Theorem 3.1]{BBJ15}. It can be seen as the generalization of Berman's expansion of log-Ding energy in \cite{Berm15} along test configurations to the case of any subgeodesic rays.
\begin{thm}\label{thm-LBexpan}
Let $\vphi_s$ be a subgeodesic ray in $\cE^1$ normalized such that $\sup(\vphi_s-\vphi_0)=0$ and let $\Phi$ be the corresponding $S^1$-invariant psh metric on $p_1^*L \rightarrow M\times\bC$. Then
\begin{equation}\label{eq-LBexpan}
\lim_{s\rightarrow+\infty} \frac{L_B(\vphi_s)}{s}=\inf_{w} \left(A_{(M,B)\times\bC}(w)-1-w(\ell^{-1} \Phi)\right)
\end{equation}
where $w$ ranges over $\bC^*$-invariant divisorial valuations on $M\times\bC$ such that $w(\tau)=1$, and $A_{(M,B)\times\bC}(w)$ is the log discrepancy of $w$.
\end{thm}

For the reader's convenience, we write down the part of proof in \cite{BBJ15} that needs to be modified (see Lemma \ref{lem-eAMB} and identity \eqref{eq-winfJ}).

\begin{proof}
By \cite{Bern09}, the function
\begin{equation}
L(\tau):=L_B(\ell^{-1}, \vphi_{\log|\tau|^{-1}})=-\log\left(\frac{1}{V}\int_M \frac{e^{-\ell^{-1}\vphi}}{|s_B|^2}\right),
\end{equation}
is subharmonic on $\bD$ and its Lelong number $\nu$ at the origin coincides with the negative of the left-hand-side of \eqref{eq-LBexpan}.
We need to show that $\nu$ is equal to
\begin{equation}\label{eq-def-s}
s:=\sup_w \left(w(\ell^{-1}\Phi)+1-A_{(M,B)\times\bC}(w)\right).
\end{equation}
By \cite[Proposition 3.8]{Berm15}, $\nu$ is the infimum of all $c\ge 0$ such that:
\begin{equation}
\int_U e^{-(L(\tau)+(1-c)\log|\tau|^2)} id\tau\wedge d\bar{\tau}=\int_{M\times U}e^{-(\ell^{-1}\Phi+\log |s_{B\times\bC}|^2+(1-c)\log|\tau|^2)}id\tau\wedge d\bar{\tau}<+\infty.
\end{equation}
Locally near each point of $M\times\{0\}$, if $B=\sum_i \alpha_i \{f_i=0\}$, then we have
$$e^{-\ell^{-1} \Phi-\log|s_{B\times\bC}|^2}id\tau\wedge d\bar{\tau}=e^{-\ell^{-1}\vphi-\sum_i \alpha_i \log|f_i|^2} dV$$
with $\vphi$ psh and $dV$ a smooth volume form. Setting $p:=\lfloor c\rfloor$ and
$r=r-p\in [0,1)$, we get:
$$
e^{-(\ell^{-1} \Phi+\log|s_{B\times\bC}|^2+(1-c)\log|\tau|^2)}id\tau\wedge d\bar{\tau}=|\tau|^{2p}e^{-(\ell^{-1}\vphi+\sum_i\alpha_i\log|f_i|^2+(1-r)\log|\tau|)}dV.
$$
It follows from \cite[Theorem 5.5]{BFJ08} that
\begin{eqnarray*}
&&\int_{M\times U}e^{-(\ell^{-1}\Phi+\log|s_{B\times\bC}|^2+(1-c)\log|\tau|^2)}id\tau\wedge d\bar{\tau}<+\infty\\
&&\hskip 3cm \Longrightarrow\hskip 1cm \sup_w\frac{w(\ell^{-1}\Phi)+w(B\times\bC)+(1-r)w(\tau)}{p w(\tau)+A_{M\times\bC}(w)}\le 1,
\end{eqnarray*}
where $w$ ranges over all divisorial valuations on $M\times\bC$. By homogeneity and by the $S^1$-invariance of $\Phi$, it suffices to consider $w$ that are $\bC^*$-invariant and normalized by
$w(\tau)=1$. We then get:
\begin{equation}
w(\ell^{-1}\Phi)+1\le p+r+A_{M\times\bC}(w)-w(B\times\bC)=c+A_{(M,B)\times\bC}(w).
\end{equation}
So we get $s\le \nu$.

Conversely, \cite[Theorem 5.5]{BFJ08} shows that:
\begin{eqnarray}\label{eq-intsuff}
&&\sup_w \frac{w(\ell^{-1}\Phi)+w(B\times\bC)+1-r}{p+A_{M\times\bC}(w)}<1\\
&&\hskip 3cm \Longrightarrow \hskip 0.5cm \int_{M\times U}e^{-(\ell^{-1}\Phi+\log|s_{B\times\bC}|^2+(1-c)\log|\tau|)}i d\tau\wedge d\bar{\tau}<+\infty. \nonumber
\end{eqnarray}

\begin{lem}\label{lem-eAMB}
The left-hand-side of \eqref{eq-intsuff} is equivalent to the existence of $\epsilon>0$ such that
\begin{equation}\label{eq-eAMB}
w(\ell^{-1}\Phi)+1\le (1-\epsilon) A_{(M,B)\times\bC}(w)+c \quad \text{ for all } w.
\end{equation}
\end{lem}
\begin{proof}
If the left-hand-side of \eqref{eq-intsuff} holds, then there exists $\epsilon>0$ such that:
\begin{eqnarray*}
w(\ell^{-1}\Phi)+1&<& (1-\epsilon)(p+ A_{M\times\bC}(w))-w(B\times\bC)+r\\
&=&(1-\epsilon)A_{(M,B)\times\bC}(w)+c-\epsilon w(B\times\bC)-\epsilon p\\
&\le&(1-\epsilon) A_{(M,B)\times\bC}(w)+c.
\end{eqnarray*}
Conversely, assume \eqref{eq-eAMB} holds. Notice that $(M\times\bC, B\times\bC)$ is log terminal, and hence
$$ \frac{A_{M\times\bC}(w)}{w(B\times\bC)}\ge \lct:=\lct(M\times\bC, B\times\bC)>1. $$ We then have:
\begin{eqnarray*}
w(\ell^{-1}\Phi)+1&\le&(1-\epsilon)A_{(M,B)\times\bC}(w)+c\\
&=& (1-\epsilon)A_{M\times\bC}(w)-w(B\times\bC)+\epsilon w(B\times\bC)+c\\
&\le &(1-\epsilon)A_{M\times\bC}(w)-w(B\times\bC)+\epsilon\cdot \lct^{-1} A_{M\times\bC}(w)+c\\
&=& (1-(1-\lct^{-1})\epsilon)A_{M\times\bC}(w)-w(B\times\bC)+c.
\end{eqnarray*}
On the other hand, since $p$ is bounded and $A_{M\times\bC}(w)\ge 1$, there exists a constant $C_1>0$ such that $p\le C_1\cdot A_{M\times \bC}(w)$. Then if we let $\epsilon'=\frac{1}{C_1+1}(1-\lct^{-1})\epsilon$, then the left-hand-side of \eqref{eq-eAMB} indeed holds:
\[
w(\ell^{-1}\Phi)+1\le (1-\epsilon')(A_{M\times\bC}(w)+p)-w(B\times\bC)+r
\]

\end{proof}
With the above lemma, the inequality $\nu\le s$ can be proved in the same way as in \cite{BBJ15} using proof by contradiction. If $\nu>s$ then there exists $\rho>0$ and a sequence of $\bC^*$-invariant divisorial valuations with $w_j(\tau)=1$ such that:
\[
w_j(\ell^{-1}\Phi)\ge\left(1-\frac{1}{j}\right)A_{(M,B)\times\bC}(w_j)-1+s+\rho.
\]
Let $W$ be the subset of the Berkovich analytification $(M\times\bC)^{\rm an}$ consisting of semivaluations $w$ that are $\bC^*$-invariant and satisfy $w(\tau)=1$. We can take limit $w_j$ converges to a semivaluation $w_\infty$. For each $m\ge 1$, the multiplier ideal sheaf $\cJ(m\Phi)$ defined in \eqref{eq-JmPhi} satisfies:
\begin{eqnarray*}
\frac{1}{m}w_j(\ell^{-1}\cJ(m\Phi))&\ge& w_j(\ell^{-1}\Phi)-\frac{1}{m}A_{M\times\bC}(w_j)\\
&\ge& (1-\frac{1}{j}-\frac{1}{m})A_{(M,B)\times\bC}(w_j)-\frac{1}{m}w_j(B\times\bC)-1+s+\rho,
\end{eqnarray*}
for all $j\ge 1$. By \cite{JM12}, $w\rightarrow w(\ell^{-1}(m\Phi))$ and $w\rightarrow w(B\times\bC)$ are continuous, while $w\mapsto A_{(M,B)\times\bC}(w)$ is lower semicontinuous on $W$.
Letting $j\rightarrow+\infty$, we get:
\begin{equation}\label{eq-winfJ}
w_\infty(\ell^{-1}\cJ(m\Phi))\ge \left(1-\frac{1}{m}\right)A_{(M,B)\times\bC}(w_\infty)-\frac{1}{m}w_\infty(B\times\bC)-1+s+\rho.
\end{equation}
In particular, $A_{(M,B)\times\bC}(w_\infty)<+\infty$.
On the other hand, using the definition of $s$ in \eqref{eq-def-s} and the inequality \eqref{eq-wJPhim}, we get:
\begin{equation}
\frac{1}{m}w(\ell^{-1}\cJ(m\Phi))\le w(\ell^{-1}\Phi)\le A_{(M,B)\times\bC}(w)-1+s
\end{equation}
for all divisorial valuations $w$. Using density of divisorial valuations in $W$ and semicontinuity properties, we get:
\begin{equation}
A_{(M,B)\times\bC}(w_\infty)-1+s\ge \left(1-\frac{1}{m}\right)A_{(M,B)\times\bC}(w_\infty)-\frac{1}{m}w_\infty(B\times\bC)-1+s+\rho,
\end{equation}
which gives a contradiction by letting $m\rightarrow+\infty$.

\end{proof}
In the above argument, we used the inequality:
\begin{equation}\label{eq-wJPhim}
 w(\cJ(m\Phi))\le  m\; w(\Phi)\le w(\cJ(m\Phi))+ A_{M\times\bC}(w).
\end{equation}
The first inequality holds because of \eqref{eq-PhivsPhim}. The second inequality follows from the definition of multiplier ideal $\cJ(m\Phi)$ in \eqref{eq-JmPhi}. With Theorem \ref{thm-LBexpan} and \eqref{eq-wJPhim}, we get the following result
\begin{prop}[\cite{BBJ15}]
We have the identity:
\begin{equation}\label{eq-limLNA}
\lim_{m\rightarrow+\infty} L_B^\NA(\phi_m)=\lim_{s\rightarrow+\infty} \frac{L_B(\vphi_s)}{s}.
\end{equation}
\end{prop}
\begin{proof}
For simplicity, denote the right-hand-side of \eqref{eq-limLNA} by $T'$ and
\[
T_+:=\limsup_{m\rightarrow+\infty} L_B^\NA(\phi_m)\quad \ge  \quad T_-:= \liminf_{m\rightarrow+\infty} L_B^\NA(\phi_m).
\]
Using \eqref{eq-PhivsPhim}, we get, for any $\bC^*$-invariant valuation $w$ on $M\times\bC$ with $w(\tau)=1$,
\begin{eqnarray*}
A_{(M,B)\times\bC}(w)-1-\frac{\ell^{-1}}{m+m_0}w(\cJ(m\Phi))&\ge& A_{(M,B)\times\bC}(w)-1-\frac{m\ell^{-1}}{m+m_0} w(\Phi)\\
&\ge& A_{(M,B)\times\bC}(w)-1-w(\ell^{-1}\Phi).
\end{eqnarray*}
Taking infimum on both sides, we get $L_B(\phi_m)\ge T'$ for any $m$ and hence $T_-\ge T'$. On the other hand, for $\epsilon>0$, there exists $v$ such that
$$
A_{(M,B)\times\bC}(v)-1-v(\ell^{-1}\Phi)\le T'+\epsilon.
$$
Then we use both inequalities \eqref{eq-wJPhim} to get:
\begin{eqnarray*}
L^\NA_B(\phi_m)&\le&A_{(M,B)\times\bC}(v)-1-\frac{\ell^{-1}}{m+m_0}v(\cJ(m\Phi))\\
&=& A_{(M,B)\times\bC}(v)-1-\frac{\ell^{-1}}{m}v(\cJ(m\Phi))+\frac{m_0 \ell^{-1}}{m(m+m_0)}v(\cJ(m\Phi)) \\
&\le& A_{(M,B)\times\bC}(v)-1- v(\ell^{-1}\Phi)+\frac{\ell^{-1}}{m+m_0}A_{M\times\bC}(v)+\frac{m_0 \ell^{-1}}{m+m_0}v(\Phi)\\
&\le& T'+\epsilon+\frac{\ell^{-1}}{m+m_0}A_{M\times\bC}(v)+\frac{m_0 \ell^{-1}}{m+m_0}v(\Phi).
\end{eqnarray*}
Taking $\limsup$ as $m\rightarrow+\infty$, we get $T_+\le T'+\epsilon$. Since $\epsilon>0$ is arbitrary, we get $T_+\le T'$ and hence $T_+=T_-=T'$ as wanted.
\end{proof}

{\bf Step 4: Completion of the proof}

With the above preparations, the last step of the proof is exactly the same as the argument in \cite{BBJ15}. Indeed, on the one hand,
\begin{equation}
\ell\cdot \lim_{s\rightarrow+\infty} \frac{L_B(\vphi_s)}{s}
=
\lim_{s\rightarrow+\infty} \frac{\cD_B(\vphi_s)}{s}+\lim_{s\rightarrow+\infty}\frac{E(\vphi_s)}{s}\le \delta'-1.
\end{equation}
On the other hand, because of the uniform K-stability, we have:
\begin{eqnarray*}
\ell\cdot L_B^\NA(\phi_m)&=&\cD_B^\NA(\phi_m)+E^\NA(\phi_m) \ge \delta J^\NA(\phi_m)+E^\NA(\phi_m)\\
&=&(1-\delta)E^\NA(\phi_m)\ge \delta-1.
\end{eqnarray*}
Because of \eqref{eq-limLNA}, these two identities contradict $\delta'<\delta$.

\vskip 3mm

\noindent
Department of Mathematics, Purdue University, West Lafayette, IN 47907-2067

\noindent
{\it E-mail address:} li2285@purdue.edu

\vskip 2mm

\noindent
School of Mathematical Sciences and BICMR, Peking University, Yiheyuan Road 5, Beijing, P.R.China, 100871

\noindent {\it E-mail address:} tian@math.princeton.edu

\vskip 2mm

\noindent
School of Mathematical Sciences, Zhejiang University, Zheda Road 38, Hangzhou, Zhejiang,
310027, P.R. China

\noindent {\it E-mail address:} wfmath@zju.edu.cn


\begin{thebibliography}{999999}

\bibitem{Berm13}
R. Berman, A thermodynamical formalism for Monge-Amp\`{e}re equations, Moser-Trudinger inequalities and K\"{a}hler-Einstein metrics. Adv. Math. {\bf 248} (2013), 1254-1297.

\bibitem
{Berm15}
R. Berman, K-stability of ${\bf Q}$-Fano varieties admitting K\"{a}hler-Einstein metrics, Invent. Math. {\bf 203} (2015), no.3, 973-1025.

\bibitem
{BB14}
R. Berman, R. Berndtsson, Convexity of the K-energy on the space of K\"{a}hler metrics, J. Amer. Math. Soc. {\bf 30} (2017), 1165-1196.

\bibitem
{BBEGZ}
R. Berman, S. Boucksom, P. Eyssidieux, V. Guedj, A. Zeriahi, K\"{a}hler-Einstein metrics and the K\"{a}hler-Ricci flow on log Fano varieties, arXiv:1111.7158. To appear in J. Reine Angew. Math.

\bibitem
{BBJ15}
R. Berman, S. Boucksom, M. Jonsson, A variational approach to the Yau-Tian-Donaldson conjecture, arXiv:1509.04561.

\bibitem
{BDL16}
R. Berman, T. Darvas, and C. H. Lu, Convexity of the extended K-energy and the large time behaviour of the weak Calabi flow, Geom. Topol. {\bf 21} (2017), 2945-2988.

\bibitem
{Bern09}
B. Berndtsson. Curvature of vector bundles associated to holomorphic fibrations. Ann. of Math. {\bf 169} (2009), 531-560.

\bibitem
{BFJ08}
S. Boucksom, C. Favre, and M. Jonsson. Valuations and plurisubharmonic singularities. Publ. RIMS {\bf 44} (2008), 449-494.

\bibitem
{BHJ15}
S. Boucksom, T. Hisamoto and M. Jonsson. Uniform K-stability, Duistermaat-Heckman measures and singularities of pairs, to appear in Ann. Inst. Fourier (Grenoble) arXiv:1504.06568.

\bibitem
{BHJ16}
S. Boucksom, T. Hisamoto and M. Jonsson. Uniform K-stability and asymptotics of energy functionals in K\"{a}hler geometry, arXiv:1603.01026.

\bibitem
{BJ17}
H. Blum and M. Jonsson, Thresholds, valuations and K-stability, arXiv:1706.04548.


\bibitem
{CMM17}
D. Coman, X. Ma and G. Marinescu, Equidistribution for sequences of line bundles on normal K\"{a}hler spaces, Geom. Topol. {\bf 21} (2017) 923-962.

\bibitem
{BM97}
E. Bierstone, P. Milman, Canonical desingularization in characteristic zero by blowing up the maximum strata of a local invariant, Invent. Math. {\bf 128}
(1997), 207-302.


\bibitem
{CDS15}
X.X. Chen, S. K. Donaldson and S. Sun, K\"{a}hler-Einstein metrics on Fano manifolds, I-III, J. Amer. Math. Soc. {\bf 28} (2015), 183-197, 199-234, 235-278.

\bibitem
{Cro80}
C. Croke. Some isoperimetric inequalities and consequences. Ann. Sci. E. N. S., Paris, {\bf 13}: 419-435, 1980.

\bibitem
{Dar14}
T. Darvas. The Mabuchi geometry of finite energy classes. Adv. Math. {\bf 285} (2015), 182-219.

\bibitem
{Dem92}
J.-P. Demailly. Regularization of closed positive currents and Intersection Theory. J. Alg. Geom. {\bf 1} (1992), 361-409.

\bibitem
{Der16}
R.Dervan, Uniform stability of twisted constant scalar curvature K\"{a}hler metrics, Int. Math. Res. Not. 2016, no 15, 4728-4783.

\bibitem
{DT92}
Ding, W.Y. and Tian, G.:  K\"{a}hler-Einstein metrics and the generalized
Futaki invariant. Invent. Math., {\bf 110} (1992), no. 2, 315-335.


\bibitem
{Don02}
S. Donaldson, Scalar curvature and stability of toric varieties, J. Differential Geom. {\bf 62} (2002), no.2, 289-349.

\bibitem
{Don12a}
S.Donaldson. K\"{a}hler metrics with cone singularities along a divisor. Essays in mathematics and its applications. 49-79, Springer, Heidelberg, 2012.

\bibitem
{Don12b}
S.Donaldson. Stability, birational transformations and the K\"{a}hler-Einstein problem. Surveys in Differential Geometry, Vol XVII, International Press 2012.

\bibitem
{DS14}
S. Donaldson, S. Sun, Gromov-Hausdorff limits of K\"{a}hler manifolds and algebraic geometry, Acta Math. {\bf 213} (2014), no. 1, 63-106.

\bibitem
{EGZ09}
P. Eyssidieux, V. Guedj, A. Zeriahi. Singular K\"{a}hler-Einstein metrics. J. Amer. Math. Soc. {\bf 22} (2009), 607-639.

\bibitem
{Fuj16}
K. Fujita, A valuative criterion for uniform K-stability of $\bQ$-Fano varieties, arXiv:1602.00901. To appear in J. Reine Angew. Math.

\bibitem
{Fuj17a}
K. Fujita, Uniform K-stability and plt blowups, arXiv:1701.00203. To appear in Kyoto J. Math.

\bibitem
{Fuj17b}
K. Fujita, Openness results for uniform K-stability, arXiv:1709.08209.

\bibitem
{Fut83}
A. Futaki. An obstruction to the existence of Einstein K$\ddot{a}$hler metrics, Inventiones Mathematicae (1983), 437-443.

\bibitem
{GP16}
H. Guenancia, M. P\v{a}un, Conic singularities metrics with prescribed Ricci curvature: the case of general cone angles along normal crossing divisors. J. Diff. Geom., {\bf 103} (2016), no.1, 15-57.

\bibitem
{GZ07}
V. Guedj, A. Zeriahi. The weighted Monge-Amp\`{e}re energy of quasiplurisubharmonic functions, J. Funct. Anal. {\bf 250} (2007), no. 2, 442-482.

\bibitem
{His19}
Hisamoto, T. Mabuchi's soliton metric and relative D-stability. arXiv:1905.05948.

\bibitem
{JMR16}
T. Jeffres, R. Mazzeo, Y. Rubinstein. K\"{a}hler-Einstein metrics with edge singularities, with an appendix by C. Li and Y. Rubinstein, Ann. of Math. (2) {\bf 183} (2016), no. 1, 95-176.


\bibitem
{JM12}
M. Jonsson, M. Musta\c{t}\u{a}. Valuations and asymptotic invarians for sequences of ideals. Ann. Inst. Fourier {\bf 62} (2012), no.6, 2145-2209.

\bibitem
{Li80}
P. Li, On the Sobolev constant and the $p$-spectrum of a compact Riemannian manifold. Ann. Sci. E.N.S., Paris, {\bf 13}, pp. 451-468, 1980.

\bibitem
{Li12}
C. Li. K\"{a}hler-Einstein metrics and K-stability. Ph.D. Thesis, Princeton, June 2012.

\bibitem
{Li15}
C. Li, K-semistability is equivariant volume minimization, Duke Math. J. {\bf 166}, number 16 (2017), 3147-3218. 

\bibitem
{Li19}
Li, C. G-uniform stability and K\"{a}hler-Einstein metrics on Fano varieties. arXiv:1907.09399.

\bibitem
{LTW19}
Li, C.; Tian, G.; Wang, F. The uniform version of Yau-Tian-Donaldson conjecture singular Fano varieties. arXiv:1903.01215.

\bibitem
{LWX14}
C. Li, X. Wang, C. Xu. On the proper moduli spaces of smoothable K\"{a}hler-Einstein Fano varieties, Duke Math. J., {\bf 168}, no. 8 (2019), 1387-1459. 

\bibitem
{LX12}
C. Li and C. Xu. Special test configurations and K-stability of Fano varieties, Ann. of Math. (2) {\bf 180} (2014), no.1, 197-232.

\bibitem
{LX16}
C. Li, C. Xu, Stability of valuations and Koll\'{a}r components, arXiv:1604.05398.

\bibitem
{NTZ15}
G. La Nave, G. Tian and Z. Zhang, Bounding diameter of singular K\"{a}hler metric, American Journal of Mathematics,
Volume {\bf 139}, no. 6, pp.1693-1731. arXiv:1503.03159.

\bibitem
{Mat57}
Y. Matsushima. Sur la structure du group d'hom\'{e}omorphismes analytiques d'une certaine
variet\'{e} Kaehl\'{e}rinne. Nagoya Math. J., {\bf 11}, 145-150 (1957).

\bibitem
{MR12}
R. Mazzeo, Y. Rubinstein, The Ricci continuity method for the complex Monge-Amp\`{e}re equation, with applications to K\"{a}hler-Einstein edge metrics.
C. R. Math. Acad. Sci. Paris {\bf 350} (2012), no. 13-14, m693-697.

\bibitem
{RZ11}
X. Rong, Y. Zhang, Continuity of extremal transitions and flops for Calabi-Yau manifolds, J. Differential Geometry, {\bf 89} (2011) 233-269.

\bibitem
{She15}
L. Shen, $C^{2,\alpha}$-estimate for conical K\"{a}hler-Ricci flow, arXiv:1412.2420.

\bibitem
{Son14}
J. Song, Riemannian geometry of K\"{a}hler-Einstein currents, arXiv:1404.0445.

\bibitem
{SSY16}
C. Spotti, S. Sun, C-J. Yao. Existence and deformations of K\"{a}hler-Einstein metrics on smoothable $\bQ$-Fano varieties,
Duke Math. J. {\bf 165}, 16 (2016), 3043-3083.

\bibitem
{Tia90}
G. Tian, On Calabi's conjecture for complex surfaces with positive first Chern class. Invent. Math. {\bf 101}, (1990), 101-172.

\bibitem
{Tia94}
G. Tian, K\"ahler-Einstein metrics on algebraic manifolds.
Transcendental methods in algebraic geometry (Cetraro, 1994), 143šC185, Lecture Notes in Math.,
1646, Springer, Berlin, 1996.


\bibitem
{Tia97}
G. Tian, K\"{a}hler-Einstein metrics with positive scalar curvature. Inv. Math. {\bf 130} (1997), 239-265.

\bibitem
{Tia12}
G. Tian, Existence of Einstein metrics on Fano manifolds, in {\it Metric and Differential Geometry}, Progr. Math., 297, pp. 119-159. Birkh\"{a}user/Springer, Basel, 2012.

\bibitem
{Tia15}
G. Tian, K-stability and K\"{a}hler-Einstein metrics. Comm.
Pure Appl. Math. {\bf 68} (7) (2015), 1085-1156. 

\bibitem
{Tia17}
G. Tian, A third derivative estimate for Monge-Ampere equations with conical singularities, Chinese Annals
of Mathematics, Series B, Volume {\bf 38}, Issue 2, pp 687-694.

\bibitem
{TW18}
G. Tian, F. Wang, Cheeger-Colding-Tian theory for conic Kahler-Einstein metrics, arXiv:1807.07209.


\bibitem
{TW19}
G. Tian, F. Wang, On the existence of conic Kaehler-Einstein metrics, arXiv:1903.12547.

\bibitem
{Wan12}
X. Wang, Heights and GIT weight, Math. Res. Lett. {\bf 19} (2012), 909-926.

\end{thebibliography}
\end{document}